\documentclass[reqno, psamsfonts, letterpaper]{amsart}

\usepackage{amssymb}

\usepackage[style=alphabetic, backend=bibtex]{biblatex}
\usepackage{tikz-cd}
\tikzcdset{			
	diagrams={sep=small}
}

\usepackage{mathtools}

\usepackage{hyperref, cleveref}
\hypersetup{
	colorlinks = true,
}

\makeatletter
\def\l@subsection{\@tocline{2}{0pt}{2.5pc}{5pc}{}}
\makeatother

\usepackage{enumitem}
\setlist[1]{labelindent=\parindent}
\setlist[enumerate, 1]{label = \textnormal{(\arabic*)}, ref = \arabic*}
\setlist[enumerate, 2]{label = \textnormal{(\alph*)}, ref = \alph*}
\setlist[enumerate, 3]{label = \textnormal{(\roman*)}, ref = \roman*}

\numberwithin{equation}{section}

\newcommand\ddaaux{\rotatebox[origin=c]{-90}{\scalebox{0.70}{$\dashrightarrow$}}} 
\newcommand\dashdownarrow{\mathrel{\text{\ddaaux}}}

\let\emptyset\varnothing

\let\AA\undefined

\let\to\longrightarrow

\let\mapsto\longmapsto

\usepackage{thmtools}

\declaretheorem[sibling=equation, style=definition]{definition}
\declaretheorem[sibling=equation, style=definition, name=Proposition-Definition]{proposition-definition}

\declaretheorem[sibling=equation]{theorem}
\declaretheorem[sibling=theorem]{proposition}
\declaretheorem[sibling=equation]{lemma}
\declaretheorem[sibling=equation]{corollary}

\declaretheorem[sibling=equation, style=remark]{remark}
\declaretheorem[sibling=equation, style=remark]{example}

\DeclareMathOperator{\GL}{GL}

\DeclareMathOperator{\Sym}{Sym}
\DeclareMathOperator{\Hom}{Hom}

\DeclareMathOperator{\Aut}{Aut}

\DeclareMathOperator{\id}{id}
\DeclareMathOperator{\im}{im}
\DeclareMathOperator{\coker}{coker}
\DeclareMathOperator{\rk}{rk}

\newcommand{\SheafHom}{\underline{\mathrm{Hom}}}

\newcommand{\SheafExt}{\underline{\mathrm{Ext}}}
\newcommand{\SheafAut}{\underline{\mathrm{Aut}}}

\DeclareMathOperator{\Spec}{Spec}

\DeclareMathOperator{\Supp}{Supp}
\DeclareMathOperator{\Pic}{Pic}
\DeclareMathOperator{\NS}{NS}

\DeclareMathOperator{\codim}{codim}

\DeclareMathOperator{\red}{red}

\DeclareMathOperator{\fppf}{fl}

\DeclareMathOperator{\pr}{pr}

\DeclareMathOperator{\Lie}{Lie}

\newcommand{\ZZ}{\mathbb{Z}}
\newcommand{\QQ}{\mathbb{Q}}

\newcommand{\CC}{\mathbb{C}}

\newcommand{\PP}{\mathbb{P}}

\newcommand{\GG}{\mathbb{G}}
\newcommand{\AA}{\mathbb{A}}
\newcommand{\TT}{\mathbb{T}}
\newcommand{\VV}{\mathbb{V}}

\DeclareMathOperator{\Fitt}{Fitt}
\newcommand{\DD}{\mathbb{D}}

\DeclareMathOperator{\Sing}{Sing}

\DeclareMathOperator{\tor}{tor}
\DeclareMathOperator{\Exc}{Exc}
\DeclareMathOperator{\Dom}{Dom}

\DeclareMathOperator{\opp}{op}

\newcommand{\Set}{\mathsf{Set}}
\newcommand{\Group}{\mathsf{Gp}}
\newcommand{\Ab}{\mathsf{Ab}}
\newcommand{\Sch}{\mathsf{Sch}}

\DeclareMathOperator{\St}{St}
\DeclareMathOperator{\elliptic}{ell}
\DeclareMathOperator{\sep}{sep}
\DeclareMathOperator{\ft}{ft}
\newcommand{\Sp}{\mathsf{Sp}}

\usepackage[OT2,T1]{fontenc}
\DeclareSymbolFont{cyrletters}{OT2}{wncyr}{m}{n}
\DeclareMathSymbol{\Sha}{\mathalpha}{cyrletters}{"58}

\begin{document}
\title[The N\'eron model of a Lagrangian fibration]{The N\'eron model of a higher-dimensional Lagrangian fibration}

\author[Y.-J. Kim]{Yoon-Joo Kim}
\address{Department of Mathematics, Columbia University, New York, NY 10027, USA}
\email{yk3029@columbia.edu}


\begin{abstract}
	Let $\pi : X \to B$ be a projective Lagrangian fibration of a smooth symplectic variety $X$ to a smooth variety $B$. Denote the complement of the discriminant locus by $B_0 = B \setminus \operatorname{Disc}(\pi)$, its preimage by $X_0 = \pi^{-1}(B_0)$, and the complement of the critical locus by $X' = X \setminus \operatorname{Sing}(\pi)$. Under an assumption that the morphism $X' \to B$ is surjective, we construct (1) the N\'eron model of the abelian fibration $\pi_0 : X_0 \to B_0$ and (2) the N\'eron model of its automorphism abelian scheme $\operatorname{Aut}^{\circ}_{\pi_0} \to B_0$. Contrary to the case of elliptic fibrations, $X'$ may not be the N\'eron model of $X_0$; this is precisely because of the existence of flops in higher-dimensional symplectic varieties. Using such techniques, we analyze when $X' \to B$ is a torsor under a smooth group scheme and also revisit some known results in the literature.
\end{abstract}

\maketitle
\tableofcontents

\section{Introduction}
	Let $\pi : S \to \Delta$ be a proper minimal elliptic fibration form a complex smooth surface $S$ to a smooth curve $\Delta$. Assume that $\pi$ has no multiple fibers and in particular $S$ has a trivial relative canonical bundle by Kodaira's canonical bundle formula. Consider a smooth, quasi-projective, and surjective morphism
	\[ \pi' : S' \coloneq S \setminus \Sing(\pi) \to \Delta .\]
	In the 1960s, N\'eron \cite{neron:original_article} and Raynaud \cite{raynaud:neron} proved the existence of a unique smooth, commutative, and quasi-projective group scheme $P \to \Delta$ making $S'$ a torsor under its action. In particular, if the elliptic fibration $\pi$ has a section then $S' = P$ has a group scheme structure. Nowadays this is called the \emph{N\'eron model theory} and it is interpreted in two steps: (1) the generic fiber $S_K$ admits a N\'eron model that is isomorphic to $\pi' : S' \to \Delta$, and (2) the (neutral) automorphism group $\Aut^{\circ}_{S_K}$ of the generic fiber admits a N\'eron model $P \to \Delta$. One can think of the torsor property as a formal consequence of these two.
	
	A \emph{smooth symplectic variety} is a smooth variety $X$ equipped with a closed non-degenerate $2$-form $\sigma \in H^0 (X, \Omega_X^2)$, and its \emph{Lagrangian fibration} is a proper surjective morphism $\pi : X \to B$ with connected fibers whose general fiber is a Lagrangian subvariety. Note that $\pi$ is flat if and only if $B$ is smooth, and that the dimension of $B$ is precisely half the dimension of $X$. The goal of this paper is to generalize the above result of N\'eron--Raynaud to a higher-dimensional flat Lagrangian fibration $\pi : X \to B$ from a smooth symplectic variety $X$. Examples include $\GL$-Hitchin fibrations for smooth projective curves or Lagrangian fibrations of projective hyper-K\"ahler manifolds.
	
	The original definition of the N\'eron model makes sense only over $1$-dimensional regular schemes. Therefore, we first need its definition over higher-dimensional bases. This is achieved in \cite[Definition 1.1]{hol19} for abelian schemes and later refined in \cite[Definition 6.1]{hol-mol-ore-poi23} in full generality. Let us quickly recall their definition.
	
	\begin{definition}
		Let $B$ be a variety, $B_0 \subset B$ a dense Zariski open subset, and $f_0 : N_0 \to B_0$ a smooth morphism. A smooth morphism $f : N \to B$ from an algebraic space $N$ with $f^{-1}(B_0) \cong N_0$ is called the \emph{N\'eron model of $f_0$} if:
		\begin{itemize}
			\item For every algebraic space $Z$ smooth over $B$, every morphism $Z_0 \to N_0$ over $B_0$ uniquely extends to a morphism $Z \to N$ over $B$.
		\end{itemize}
	\end{definition}
	
	A technical subtlety is the introduction of algebraic spaces. We justify their necessity in \Cref{sec:backgrounds}. A scheme is called \emph{nowhere reduced} if it is non-reduced at every point. We are now ready to state the main theorem of this article.
	
	\begin{theorem} [Existence of the N\'eron models] \label{main:Neron model}
		Let $\pi : X \to B$ be a projective Lagrangian fibration from a smooth symplectic variety $X$ to a smooth variety $B$, both defined over $\CC$. Let $B_0 \subset B$ be the Zariski open subset over which $\pi$ is smooth and $\pi_0 : X_0 = \pi^{-1}(B_0) \to B_0$ a smooth morphism. Assume that $\pi$ does not have any nowhere reduced fiber. Then
		\begin{enumerate}
			\item There exists a N\'eron model $X^n \to B$ of $\pi_0 : X_0 \to B_0$.
			\item There exists a N\'eron model $P \to B$ of the abelian scheme
			\[ P_0 = \Aut^{\circ}_{\pi_0} \to B_0 .\]
		\end{enumerate}
		As a result, $X^n$ is a $P$-torsor.
	\end{theorem}
	
	The two N\'eron models $X^n$ and $P$ are isomorphic if and only if $\pi$ has a rational section. However, we should consider them as different geometric objects even when they are isomorphic. To have a better feeling for the N\'eron model $P \to B$, it is a smooth and commutative group algebraic space that is occasionally quite large: for example, it may be non-separated or non-finite type (\Cref{main:torsor locus}), which is a phenomenon that never arose for elliptic fibrations. For the two examples mentioned above, recall that a $\GL$-Higgs bundle on a smooth projective curve $C$ is a rank $1$ torsion-free sheaf supported on a spectral curve over $C$. A natural group scheme to consider is the relative Jacobian $J \to B$ of the universal family of spectral curves. The N\'eron model $P$ contains the relative Jacobian $J$ but is typically strictly larger. A Lagrangian fibration of a projective hyper-K\"aher manifold a priori does not come with a group algebraic space, but \Cref{main:Neron model} \emph{constructs} one.
	
	One interpretation of the N\'eron model $P \to B$ is a moduli space of (certain) birational automorphisms of $X$ over $B$. Hence $P$ acts on $X$ only birationally but not regularly in general. However, it always contains the largest open subgroup $P^a$ that acts regularly on $X$, the moduli subspace of (certain) automorphisms of $X$ over $B$. The N\'eron model $P$ satisfies the \emph{$\delta$-regularity} condition of Ng\^o (\S \ref{sec:delta-regularity}) as well: the closed subsets
	\[ D_i = \{ b \in B : \dim (P_b)_{\operatorname{aff}} \ge i \} \ \subset \ B \]
	have codimension $\ge i$ for all $i \ge 0$. Here $(P_b)_{\operatorname{aff}}$ denotes the maximal connected linear algebraic subgroup of $P_b$ in the Chevalley structure theorem.
	
	\begin{theorem} [Structure of the N\'eron model $P$] \label{main:structure of P}
		Keep the notations and assumptions of \Cref{main:Neron model}. Then
		\begin{enumerate}
			\item $P \to B$ is a smooth, commutative, and $\delta$-regular group algebraic space.
			
			\item $P$ contains a maximal open subgroup space $P^a$ that acts regularly on $X$ over $B$. Such $P^a \to B$ is a smooth, commutative, and quasi-projective group scheme, and the $P^a$-action makes $\pi : X \to B$ a $\delta$-regular abelian fibration.
			
			\item For each \'etale morphism $U \to B$, we have
			\begin{align*}
				P(U) = \{ f : X_U \dashrightarrow X_U : & \, \mbox{birational automorphism over } U \mbox{ that acts on} \\
					& \mbox{the generic fiber as a translation automorphism} \} .
			\end{align*}
			Same holds for the group $P^a(U)$ but with automorphisms $f : X_U \to X_U$.
			
			\item The open immersion $P^a \subset P$ is an equality over
			\[ B_1 = \{ b \in B : \dim (P_b)_{\operatorname{aff}} \le 1 \} \quad (= B \setminus D_2) .\]
		\end{enumerate}
	\end{theorem}
	
	The $\delta$-regularity of $P$ and the first $\delta$-locus $B_1$ in the theorem will play a crucial role in our study of Lagrangian fibrations. Note that $B_1 \subset B$ is a Zariski open subset with codimension $\ge 2$ complement ($\delta$-regularity), so it captures all \emph{codimension $1$ behaviors} of the Lagrangian fibration.
	
	\begin{remark} \label{main:rmk:strict neutral component}
		Let $P^{\circ\circ} \subset P$ be the smallest open subgroup space of $P$ (\S \ref{sec:strict neutral component}). The group scheme $P^{\circ\circ}$ has already been constructed several times in literature. In \cite[\S 2]{mar96} and later in \cite[\S 1]{aba-rogov21} (their notations are $A \to B_{nm}$ and $\Aut^0_{\pi} \to B$), $P^{\circ\circ}$ is analytically constructed as the image of the \emph{exponential map} $\exp : \Omega_B \to \Aut_{\pi}$. In \cite[\S 8]{ari-fed16}, under a stronger assumption that every fiber of $\pi$ is integral, $P^{\circ\circ}$ is algebraically constructed, again as a locally closed subgroup scheme of $\Aut_{\pi}$. Our approach is closer to Arinkin--Fedorov's and is purely algebraic, except for one part about the cohomology vanishing (\Cref{prop:vanishing of higher direct image sheaf}) relying on the Hodge module theory in \cite{sch23}. Nevertheless, we emphasize that one of our main claims is that the N\'eron model $P$ contains more information than $P^{\circ\circ}$ and should be of independent interest.
	\end{remark}
	
	To study the N\'eron model $X^n$ in \Cref{main:Neron model}, notice first that the nowhere reduced fiber assumption in the theorem is equivalent to the surjectivity of the morphism
	\begin{equation} \label{main:eq:X'}
		\pi' : X' \coloneq X \setminus \Sing(\pi) \to B .
	\end{equation}
	Our original motivation was to check whether this smooth locus $X'$ is precisely the N\'eron model $X^n$ as in the discussion for elliptic fibrations.\footnote{Some positive results in literature to this direction are \cite[Theorem 1.1]{liu-tong16} for minimal, flat, and proper families of higher-genus curves over a Dedekind scheme, and \cite[Theorem 1.1]{bog-halle-paz-tan18} for abelian surface-fibered Calabi--Yau threefolds over $\PP^1$. Both results are over Dedekind schemes.} This may no longer be the case for higher-dimensional Lagrangian fibrations: we show instead that $X^n$ is the union of smooth loci $Y'$ for every birational model $\tau : Y \to B$ of $\pi$. In particular, $X^n$ will fail to be separated along the flopping locus when we glue the non-isomorphic $X'$ and $Y'$ together (if such a situation arises).
	
	\begin{theorem} [Structure of the N\'eron model $X^n$] \label{main:structure of Xn}
		Keep the notations and assumptions of \Cref{main:structure of P} and \eqref{main:eq:X'}. Then
		\begin{enumerate}
			\item There exists a $P^a$-equivariant open immersion $X' \hookrightarrow X^n$.
			
			\item There exists an equality
			\[ X^n = \bigcup_{Y_U} Y'_U ,\]
			where $U \to B$ runs through every quasi-projective \'etale morphism, $\tau : Y_U \to U$ a Lagrangian fibered smooth symplectic variety that is birational to $X_U \to U$, and $Y'_U = Y_U \setminus \Sing(\tau)$ the smooth locus of $\tau$.
			
			\item The open immersion $X' \subset X^n$ is an equality over $B_1$.
		\end{enumerate}
		As a result, $X'$ is a $P$-torsor over $B_1$.
	\end{theorem}
	
	The last items of \Cref{main:structure of P} and \ref{main:structure of Xn} show that Lagrangian fibrations behave relatively gentle over $B_1$, explaining why the nicest results of $\pi$ were obtained over codimension $1$ points on the base (e.g., Kodaira and N\'eron's results on elliptic fibrations or \cite{hwang-ogu09}). The behavior of $\pi$ over the deeper $\delta$-strata can be wilder but yet its smooth locus $X' \subset X$ is dominated by the behavior of the abelian fibration $\pi_0 : X_0 \to B_0$ by the N\'eron mapping property of $X^n$.
	
	\bigskip
	
	Let us explain the failure of quasi-compactness and separatedness of the N\'eron models over $B$ in \Cref{main:Neron model}. Because $X^n$ is a $P$-torsor, $P$ is of finite type (resp. separated) if and only if $X^n$ is. From \Cref{main:structure of Xn}, the failure of quasi-compactness or separatedness of $X^n$ is intimately related to how different $X^n$ is from the smooth locus $X'$ of the Lagrangian fibration. One way of capturing this difference is to define the following subsets of $B$ over which $X^n$ is nice:
	\begin{itemize}
		\item $B_{\tor}$ is the subset of $B$ over which $X' \subset X^n$ is an equality.
		\item $B_{\ft}$ (resp. $B_{\sep}$) is a maximal Zariski open subset of $B$ over which $X^n$ is of finite type (resp. separated).
	\end{itemize}
	Then $B_{\tor} = B$ if and only if $X'$ is a $P$-torsor. The strict inclusion $B_{\ft} \subsetneq B$ (resp. $B_{\sep} \subsetneq B$) will mean that $X^n$ is not of finite type (resp. not separated) over $B$.
	
	\begin{proposition} [Failure of separatedness of $X^n$ and the torsor property of $X'$] \label{main:torsor locus}
		Keep the notations and assumptions as above. Then
		\begin{enumerate}
			\item $B_{\tor}$ is a Zariski open subset of $B$ and $B_1 \subset B_{\tor} \subset B_{\ft} \subset B_{\sep} \subset B$.
			\item Assume $\dim X = 4$. Then $B_{\tor} = B_{\ft} = B_{\sep}$.
			\item The inclusions $B_1 \subset B_{\tor}$ and $B_{\sep} \subset B$ may be strict.
		\end{enumerate}
		In other words, $X^n$ may fail to be of finite type or separated, and $X'$ may fail to be a $P$-torsor over $B$. When $\dim X = 4$, these happen simultaneously.
	\end{proposition}
	
	We will later interpret the result \cite[Theorem 2]{ari-fed16} as $B_{\elliptic} \subset B_{\tor}$ where $B_{\elliptic} \subset B$ is the open subset over which $\pi$ has geometrically integral fibers. This complements \Cref{main:torsor locus}(1). We believe the second item will hold in arbitrary dimension. The third item will be presented as three different examples in \S \ref{sec:main examples}, which will be instructive to understand the behaviors of $P$ and $X^n$. Though we have proved that $X'$ may fail to be a $P$-torsor, one may still wonder whether it is possible to find another group algebraic space $G \to B$ making $X'$ a $G$-torsor. Unfortunately, this weaker statement may not hold as well (\Cref{prop:X' is not G-torsor}).
	
	\bigskip
	
	Finally, there is a conjectural \emph{dual Lagrangian fibration} $\check \pi : \check X \to B$ predicted by a holomorphic variant of the Strominger--Yau--Zaslow conjecture. We note that this prediction is verified for $G$-Hitchin fibrations for connected reductive groups $G$ in \cite[Theorem C]{don-pan12} and arbitrary Lagrangian fibrations of projective hyper-K\"ahler manifolds of known deformation types in \cite[Theorem 1.1]{kim25}. If one believes such a conjecture and that it satisfies similar properties to \Cref{main:Neron model}--\ref{main:torsor locus}, then the dual abelian scheme $\check P_0 \to B_0$ is expected to admit a N\'eron model over $B$. This is indeed the case, providing evidence for the existence of a dual Lagrangian fibration $\check \pi$ for arbitrary $\pi$.
	
	\begin{theorem} [Dual N\'eron model]
		Keep the notations and assumptions as above. Let
		\[ \check P_0 = \Pic^{\circ}_{\pi_0} \to B_0 \]
		be the dual abelian scheme of $P_0 \to B_0$. Then
		\begin{enumerate}
			\item There exists a N\'eron model $\check P \to B$ of $\check P_0 \to B_0$.
			\item $\check P$ is $\delta$-regular and shares the same $\delta$-loci with $P$.
		\end{enumerate}
	\end{theorem}
	
	The dual Lagrangian fibration $\check \pi : \check X \to B$, if it exists, should be a compactification of a $\check P_1$-torsor over $B_1$. A recent result of Sacc\`a \cite[Theorem 1]{sacca24} studies this problem in depth (see \cite[Theorem 1.1]{liu-liu-xu24} for an alternative version). The total space $\check X$ of such a compactification turns out to be a $\QQ$-factorial terminal symplelctic variety. This agrees with the result of \cite[Theorem 1.1]{kim25} where the expected dual fibration $\check \pi$ is explicitly constructed when $X$ is a generalized Kummer variety or an OG6-type hyper-K\"ahler manifold. We plan to revisit this problem in future work.

	\subsection{Relation to results in literature and applications}
		This paper grew from an attempt to understand the holomorphic Strominger--Yau--Zaslow conjecture and to understand to what extent a Lagrangian fibration can be considered as a compactification of a smooth group scheme torsor. We arrive to a conclusion that the N\'eron model theory is extremely helpful to study these questions. Such a viewpoint can also be beneficial to connect several other known results in literature, which we summarize here. See \S \ref{sec:applications} for detailed discussions.
		
		\begin{itemize}
			\item (\S \ref{sec:Arinkin-Fedorov}) When $\pi$ has integral fibers, \cite[Theorem 2]{ari-fed16} constructed a smooth group scheme $G \to B$ making $X'$ a $G$-torsor. We recover their result by showing $G = P^{\circ\circ} = P^a = P$. An alternative interpretation of their result is that $B_{\tor}$ contains every point $b \in B$ such that $X_b$ is geometrically integral.
			
			\item (\S \ref{sec:rational section}) Every rational section of $\pi$ is necessarily defined over $B_{\tor}$ but in general not extendable over the entire $B$. If $X'$ is a $P$-torsor, then $B_{\tor} = B$ so every rational section is regular.
			
			\item (\S \ref{sec:Hwang-Oguiso}) \cite{hwang-ogu09} studied the structure of general singular fibers of higher-dimensional Lagrangian fibrations, generalizing the work of Kodaira for elliptic fibrations. Many results in their \S 2--3 can be recovered algebraically using the $P^a$-action in \Cref{main:structure of P}.
			
			\item (\S \ref{sec:Abasheva-Rogov}) When $X$ is compact, \cite{aba-rogov21} studied the Tate--Shafarevich group of a Lagrangian fibration, which is $H^1_{an} (B, P^{\circ\circ})$ in our notation. The N\'eron mapping property of $P$ helps studying this group further.
			
			\item (\S \ref{sec:de Cataldo-Rapagnetta-Sacca}) \cite{decat-rap-sacca21} compared singular fibers of two explicit Lagrangian fibrations of projective hyper-K\"ahler manifolds $\pi : X \to B$ and $\tau : Y \to B$ where $X$ is of OG10-type and $Y$ is of $\text{K3}^{[5]}$-type. We show both $X'$ and $Y'$ are $P$-torsors (ignoring the nowhere reduced fibers), explaining some of the behaviors they discovered.
		\end{itemize}

	\subsection{Sketch proofs and organization of the paper}
		Though we have taken a bird's-eye viewpoint in the introduction to emphasize the role of the N\'eron models $P$ and $X^n$, our proof will go in the opposite direction. We will first build various coarser invariants and patch them together to construct the N\'eron models as a final outcome.
		
		The first two sections \S \ref{sec:group spaces}--\ref{sec:Neron models} are devoted to more abstract discussions about the theory of group algebraic spaces and N\'eron models. We decided to write these sections separately and with greater generality due to the following two reasons. First, the proof of the main results (\Cref{thm:main component 1} and \ref{prop:main component 2} in \S \ref{sec:group spaces}, and \Cref{cor:Neron under quasi-projective etale homomorphism}, \ref{thm:weak Neron model of group space}, and \ref{thm:weak Neron model of torsor} in \S \ref{sec:Neron models}) are rather lengthy and require abstract generalities. Second, we bear in mind of future generalizations of the results in this article into more general set-ups, such as for $\QQ$-factorial terminal symplectic algebraic spaces $X$ or proper Lagrangian fibrations $\pi$.
		
		We start the discussions about Lagrangian fibrations from \S \ref{sec:translation automorphism scheme}. The main goal of this section is to construct the group scheme $P^a$. We realize $P^a$ as a locally closed subgroup scheme of $\Aut_{\pi}$ as in \cite{ari-fed16} and \cite{aba-rogov21}, but use different techniques to achieve this because our group scheme may be larger than theirs. More precisely, we use \Cref{thm:main component 1} and \ref{prop:main component 2} to construct $P^a$ as the equidimensional locus of the main component of $\Aut_{\pi} \to B$ (\S \ref{sec:main component}). Infinitesimal version of this claim should be proved beforehand: we first need to show that $\Lie (\Aut_{\pi})$ contains the cotangent bundle of $B$ as a Lie subalgebra scheme. This is a formal consequence of the vanishing of a certain higher direct image coherent sheaf in \Cref{prop:vanishing of higher direct image sheaf}, which is proved from the Hodge module symmetry of the decomposition theorem proposed in \cite{shen-yin22} and proved in \cite{sch23}. We next prove the $\delta$-regularity of the group scheme $P^a$. The idea is to observe that $X$ equipped with the $P^a$-action is a weak abelian fibration. This strategy is already contained in Ng\^o's original articles \cite{ngo10, ngo11} and also revisited several times in literature, so we make the discussion short with some proper citations.
		
		In \S \ref{sec:Neron model of fibration codimension 1}, we study the behavior of $\pi$ over the first nontrivial $\delta$-locus $B_1$. The main theorem shows $P^a_1$ and $X'_1$ are N\'eron under $B_0 \subset B_1$. We start with an observation in \Cref{lem:equivalence between Neron and birational automorphism} (using \Cref{thm:weak Neron model of group space} and \ref{thm:weak Neron model of torsor}) that the N\'eron mapping property of $X'$ is essentially equivalent to understanding the birational behavior of the Lagrangian fibration $\pi$. This turns our interest to the study of birational maps $f : X \dashrightarrow Y$ between Lagrangian fibered symplectic varieties. \Cref{thm:birational map of fibration codimension 1} is a technical input to this direction, which says that such a birational map $f$ is necessarily an isomorphism over $B_1$. To prove this, we use Kawamata's theorem \cite{kaw08} to decompose $f$ into a finite sequence of flops and analyze each flop separately using the $\delta$-regularity of the action. This sometimes explains the behavior of $f$ over the entire base $B$ and we will need this in the following section. Along the way, we partially recover the results of \cite[\S 2--3]{hwang-ogu09} using the $P^a$-action on $X$ instead of the original foliation theory.
		
		In \S \ref{sec:Neron model of fibration}, we construct the N\'eron models $P$ and $X^n$ over the entire base $B$. The N\'eron model $P$ is obtained by gluing several copies of $P^a$ via translations. The gluing datum is extracted from the neutral Picard algebraic space $\Pic^{\circ}_{\pi}$, leading us to study the dual side and construct $P$ and its dual $\check P$ at the same time. More specifically, we observe the two groups $P^a$ and $\Pic^{\circ}_{\pi}$ have complementary N\'eron behaviors: the former is N\'eron under $B_0 \subset B_1$ and the latter is N\'eron under $B_1 \subset B$. A polarization $P^a \to \Pic^{\circ}_{\pi}$ connects their behaviors and yields the N\'eron models of both sides (the technical input is \Cref{cor:Neron under quasi-projective etale homomorphism}). Once this is done, $X^n$ can be constructed as a $P$-torsor with a twisting datum inherited from that of $X'$. The remaining claims are proved after this. Finally, we construct three examples illustrating (the failure of) the torsor property of $X'$ in \S \ref{sec:main examples}, and present applications of our constructions to results in literature in \S \ref{sec:applications}.

	\subsection*{Acknowledgment}
		I would like to thank Daniel Huybrechts for his numerous comments and invaluable discussions on this project. I would also like to thank Olivier Benoist, Andres Fernandez Herrero, James Hotchkiss, Johan de Jong, Radu Laza, Lisa Marquand, Gebhard Martin, Mirko Mauri, Giacomo Mezzedimi, Antonio Rapagnetta, Giulia Sacc\`a, and Claire Voisin for helpful discussions. This work was partially supported by the ERC Synergy Grant HyperK (ID 854361).

\section{Group algebraic spaces} \label{sec:group spaces}
	The first two sections of this article will be devoted to the general theory of group algebraic spaces (\S \ref{sec:group spaces}) and N\'eron models (\S \ref{sec:Neron models}). The results in these two sections are rather collective, so the reader is invited to jump directly to \Cref{sec:translation automorphism scheme} and come back to them whenever needed. However, we emphasize that some of the technical hearts of the article are contained in these first two sections. The reference \cite{stacks-project} will be cited often, so we simply write their Tag numbers to cite them. For example, (Tag~ABCD) means \cite[Tag~ABCD]{stacks-project}.
	
	Let $S$ be a locally separated and locally Noetherian algebraic space over a field of characteristic $0$.
	
	\begin{definition}
		A \emph{group algebraic space}, or simply a \emph{group space}, over $S$ is an algebraic space $G \to S$ equipped with an identity $e : S \to G$, multiplication $m : G \times_S G \to G$, and inverse $i : G \to G$ satisfying the axioms of groups.
		
		\emph{Every group space in this article will be assumed to be locally separated and locally of finite type over a locally separated and locally Noetherian base algebraic space $S$ of characteristic $0$.}
	\end{definition}
	
	Schemes are always locally separated. Every locally separated and locally of finite type algebraic space is quasi-separated (Tag~088J and 0BB6), so every group space $G \to S$ in this article will be quasi-separated as well. A \emph{(locally) algebraic group} over a field $k$ is a group space (locally) of finite type over $S = \Spec k$. A locally algebraic group is necessarily a scheme \cite[Lemma 4.2]{artin69}, separated (Tag~047L), and smooth (Tag~047N: Cartier's theorem); everything is as nice as possible over a field. We highlight Cartier's theorem as it is the main reason why group spaces behave nicer in characteristic $0$.
	
	\begin{theorem} [Cartier's theorem]
		Every locally algebraic group over a field of characteristic $0$ is smooth.
	\end{theorem}
	
	One should keep in mind that a group space over a general base $S$ may easily fail to be smooth or separated, even when $S$ is regular of characteristic $0$. See \S \ref{sec:examples of group spaces} for many examples. A group space is separated if and only if its identity morphism $e : S \to G$ is a closed immerion (Tag~06P6). Over a reduced base $S$, basic smoothness criterion is provided by the following proposition.
	
	\begin{proposition} \label{prop:smoothness criterion of group space 1}
		Let $f : G \to S$ be a group space over a reduced space $S$. Then the following are equivalent.
		\begin{enumerate}
			\item $f$ is smooth.
			\item $f$ is flat.
			\item $f$ is universally equidimensional (in the sense of \cite[\S 2.1]{sus-voe00}, see \S \ref{sec:equidimensional morphism}).
		\end{enumerate}
	\end{proposition}
	\begin{proof}
		The implications (1) $\Longrightarrow$ (2) $\Longrightarrow$ (3) are clear. (2) $\Longrightarrow$ (1) follows from Cartier's theorem that every fiber is smooth. For (3) $\Longrightarrow$ (2), notice that $f$ is universally open by \Cref{prop:universally equidimensional morphism 1}, the fibers of $f$ are (geometrically) reduced by Cartier's theorem, and $S$ is reduced by assumption. These are enough to prove the flatness of $f$ by \Cref{thm:universal open implies flat}. We note that \cite[{$\mathrm{VI_B}$}]{SGA3I} Corollary 4.4, Proposition 1.6, and Theorem 3.10 proves this statement under an assumption that every fiber of $f$ is connected.
	\end{proof}
	
	We also recall the following result, stating that a group space in certain cases is a scheme.
	
	\begin{theorem} [{\cite[Theorem I.1.9]{fal-chai:abelian}, \cite[Theorem 3.3.1]{ray70}}] \label{thm:group space is scheme} \ 
		\begin{enumerate}
			\item Every smooth proper group space $A \to S$ with connected fibers over a scheme $S$ is a scheme, called an \emph{abelian scheme}.
			\item Every separated group space $G \to \Delta$ over a Dedekind scheme $\Delta$ is a scheme.
		\end{enumerate}
	\end{theorem}

	\subsection{Lie algebra spaces} \label{sec:Lie algebra scheme}
		Since we are assuming $G \to S$ is locally separated, the identity section $e : S \to G$ is an immersion. We denote its conormal sheaf by $C_e$, or $C_G$ for simplicity; it is a coherent sheaf on $S$. The conormal sheaf of a group space plays an equally important role to the relative cotangent sheaf by the following:
		
		\begin{proposition} [Tag~047I] \label{prop:cotangent sheaf of group scheme}
			Let $f : G \to S$ be a group space with an identity section $e : S \to G$. Then the cotangent sheaf of $f$ and the conormal sheaf of $e$ are related via canonical isomorphisms
			\[ \Omega_f = f^* C_e, \qquad C_e = e^* \Omega_f .\]
		\end{proposition}
		
		Every group space $G \to S$ has an associated Lie algebra sheaf $\underline {\Lie}\, G$, which by definition is an fppf $\mathcal O_S$-module (see \cite[II Definition 3.9.0]{SGA3I} or equivalently the paragraph above \cite[Proposition 8.2]{ari-fed16})
		\[ \underline \Lie \, G = \Hom^{\fppf} (C_e, \mathcal O_S) \ : \ \big[ h : T \to S \big] \mapsto \Hom_{\mathcal O_T} (h^* C_e, \mathcal O_T) ,\]
		equipped with a Lie algebra structure. It is the fppf dual of the conormal sheaf $C_e$, so by \S\ref{sec:linear scheme} it is represented by a linear space
		\[ \Lie G = \Spec_S (\Sym^* C_e) \to S ,\]
		which we call the \emph{associated Lie algebra space} of $G$. The Lie bracket operation $[\, , \, ] : \Lie G \times_S \Lie G \to \Lie G$ translates to a Lie cobracket operation
		\[ d : C_e \to \bigwedge^2 C_e ,\]
		a homomorphism of coherent sheaves with $d^2 = 0$ (cf., \cite{mic80}). Both the conormal sheaf and Lie algebra space perspectives are beneficial.

		Recall a series of implications: a scheme is normal $\Longrightarrow$ it is reduced and geometrically unibranch (Tag~0BQ3) $\Longrightarrow$ its connected components are integral. Given a group space $G \to S$, taking a base change by the normalization $S' \to S$ produces a new group space $G' = G \times_S S' \to S'$ whose base is normal.
		
		\begin{proposition} \label{prop:smoothness criterion of group space 2}
			Let $f : G \to S$ be a group space over a reduced space $S$. Then the following are equivalent.
			\begin{enumerate}[start=4]
				\item $f$ is smooth along the identity section $e$.
				\item $f$ has constant fiber dimension.
				\item The conormal sheaf $C_e$ of the identity section is locally free.
				\item $\Lie f : \Lie G \to S$ is smooth.
			\end{enumerate}
			Assume furthermore that $S$ is geometrically unibranch and every irreducible component of $G$ dominates an irreducible component of $S$. Then these conditions are further equivalent to the conditions in \Cref{prop:smoothness criterion of group space 1}.
		\end{proposition}
		\begin{proof}
			The implications (1) $\Longrightarrow$ (4) $\Longleftrightarrow$ (5) $\Longleftrightarrow$ (6) $\Longleftrightarrow$ (7) are combinations of \cite[$\mathrm{VI_B}$ Proposition 1.6]{SGA3I} and \cite[Proposition 8.2]{ari-fed16}. Let us now assume $S$ and $G$ satisfy the given assumptions, and assume (5). Then $f$ is equidimensional and furthermore, universally equidimensional by \Cref{prop:universally equidimensional morphism 2}. This proves (3).
		\end{proof}

		\subsubsection{Deformation to the Lie algebra}
			Let $G \to S$ be a separated group space. The identity section $e : S \to G$ is a closed immersion, so we may consider its deformation to the normal cone construction, which we recall for reader's convenience (see \cite[\S 5]{ful:intersection} for details). Start with the trivial family $G \times \PP^1 \to \PP^1$, blow it up along its closed subscheme $e(S) \times \{ \infty \}$ to obtain $\operatorname{Bl}_{e(S) \times \infty} (G \times \PP^1) \to \PP^1$, and take a complement in the central fiber
			\[ \mathfrak G = \Big( \operatorname{Bl}_{e(S) \times \infty} (G \times \PP^1) \Big) \setminus \operatorname{Bl}_{e(S)} G \to \PP^1 .\]
			It is a flat algebraic space over $\PP^1$ and admits a morphism
			\begin{equation} \label{eq:deformation to Lie algebra}
				\mathfrak G \to S \times \PP^1 \quad \mbox{with a section } \ e : S \times \PP^1 \to \mathfrak G ,
			\end{equation}
			which can be considered as a deformation of the group space $G \to S$ with its section. It satisfies the following properties:
			\begin{itemize}
				\item Over $S \times \AA^1$ for $\AA^1 = \PP^1 \setminus \{ \infty \}$, \eqref{eq:deformation to Lie algebra} is isomorphic to the trivial $\AA^1$-family of $G \to S$ together with the identity section $e : S \to G$; and
				\item Over $S \times \{\infty\}$, it is isomorphic to the normal cone
				\[ N = \Spec_S \Big( \bigoplus_{k \ge 0} I^k / I^{k+1} \Big) \to S ,\]
				of $e(S) \subset G$, where $I$ is the ideal sheaf of the closed subscheme $e(S) \subset G$.
			\end{itemize}
			
			\begin{lemma}
				Let $G \to S$ be a separated group space and \eqref{eq:deformation to Lie algebra} the deformation to the normal cone of its identity section. Then there exists a closed immersion $\mathfrak G_{\infty} \subset \Lie G$ that is set-theoretically equal.
			\end{lemma}
			\begin{proof}
				Let $I$ be the ideal sheaf of $e(S) \subset G$. The conormal sheaf $C_e$ is by definition $I/I^2$, so there exists a surjective homomorphism of quasi-coherent $\mathcal O_S$-algebras
				\[ \bigoplus_k \Sym^k C_e \twoheadrightarrow \bigoplus_k I^k/I^{k+1} \]
				which in turn induces a closed immersion $\mathfrak G_{\infty} \subset \Lie G$ over $S$. Over a point $s \in S$, the fiber of the group space is a locally algebraic group $G_s$ over $k(s)$, which is always smooth by Cartier's theorem. Hence the identity section $\Spec k(s) \to G_s$ is a regular point, so its ideal sheaf $I_s = I \otimes_{\mathcal O_S} k(s)$ satisfies $\Sym^k (I_s/I_s^2) \cong I_s^k/I_s^{k+1}$. This proves the closed immersion $\mathfrak G_{\infty} \subset \Lie G$ is an equality over $s \in S$. Hence $\mathfrak G_{\infty} \subset \Lie G$ is a set-theoretic equality, i.e., $\mathfrak G_{\infty}$ is defined by a nilpotent ideal.
			\end{proof}
			
			Because of this lemma, we will (somewhat incorrectly) call the deformation to the normal cone construction above a \emph{deformation to the Lie algebra}.
			
			\begin{definition}
				For a separated group space $G \to S$, we call the morphism \eqref{eq:deformation to Lie algebra} its \emph{deformation to its Lie algebra}.
			\end{definition}
			
			\begin{lemma} \label{lem:deformation to Lie algebra respects base change}
				Let $G \to S$ be a separated group space.
				\begin{enumerate}
					\item The deformation to the Lie algebra morphism respects arbitrary base change $T \to S$.
					
					\item If $H \subset G$ is a locally closed subgroup space, then there exists an immersion between deformation to the Lie algebras
					\[\begin{tikzcd}[column sep=tiny]
						\mathfrak H \arrow[r, phantom, "\subset"] \arrow[dr] & \mathfrak G \arrow[d] \\
						& S \times \PP^1
					\end{tikzcd}.\]
				\end{enumerate}
			\end{lemma}
			\begin{proof}
				Immediate from construction.
			\end{proof}
			
			\begin{lemma} \label{lem:deformation to Lie algebra smooth}
				Let $G \to S$ be a separated group space and $\mathfrak G \to S \times \PP^1$ its deformation to the Lie algebra in \eqref{eq:deformation to Lie algebra}.
				\begin{enumerate}
					\item If $\Lie G \to S$ is smooth, then $\mathfrak G_{\infty} = \Lie G$ scheme-theoretically.
					\item If $G \to S$ is smooth, then $\mathfrak G \to S \times \PP^1$ is smooth.
				\end{enumerate}
			\end{lemma}
			\begin{proof}
				(1) If $\Lie G \to S$ is smooth then the identity section $e : S \to G$ is a regular immersion by \Cref{prop:smoothness criterion of group space 2}. Hence we have an isomorphism $\Sym^k (I/I^2) = I^k/I^{k+1}$ for the ideal sheaf $I$ of $e$, proving that the Lie algebra space is the normal cone. (2) The algebraic space $\mathfrak G \to \PP^1$ is flat by construction. By (Tag~05X1: flatness criterion by fibers), $g : \mathfrak G \to S \times \PP^1$ is smooth if and only if $g_t : \mathfrak G_t \to S_t$ is smooth for all $t \in \PP^1$. If $G \to S$ is smooth, then $\mathfrak G_{\infty} = \Lie G \to S$ is smooth as well by the first item. 
			\end{proof}

	\subsection{Neutral components}
		Given a locally algebraic space $G$ over $k$, the unique connected component $G^{\circ}$ containing the identity defines an open and closed subgroup scheme of $G$, called the neutral component. A group space $G \to S$ over a general base may have two different notion of ``neutral components'', which we describe in this section.

		\subsubsection{Group spaces with connected fibers and strict neutral component} \label{sec:strict neutral component}
			In general, the smoothness of $\Lie G \to S$ does not imply the smoothness of $G \to S$ (see \Cref{prop:smoothness criterion of group space 2} and \S \ref{sec:examples of group spaces}). However, this implies the fiberwise neutral component of $G$, the \emph{strict neutral component}, is smooth.
			
			\begin{theorem} [{\cite[$\mathrm{VI_B}$ Theorem 3.10]{SGA3I}}] \label{thm:strict neutral component}
				Let $G \to S$ be a group space with smooth Lie algebra space $\Lie G \to S$. Then there exists a unique smooth, open, and normal subgroup space $G^{\circ\circ} \subset G$ such that for every $s \in S$, the fiber $G^{\circ\circ}_s$ is the neutral component of the locally algebraic group $G_s$.
			\end{theorem}
			
			\begin{definition} [Strict neutral component]
				Let $G \to S$ be a group space with smooth Lie algebra. We call the smooth open subgroup space $G^{\circ\circ} \to S$ in \Cref{thm:strict neutral component} the \emph{strict neutral component} of $G$.
			\end{definition}
			
			Recall that an arbitrary group space may very well be non-separated. However, for smooth group spaces, only the disconnectedness of fibers is responsible for its non-separatedness.
			
			\begin{theorem} [{\cite[$\mathrm{VI_B}$ Corollary 5.5]{SGA3I}}] \label{thm:smooth group space with connected fibers}
				Every smooth group space $G \to S$ with connected fibers is separated and of finite type.
			\end{theorem}
			
			\begin{lemma}
				Let $f : G \to H$ be a homomorphism between smooth group spaces over $S$. Then the restriction of $f$ induces $G^{\circ\circ} \to H^{\circ\circ}$.
			\end{lemma}
			\begin{proof}
				For each point $s \in S$, $f_s : G^{\circ\circ}_s \to H_s$ factors through $H^{\circ\circ}_s$. Since $H^{\circ\circ} \subset H$ is an open subspace, this will suffice to prove the claim.
			\end{proof}
			
			The following theorem will be very useful later on.
			
			\begin{theorem}[{\cite[Corollary IX.1.4]{ray:group_schemes}}] \label{thm:Raynaud extension of homomorphism}
				Let $G \to S$ and $H \to S$ be group spaces. Assume that $S$ is $\mathrm{(S_2)}$, $S_1 \subset S$ is a dense open subspace with codimension $\ge 2$ complement, and $G \to S$ is smooth with connected fibers. Then every homomorphism $f_1 : G_1 \to H_1$ over $S_1$ uniquely extends to a homomorphism $f : G \to H$ over $S$.
			\end{theorem}

		\subsubsection{Connected group spaces and (weak) neutral component} \label{sec:neutral component}
			We assume $S$ is connected for simplicity. When $S$ is disconnected, repeat the definition for all connected components of $S$.
			
			\begin{definition} [Neutral component]
				Let $G \to S$ be a group space over a connected space $S$. We call the unique connected component $G^{\circ} \subset G$ containing the identity section the \emph{(weak) neutral component} of $G$.
			\end{definition}
			
			When $G$ is a locally algebraic group over $S = \Spec k$, the notion of neutral and strict neutral components make no difference. In general the two are different, and we may have a strict open immersion $G^{\circ\circ} \subsetneq G^{\circ}$. See \S \ref{sec:examples of group spaces}.
			
			\begin{proposition} \label{prop:neutral component}
				Let $G \to S$ be a group space. Assume either
				\begin{enumerate}[label=\textnormal{(\roman*)}]
					\item $G \to S$ is smooth; or
					\item $S$ is reduced and geometrically unibranch, and $\Lie G \to S$ is smooth.
				\end{enumerate}
				Then the neutral component $G^{\circ} \to S$ is a smooth, open and closed, and normal subgroup space of $G$.
			\end{proposition}
			\begin{proof}
				(i) Assume $S$ is connected. The inverse morphism $i : G \to G$ certainly restricts to a morphism $i : G^{\circ} \to G^{\circ}$ between connected components. To show $m : G \times_S G \to G$ restricts to $G^{\circ} \times_S G^{\circ} \to G^{\circ}$, it is enough to show $G^{\circ} \times_S G^{\circ}$ is a connected algebraic space. The strict neutral component $G^{\circ\circ} \to S$ has geometrically connected fibers (every connected algebraic group is geometrically connected), so $G^{\circ\circ} \times_S G^{\circ}$ is connected (Tag~0385). Since $G^{\circ}$ is smooth, it is the union of translations of $G^{\circ\circ}$ by \'etale local sections (\Cref{prop:etale local section}), so $G^{\circ} \times_S G^{\circ}$ is connected as well. This shows $G^{\circ}$ is an open subgroup space of $G$.
				
				To prove it is a normal subgroup space, we prove it is a characteristic subgroup. Given an \'etale morphism $U \to S$, let $f : G_U \to G_U$ be an automorphism of a group space $G_U \to U$. Then $f$ sends $G^{\circ}_U$ to itself. That is, $G^{\circ} \subset G$ is a characteristic subgroup and hence a normal subgroup.
				
				\bigskip
				
				(ii) The same argument above will work once we prove $G^{\circ} \to S$ is smooth. Let us first reduce the problem to the case when $G \to S$ is locally quasi-finite. By \Cref{thm:strict neutral component}, $G^{\circ\circ}$ is a smooth subgroup space of $G$. The above argument still constructs a morphism $G^{\circ\circ} \times_S G^{\circ} \to G^{\circ}$, meaning $G^{\circ}$ is closed under the $G^{\circ\circ}$-action by translation. Take the smooth quotients \cite[Proposition 8.3.9]{neron}
				\[ K^{\circ} = G^{\circ}/G^{\circ\circ} \ \ \subset \ \ K = G / G^{\circ\circ} .\]
				Then $G^{\circ}$ will be smooth once we prove $K^{\circ}$ is smooth, reducing the problem to the case when $G \to S$ has a constant fiber dimension $0$.
				
				Assume $G \to S$ is locally quasi-finite. To prove $G^{\circ} \to S$ is \'etale, it is enough to prove it is \'etale at all points on \emph{one} irreducible component $Y$ of $G^{\circ}$. This is because such $Y$ becomes geometrically unibranch, so it cannot have any intersection with other irreducible components of $G^{\circ}$ (while $G^{\circ}$ is connected). Consider the irreducible component $Y = \overline{e(S)} \subset G^{\circ}$, the closure of the identity section. Let $x \in Y \subset G^{\circ}$ be a point and $X = \Spec A \to G^{\circ}$ a connected affine \'etale chart around $x$. Note that $X$ dominates $S$. We prove in the following \Cref{lem:prop:neutral component} that $X$ is \'etale over $S$ in this situation. This completes the proof.
			\end{proof}
			
			\begin{lemma} \label{lem:prop:neutral component}
				Let $f : X \to S$ be a quasi-finite, separated, and dominant morphism over a reduced and geometrically unibranch space $S$. Assume that $X$ is connected and $f$ has smooth fibers. Then $f$ is \'etale.
			\end{lemma}
			\begin{proof}
				Shrinking, we may assume $S$ is Noetherian. Since $f$ is dominant, there exists an integral component $Y \subset X$ that dominates $S$.
				
				We claim that $f$ is \'etale along points in $Y$, or equivalently $f$ is flat along points in $Y$. Let $x \in Y \subset X$ be a point and $s = f(x) \in S$ its image. To show $f$ is flat at $x$, it is enough to show that for each trait $\Delta \to S$ passing through $s$, the base change $X_{\Delta} \to \Delta$ is flat over every point $x' \in X_{\Delta}$ over $x$ (\Cref{prop:valuative criterion for flatness}). Let us first claim that $Y_{\Delta} \to \Delta$ is dominant (``going-down''). By Zariski's main theorem (Tag~082K), $Y \to S$ factorizes through an open immersion followed by a finite morphism $Y \subset \bar Y \to S$. Since $Y \to S$ was dominant, $\bar Y \to S$ is finite surjective. Therefore, we can use the valuative criterion for properness to lift the trait $\Delta \to S$ to a trait $\Delta \to \bar Y$ through $x$. Such a trait factorizes through $Y$, so $Y_{\Delta} \to \Delta$ is dominant. Now that we know $Y_{\Delta} \to \Delta$ is dominant and $f_{\Delta} : X_{\Delta} \to \Delta$ is quasi-finite, $f_{\Delta}$ is universally equidimensional (and hence universally open) along $Y_{\Delta}$. The closed fiber of $f_{\Delta}$ is smooth by assumption, so this proves $f_{\Delta}$ is flat along $Y_{\Delta}$ by \Cref{thm:universal open implies flat}.
				
				Now $f$ is \'etale along $Y$, so $X$ is geometrically unibranch along $Y$. This means $Y$ cannot intersect any other irreducible components of $X$. Since $X$ is connected, this means $Y = X$ and $f$ is \'etale.
			\end{proof}
			
			\begin{corollary} [Generic smoothness of the neutral component] \label{cor:generic smoothness}
				Let $G \to S$ be a group space over a reduced space $S$. Then there exists a dense open subspace $S_0 \subset S$ such that $G^{\circ}_0 \to S_0$ is a smooth group space.
			\end{corollary}
			\begin{proof}
				We claim that there exists a dense open subspace $S_0 \subset S$ such that $S_0$ is geometrically unibranch and $G_0 \to S_0$ has constant fiber dimension. Such $S_0$ will satisfy the desired property by \Cref{prop:neutral component} and \ref{prop:smoothness criterion of group space 2}.
				
				The claim is \'etale local on $S$, so we may assume $S = \Spec A$ is a reduced, Noetherian, and affine scheme. The normalization $S^{\nu} \to S$ is then birational, so we may further assume $S$ is normal (which is geometrically unibranch). Now by the semicontinuity of fiber dimension of $G \to S$ on the base $S$ \cite[{$\mathrm{VI_B}$ Proposition 4.1}]{SGA3I}, there exists a dense open subscheme $S_0 \subset S$ over which the fiber dimension of $G \to S$ is locally constant.
			\end{proof}
			
			\begin{remark}
				The full group space $G \to S$ may not be generically smooth in the above sense (when it is not finite type). See \S \ref{sec:examples of group spaces}.
			\end{remark}

	\subsection{Alignment} \label{sec:alignment}
		A smooth group space $G \to S$ may very well be non-separated, but its strict neutral component $G^{\circ\circ}$ is always separated by \Cref{thm:smooth group space with connected fibers}. This means the non-neutral fiber components of $G$ are responsible for the non-separatedness of $G$. In this section, we will see there are in fact two levels of non-separatedness depending on an \emph{alignment} of the fiber components.
		
		Recall that a group space is separated if and only if its identity section is a closed immersion. Denoting by $E$ the scheme-theoretic image of the identity section $e : S \to G$, the group $G$ is separated if and only if $E \to S$ is an isomorphism.
		
		\begin{definition} [Aligned group] \label{def:alignment}
			A smooth group space $G \to S$ is called \emph{aligned} if the scheme-theoretic image $E$ of its identity section $e : S \to G$ is \'etale over $S$.
		\end{definition}
		
		The idea and terminology of this definition is taken from \cite[Definition 2.11]{hol19}, where the notion of an alignment is defined for a proper flat family of semistable curves $f : \mathcal C \to S$. Our definition models their Theorem~5.17, where they showed that $f : \mathcal C \to S$ is aligned iff its neutral Picard space $\Pic^{\circ}_f \to S$ is aligned in the sense of \Cref{def:alignment}. The following models Theorem~6.2 in loc. cit.
		
		\begin{proposition} \label{prop:separated quotient}
			Let $G \to S$ be a smooth group space over a reduced space $S$. Then the following are equivalent.
			\begin{enumerate}
				\item $G \to S$ is aligned.
				\item There exists a surjective homomorphism $h : G \to \bar G$ to a smooth and separated group space $\bar G \to S$, whose restriction to the strict neutral components $h_{|G^{\circ\circ}} : G^{\circ\circ} \to \bar G^{\circ\circ}$ is an isomorphism.
			\end{enumerate}
			Such a homomorphism $h$ is called a \emph{separated quotient} of $G$.
		\end{proposition}
		\begin{proof}
			(1) $\Longrightarrow$ (2) Assume $E \to S$ is \'etale. Let us first claim that $E$ is a closed \'etale normal subgroup space of $G$. The inverse morphism $i : G \to G$ clearly restricts to $i : E \to E$. To show $m : G \times_S G \to G$ restricts to $E \times_S E \to E$, it is enough to show that for two \'etale local sections $a, b : S \to E$ (after shrinking $S$), its product $a \cdot b : S \to G$ factors through $E$. This is because $a$ and $b$ are the identity over a dense open subspace $S_0 \subset S$, so $a \cdot b$ is the identity over $S_0$ as well. This shows $E$ is a closed subgroup space of $G$. It is characteristic because every (\'etale local) automorphism $f : G \to G$ of the group space $G$ sends $E \subset G$ to $E \subset G$. In particular, it is a normal subgroup.
			
			We may now consider a smooth quotient group space $\bar G = G/E$ \cite[Proposition 8.3.9]{neron}. Since $G^{\circ\circ} \to S$ is separated by \Cref{thm:smooth group space with connected fibers}, $E \cap G^{\circ\circ}$ is trivial hence $G \to \bar G$ induces an open immersion of $G^{\circ\circ}$ to $\bar G$, whose image is necessarily the strict neutral component of $\bar G$.
			
			\bigskip
			
			(2) $\Longrightarrow$ (1)  Conversely, assume we have a separated quotient $h : G \to \bar G$. Notice that $h_s : G_s \to \bar G_s$ is flat for all points $s \in S$ because $h_s (G_s^{\circ\circ}) = \bar G_s^{\circ\circ}$ \cite[Lemma 7.3.1]{neron}. The fiberwise flatness criterion (Tag~05X1) concludes $h$ is flat. By definition, $h$ induces an isomorphism between $G^{\circ\circ}$ and $\bar G^{\circ\circ}$, meaning its kernel is locally quasi-finite. This proves $h$ is \'etale, or equivalently $\ker h \to S$ is \'etale. Now $h$ sends $E$ to the identity section of $\bar G$, so we have $E \subset \ker h$. Since $\ker h$ is \'etale and contains the identity section, this is an equality. Hence $E = \ker h$ is \'etale over $S$.
		\end{proof}
		
		\begin{corollary}
			Let $G \to S$ be a smooth aligned group space over a reduced space with a separated quotient $G \to \bar G$ as above. Then every homomorphism $G \to H$ to a separated group space $H \to S$ uniquely factors through $\bar G$.
		\end{corollary}
		\begin{proof}
			Let $f : G \to H$ be such a homomorphism. Since the identity section of $H$ is a closed immersion, $\ker f$ contains the closure $E$ of the identity section of $G$. The claim follows from $\bar G = G / E$.
		\end{proof}
		
		\begin{proposition} \label{prop:smooth group space over Dedekind is aligned}
			Every smooth group space $G \to \Delta$ over a Dedekind scheme is aligned.
		\end{proposition}
		\begin{proof}
			We may assume $\Delta = \Spec R$ where $R$ is a dvr with function field $K$. Let $E$ be the closure of the identity section. Then $E \to \Delta$ is flat since $E$ is integral and dominating $\Delta$ (\Cref{prop:flat over Dedekind}). To show it is \'etale, note that the locally algebraic group $G_K$ is separated so $E \to \Delta$ is an isomorphism over $K$. Hence the result follows from \cite[Lemma 5.18]{hol19}.
		\end{proof}
		
		\begin{proposition} \label{prop:generic separatedness and alignment}
			Let $f : G \to S$ be smooth group space of finite type over a normal Noetherian space $S$. Then
			\begin{enumerate}
				\item There exists a dense open subspace $S_0 \subset S$ such that $G_0 \to S_0$ is separated.
				\item There exists a dense open subspace $S_1 \subset S$ with codimension $\ge 2$ complement such that $G_1 \to S_1$ is aligned.
			\end{enumerate}
		\end{proposition}
		\begin{proof}
			The first item is \cite[Theorem 8.10.5]{EGAIV3}, so let us focus on the second item. We may assume $S$ is a normal connected scheme. Let $E \subset G$ be the closure of the identity section, which is trivial over $S_0$ by the first item. Consider the \'etale locus $E' \subset E$ of the morphism $E \to S$. The image of its complement $f(E \setminus E') \subset S$ is a constructible set by Chevalley's theorem. Let $\eta \in S$ be a codimension $1$ point and $\Delta = \Spec \mathcal O_{S, \eta} \to G$ a flat trait. Then $E_{\Delta} \subset G_{\Delta}$ is the closure of the identity section \cite[Proposition 2.5.2]{neron}, so it is \'etale over $\Delta$ by \Cref{prop:smooth group space over Dedekind is aligned}. This means the constructible set $f(E \setminus E')$ contains no codimension $1$ points. Hence there exists a codimension $\ge 2$ complement open subset $S_1 \subset S$ such that $E$ is \'etale over $S_1$.
		\end{proof}
		
		\begin{proposition} \label{prop:aligned under etale homomorphism}
			Let $G$ and $H$ be smooth group spaces over $S$, and $f : G \to H$ an \'etale homomorphism between them. If $H$ is aligned, then $G$ is aligned.
		\end{proposition}
		\begin{proof}
			Let $E_G \subset G$ and $E_H \subset H$ be the scheme-theoretic images of the identity sections. Since $f$ is \'etale and $E_H \subset H$ is closed, the preimage $f^{-1}(E_H) \subset G$ is closed and \'etale over $S$. Its closed subscheme $E_G$ is \'etale over $S$ as well.
		\end{proof}

	\subsection{\'Etale group algebraic spaces} \label{sec:etale group space}
		\'Etale group spaces are smooth group spaces of relative dimension $0$, not necessarily of finite type or separated. A finite type \'etale group space is often called an (\'etale) constructible sheaf of finite groups. Here we collect their basic properties.
		
		\begin{proposition}
			Let $G \to S$ be a locally quasi-finite group space over a reduced and geometrically unibranch scheme $S$. Then its neutral component $G^{\circ} \to S$ is an \'etale subgroup space.
		\end{proposition}
		\begin{proof}
			This is \Cref{prop:neutral component} with constant fiber dimension $0$.
		\end{proof}
		
		\begin{lemma} \label{lem:etale separated group space is trivial}
			Let $f : G \to S$ be a separated \'etale group space. If there exists a dense open subscheme $S_0 \subset S$ over which $f$ is trivial, then $f$ is trivial.
		\end{lemma}
		\begin{proof}
			Let $s \in S$ be a point and $x \in G_s$ a closed point. Shrinking $S$ \'etale locally, we can choose a section $s : S \to G$ passing through $x$. Since $G \to S$ is trivial over $S_0$, the section $s$ agrees with the identity section $e$ over $S_0$. By \Cref{lem:morphism to separated space is determined by dense subspace}, $s = e$ and thus $x$ lies on the identity section.
		\end{proof}
		
		Every locally algebraic group $G$ over a field induces an \'etale algebraic group $\pi_0(G) = G / G^{\circ}$, the group of its connected components. This phenomenon generalizes to group spaces with smooth Lie algebra. Recall from \Cref{thm:strict neutral component} that their strict neutral component is a smooth normal subgroup.
		
		\begin{definition} [$\pi_0$-group]
			Let $G \to S$ be a group space with a smooth Lie algebra. The \emph{component group} or \emph{$\pi_0$-group} of $G$ is a locally quasi-finite group space $\pi_0(G) \to S$ obtained by the quotient
			\begin{equation} \label{eq:ses component group}
			\begin{tikzcd}
				1 \arrow[r] & G^{\circ\circ} \arrow[r] & G \arrow[r, "\pi_0"] & \pi_0(G) \arrow[r] & 1
			\end{tikzcd} .
			\end{equation}
		\end{definition}
		
		If $G$ is a smooth group space, then $\pi_0(G)$ is an \'etale group sapce.
		
		\begin{proposition} \label{prop:finite type from pi0 group}
			A group space $G \to S$ with smooth Lie algebra is of finite type if and only if $\pi_0(G)$ is quasi-finite.
		\end{proposition}
		\begin{proof}
			The strict neutral component $G^{\circ\circ}$ is of finite type by \Cref{thm:smooth group space with connected fibers}, so the morphism $\pi_0$ is always of finite type.
		\end{proof}
		
		\begin{proposition}
			Let $G \to S$ be a group space with smooth Lie algebra. Then there exists an inclusion-preserving bijective correspondence
			\begin{multline*}
				\quad \{ \mbox{open (normal) subgroup spaces of } G \} \\
				\longleftrightarrow \{ \mbox{open (normal) subgroup spaces of } \pi_0(G) \} \quad
			\end{multline*}
		\end{proposition}
		\begin{proof}
			The correspondence is defined by $H \mapsto \pi_0(H)$ with inverse $K \mapsto \pi_0^{-1}(K)$. To show they are inverse to each other, it is enough to show that every open subgroup space $H \subset G$ contains the strict neutral component $G^{\circ\circ}$ as a (normal) subgroup. Since both $H$ and $G^{\circ\circ}$ are open, this can be checked pointwise on the base: for each point $t \in S$, we have an open algebraic subgroup $H_t \subset G_t$, so we have $G^{\circ\circ}_t \subset H_t$.
		\end{proof}
		
		If $G \to S$ is a smooth group space of finite type, then $\pi_0(G)$ is an \'etale constructible sheaf of finite groups. Hence open subgroup spaces of $G$ are classified by subsheaves of the constructible sheaf $\pi_0(G)$.
		
		\bigskip
		
		Every \'etale \emph{commutative} group space $G \to S$ has an open subgroup space
		\[ G_{\tor} = \bigcup_{m \in \ZZ_{>0}} \ker \big( [m] : G \to G \big) ,\]
		called the \emph{torsion subgroup} of $G$. We say $G$ is \emph{torsion} if $G_{\tor} \subset G$ is an equality.
		
		\begin{definition} [Numerically trivial commutative group] \
			\begin{enumerate}
				\item A smooth commutative group space $G \to S$ is called \emph{numerically trivial} if its $\pi_0$-group $\pi_0(G) \to S$ is torsion.
				\item Given a smooth commutative group space $G \to S$, its \emph{numerically trivial subgroup} is an open subgroup space
				\[ G^{\tau} \subset G \]
				corresponding to the torsion subgroup space $\pi_0(G)_{\tor} \subset \pi_0(G)$.
			\end{enumerate}
		\end{definition}
		
		\begin{proposition} \label{prop:numerically trivial group is aligned}
			Let $G \to S$ be a smooth commutative group space over a normal scheme $S$. If $G$ is numerically trivial, then it is aligned. In particular, if $G$ is of finite type, then it is aligned.
		\end{proposition}
		\begin{proof}
			Let $G \to S$ be an arbitrary smooth commutative group space. Set $E$ to be the closure of the identity section, $E' \subset E$ the \'etale locus of $E \to S$, and $Z = E \setminus E'$ its complement, a closed subscheme of $G$. We claim that $\pi_0(Z) \subset \pi_0(G)$ is closed under taking power, i.e., if $\alpha \in \pi_0(Z)$ and $m \in \ZZ_{>0}$ then $m \cdot \alpha \in \pi_0 (Z)$. This in particular implies that $\pi_0(Z) \cap \pi_0(G)_{\tor} = \emptyset$ and hence the desired claim.
			
			Let $\alpha \in \pi_0(E)$ be a point and $s \in S$ its image. Denote by
			\[ G_s^{\alpha} = \pi_0^{-1} (\alpha) \subset G_s ,\qquad E_s^{\alpha} = E \cap G_s^{\alpha} \]
			the connected component of a locally algebraic group $G_s$ and the restriction of $E$ to it, respectively. We claim that
			\[ E_s^{\alpha} = \begin{cases}
				\mbox{singleton} \quad & \mbox{when } \ \alpha \in \pi_0 (E') ,\\
				\mbox{positive dimensional} \quad & \mbox{when } \ \alpha \in \pi_0(Z) .
			\end{cases} \]
			Choose an arbitrary section $a : S \to G$ through the component $G_s^{\alpha}$ (possibly after shrinking $S$) and consider its $G^{\circ\circ}$-orbit $X = G^{\circ\circ}.a \subset G$, an open subspace because the $G^{\circ\circ}$-action on $G$ is free. \Cref{thm:smooth group space with connected fibers} says $X$ is separated and of finite type. If $\alpha \in \pi_0(E')$ then the intersection $E \cap X \to S$ is separated and \'etale over $s \in S$, so $E^{\alpha}_s = (E \cap X)_s$ is a point. If $\alpha \in \pi_0(Z)$ then $E \cap X \to S$ is separated and birational but not \'etale over $s \in S$, so by Zariski's main (connectedness) theorem its fiber $E^{\alpha}_s$ is positive dimensional and connected.
			
			Now the multiplication endomorphism $[m] : G \to G$ is \'etale so we have $m \cdot E^{\alpha}_s = E^{m \cdot \alpha}_s$. If $\alpha \in \pi_0(Z)$ then $E^{\alpha}_s$ had dimension $\ge 1$, so its image under $[m]$ has dimension $\ge 1$ and hence $m \cdot \alpha \in \pi_0(Z)$. This completes the proof.
		\end{proof}

	\subsection{Main component} \label{sec:main component}
		Let $S$ be an integral scheme. By generic smoothness in \Cref{cor:generic smoothness}, there exists a dense open subscheme $S_0 \subset S$ such that the neutral component $G^{\circ}_0 \to S_0$ of $G_0 \to S_0$ is a smooth subgroup space.
		
		\begin{definition} [Main component] \label{def:main component}
			Let $G \to S$ be a separated group space over an integral scheme $S$. Let $S_0 \subset S$ be a dense open subscheme over which the neutral component $G^{\circ}$ is smooth (\Cref{cor:generic smoothness}). The \emph{main component} $M$ of $G$ is the Zariski closure of $G^{\circ}_0$ in $G$ with the reduced induced structure.
		\end{definition}
		
		Consider the generic fiber $G_K$, a locally algebraic group over the function field $K$. Then the main component $M$ can be alternatively described by the Zariski closure of the generic point $\xi \in G^{\circ}_K$, proving its well-definedness independent to the choice of $S_0 \subset S$. Since $M_0 = G^{\circ}_0$ is a subgroup space of $G_0$, it is natural to ask whether $M$ is still a subgroup space of $G$ over the full base $S$. Unfortunately, this statement is too optimistic to hold in general (\S \ref{sec:examples of group spaces}).
		
		The goal of this section is to answer this question positively in certain special cases. Before stating the theorem, we note that a Lie algebra scheme over $S$ is a commutative group scheme, so it makes sense to consider its main component as in \Cref{def:main component}.
		
		\begin{theorem} \label{thm:main component 1}
			Let $G \to S$ be a separated group space over a reduced and geometrically unibranch scheme $S$. Assume that the main component $\mathfrak h$ of $\Lie G$ is smooth over $S$. Then the main component $M$ of $G$ has a unique open subspace $H \subset M$ with the following properties:
			\begin{enumerate}
				\item $H \to S$ is universally equidimensional (see \S \ref{sec:equidimensional morphism}).
				\item For every universally open morphism $Z \to S$ from a reduced algebraic space $Z$, every morphism $f : Z \to M$ uniquely factors through
				\[\begin{tikzcd}[column sep=tiny]
					Z \arrow[rr, "f"] \arrow[rd] & & M \\
					& H \arrow[ru, hook]
				\end{tikzcd}.\]
			\end{enumerate}
		\end{theorem}
		
		The subspace $H \subset G$ behaves like a subgroup space in the following sense. For the notion of a non-degenerate trait, see \S \ref{sec:trait}.
		
		\begin{proposition} \label{prop:main component 2}
			In \Cref{thm:main component 1}, we have the following properties of $H$.
			\begin{enumerate}
				\item The following are equivalent.
				\begin{enumerate}
					\item $H \to S$ is a locally closed subgroup space of $G$.
					\item $H \to S$ has reduced fibers.
				\end{enumerate}
				In this case, $H \to S$ is smooth and $\Lie H = \mathfrak h$.
				
				\item Fix a dense open subscheme $S_0 \subset S$ over which $G^{\circ}$ is smooth. Then for every non-degenerate trait $\Delta \to S$ with respect to $S_0 \subset S$, reduction of the base change
				\[ (H_{\Delta})_{\red} \to \Delta \]
				is a smooth locally closed subgroup of $G_{\Delta}$. As a result,
				\begin{enumerate}
					\item $(H_{\Delta})_{\red}$ is an open subgroup space of the main component of $G_{\Delta}$ (see \Cref{prop:main component over Dedekind}).
					\item For every point $s \in S$, the reduction $(H_s)_{\red}$ is a closed algebraic subgroup of $G_s$.
				\end{enumerate}
			\end{enumerate}
		\end{proposition}
		
		\begin{remark}
			There are examples where $H \to S$ has non-reduced fibers: see \S \ref{sec:examples of group spaces}. We do not know whether the inclusion $H \subset M$ in \Cref{thm:main component 1} can be strict.
		\end{remark}
		
		The rest of this section is devoted to the proof of \Cref{thm:main component 1} and \ref{prop:main component 2}.
		
		\begin{lemma} \label{lem:main component criterion to be subgroup}
			Let $S$ be an integral scheme.
			\begin{enumerate}
				\item Let $G \to S$ be a separated group space and $M$ its main component. If $M \times_S M$ is integral, then $M \to S$ is a closed subgroup space of $G$.
				
				\item Let $V \to S$ be a linear (resp. Lie algebra) scheme and $W$ its main component. If $W \times_S W$ is integral, then $W \to S$ is a closed sublinear (resp. Lie subalgebra) scheme of $V$.
			\end{enumerate}
		\end{lemma}
		\begin{proof}
			Let $\eta \in S$ be the generic point. The main component $M$ is characterized by the maximal integral subspace of $G$ containing $e(\eta)$. Since the inverse automorphism $i : G \to G$ sends $e(\eta)$ to itself, it restricts to $i : M \to M$. Similarly, let $m : G \times_S G \to G$ be the multiplication morphism. Since we have assumed $M \times_S M$ is integral, the scheme-theoretic image of $M \times_S M$ by $m$ is an integral subspace of $G$ containing $e(\eta) \cdot e(\eta) = e(\eta) \in G$. We thus have a restriction $m : M \times_S M \to M$. The group axioms are automatically satisfied.
			
			Same argument applies to the second item and proves that $W$ respects the Lie bracket operation. The scalar multiplication $\AA^1_S \times_S V \to V$ restricts to $\AA^1_S \times_S W \to W$ because $\AA^1_S \times_S W$ is integral.
		\end{proof}
		
		\begin{remark} \label{rmk:smooth main component is subgroup scheme}
			Let $\mathfrak g \to S$ be a Lie algebra scheme over an integral scheme. If its main component $\mathfrak h \to S$ is smooth (as in the assumption of \Cref{thm:main component 1}), then it is automatically a closed Lie subalgebra scheme from \Cref{lem:main component criterion to be subgroup}.
		\end{remark}
		
		The following is a version of \Cref{thm:main component 1} when $S = \Delta$ is a Dedekind scheme.
		
		\begin{proposition} \label{prop:main component over Dedekind}
			Let $G \to \Delta$ be a separated group space over a Dedekind scheme. Then its main component $H = M \to \Delta$ is always a smooth and closed subgroup space of $G$. Moreover, $\Lie H$ is the main component of $\Lie G$.
		\end{proposition}
		\begin{proof}
			The morphism $H \to \Delta$ is flat because $H$ is integral (\Cref{prop:flat over Dedekind}). Hence
			\[ H^{\times 2} \coloneq H \times_{\Delta} H \to \Delta \]
			is flat as well. Fix an open subset $\Delta_0 \subset \Delta$ over which $G^{\circ}$ is smooth (\Cref{cor:generic smoothness}). Then $H^{\times 2}_0 = (G_0^{\circ})^{\times 2}$ is a smooth connected group space over $\Delta_0$. Together with the flatness of $H^{\times 2} \to \Delta$, this shows (1) $H^{\times 2}$ is irreducible because $H^{\times 2}_0$ is, (2) it is $\mathrm{(R_0)}$ because $H^{\times 2}_0$ is regular, and (3) it is $\mathrm{(S_1)}$ because it has no embedded points by \Cref{prop:flat over Dedekind}. This proves $H^{\times 2}$ is integral, and therefore $H$ is a closed subgroup space by \Cref{lem:main component criterion to be subgroup}. It is smooth because it is flat (\Cref{prop:smoothness criterion of group space 1}).
			
			To show $\Lie H \subset \Lie G$ is the main component, simply notice that $\Lie H \to \Delta$ is an $\AA^d$-bundle so $\Lie H$ is integral. Since $\Lie H = \Lie G$ generically, it is the main component.
		\end{proof}
		
		\begin{remark} \label{rmk:main component over Dedekind for Lie algebra}
			Same argument shows that the main component of a Lie algebra scheme over $\Delta$ is always a smooth subalgebra scheme.
		\end{remark}
		
		\begin{lemma} \label{lem:main component is equidimensional along section}
			Let $G \to S$ be a separated group space over an integral scheme $S$ and $M$ its main component. Assume that the main component $\mathfrak h$ of $\Lie G$ is smooth over $S$. Then for every section $u : S \to M$, the morphism $M \to S$ is equidimensional along $u(S) \subset M$.
		\end{lemma}
		\begin{proof}
			Translating everything in $G$ by $u^{-1}$, we may assume $u = e$ is the identity section. Consider the deformation of $G \to S$ to its Lie algebra (\S \ref{sec:Lie algebra scheme})
			\[ g : \mathfrak G \to S \times \PP^1 .\]
			To minimize potential confusion between the special fiber of $\mathfrak G$ and restriction of $\mathfrak G$ to the open subset $S_0$, we write $t_0 = \infty \in \PP^1$ for the special point: the special fiber is set-theoretically $\mathfrak G_{t_0} = \Lie G$. Let $S_0 \subset S$ be a dense open subscheme such that $M_0 \coloneq G^{\circ}_0 \to S_0$ is a smooth group space. Its deformation to the Lie algebra $\mathfrak M_0 = \mathfrak G^{\circ}_0 \to S_0 \times \PP^1$ satisfies the following:
			\begin{itemize}
				\item $\mathfrak M_0 \to S_0 \times \PP^1$ is smooth (\Cref{lem:deformation to Lie algebra smooth}) and is a restriction of $g$ (\Cref{lem:deformation to Lie algebra respects base change}).
				\item $(\mathfrak M_0)_{t_0} = \mathfrak h_0$ over $S_0$ scheme-theoretically (\Cref{lem:deformation to Lie algebra smooth} and \ref{rmk:smooth main component is subgroup scheme}).
			\end{itemize}
			Consider the closure $\mathfrak M = \overline{\mathfrak M_0} \subset \mathfrak G$ with a reduced induced structure:
			\[\begin{tikzcd}
				\mathfrak M \arrow[r, phantom, "\subset"] \arrow[rd, "m"'] & \mathfrak G \arrow[d, "g"] \\
				& S \times \PP^1
			\end{tikzcd}.\]
			Over the open subset $S \times (\PP^1 \setminus \{ t_0 \})$, this diagram is a trivial family of the main component $M \subset G$ over $S$. Over the open subset $S_0 \times \PP^1$, we had $\mathfrak M_0 = \mathfrak G^{\circ}_0$ by construction. Therefore, over $S \times \{ t_0 \}$, we have a sequence of closed immersions
			\[ \mathfrak h \subset \mathfrak M_{t_0} \subset \mathfrak G_{t_0} \subset \Lie G \qquad\mbox{over }\ S ,\]
			where the last inclusion is a set-theoretic equality.
			
			We claim that the first closed immersion $\mathfrak h \subset \mathfrak M_{t_0}$ is a set-theoretical equality (i.e., $\mathfrak h$ is the reduction of $\mathfrak M_{t_0}$). Let $\Delta \to S$ be a non-degenerate trait with respect to $S_0$ and $G_{\Delta} \to \Delta$ the base change. By \Cref{prop:main component over Dedekind}, the main component of $G_{\Delta}$ is a smooth subgroup space $H^{\Delta} \to \Delta$ with $\Lie H^{\Delta} = \mathfrak h_{\Delta}$. Consider its deformation to the Lie algebra $h^{\Delta} : \mathfrak H^{\Delta} \to \Delta \times \PP^1$, which is smooth by \Cref{lem:deformation to Lie algebra smooth}. From \Cref{lem:deformation to Lie algebra respects base change} we can form a commutative diagram
			\[\begin{tikzcd}[column sep=tiny]
				\mathfrak H^{\Delta} \arrow[r, phantom, "\subset"] \arrow[rd, "h^{\Delta}"'] & \mathfrak M_{\Delta} \arrow[r, phantom, "\subset"] \arrow[d, "m_{\Delta}"] & \mathfrak G_{\Delta} \arrow[ld, "g_{\Delta}"] \\
				& \Delta \times \PP^1
			\end{tikzcd}.\]
			Since $\mathfrak M$ was the closure of $\mathfrak M_0$, for each point $x \in \mathfrak M$ there exists a non-degenerate trait $a : \Delta \to \mathfrak M$ through $x$ with respect to $\mathfrak M_0$ \cite[Lemma 5.2]{hol19}. Composing with $\mathfrak M \to S \times \PP^1 \to S$, we obtain a non-degenerate trait $\Delta \to S$ with respect to $S_0$:
			\[\begin{tikzcd}
				\Delta \arrow[r, "a"] \arrow[rd, "\id"'] & \mathfrak M_{\Delta} \arrow[r] \arrow[d] & \mathfrak M \arrow[d] \\
				& \Delta \arrow[r] & S
			\end{tikzcd}.\]
			Recalling the equality $\mathfrak M_0 = \mathfrak G^{\circ}_0$, we can see the section $a$ sends $\Delta_0$ into $(\mathfrak H^{\Delta})_{|\Delta_0} = \mathfrak M_{\Delta_0} = \mathfrak G^{\circ}_{\Delta_0}$. In particular, the image of $a$ is in $\mathfrak H^{\Delta}$. We have proved a set-theoretic equality
			\[ \mathfrak M = \bigcup_{\Delta \to S} \mathfrak H^{\Delta} ,\]
			where $\Delta \to S$ runs through every non-degenerate trait of $S$. In particular, over a point $(s, t_0) \in S \times \PP^1$ we have a set-theoretic equality
			\[ \mathfrak M_{(s, t_0)} = \bigcup_{\Delta \to S} \mathfrak H^{\Delta}_{(s, t_0)} ,\qquad \mathfrak H^{\Delta}_{(s, t_0)} = (h^{\Delta})^{-1} (s, t_0) \]
			where $\Delta \to S$ runs through every non-degenerate trait through $s$. Since we already know an isomorphism $\mathfrak H^{\Delta}_{t_0} = \mathfrak h_{\Delta}$, we have $\mathfrak H^{\Delta}_{(s, t_0)} = \mathfrak h_s$ for every $\Delta \to S$ through $s$. This shows $\mathfrak M_{(s, t_0)} = \mathfrak h_s$ and hence $\mathfrak M_{t_0} = \mathfrak h$ set-theoretically.
			
			We now know the set-theoretic equality $\mathfrak M_{t_0} = \mathfrak h$. Say the generic fiber of $G \to S$ has dimension $d$. Our goal was to prove $M \to S$ has constant fiber dimension $d$ along the identity section $e(S)$. To do so, we study the fiber dimensions of the entire morphism $m : \mathfrak M \to S \times \PP^1$. Its fiber dimension function
			\[ \dim_m : \mathfrak M \to \ZZ, \qquad x \mapsto \dim_x \mathfrak M_{m(x)} \]
			is an upper semicontinuous function with minimum value $d$. From $\mathfrak M_{t_0} = \mathfrak h$ we obtain $\dim_m(x) = d$ for every $x \in e(S \times \{ t_0 \})$. By semicontinuity, there exists a Zariski open subset $V \subset \mathfrak M$ containing $e(S \times \{ t_0 \})$ such that $\dim_m(x) = d$ for every $x \in V$. The open set $V$ contains $e(S \times \{ \eta \})$ where $\eta \in \PP^1$ is the generic point. This shows $M \to S$ has the fiber dimension $d$ along $e(S)$.
		\end{proof}
		
		\begin{proof}[Proof of \Cref{thm:main component 1}]
			Since $H$ is reduced and $H \to S$ is universally open (\Cref{prop:universally equidimensional morphism 1}), the universal property proves the uniqueness of such $H$. To prove its existence, denote by $g : G \to S$ the given group space and $d$ its generic relative dimension. Set
			\[ H = \{ x \in M : \dim_x M_{g(x)} = d \} ,\]
			which is an open subspace of $M$ by Chevalley's theorem on upper semicontinuity of fiber dimension. Since $H$ is irreducible and $S$ is geometrically unibranch, $H \to S$ is universally equidimensional by \Cref{prop:universally equidimensional morphism 2}.
			
			It remains to prove that such $H$ satisfies the universal property. Let $Z \to S$ be a universally open morphism from a reduced algebraic space $Z$ and $u : Z \to M$ an arbitrary morphism. Since $Z$ is reduced, it is enough to show a set-theoretic inclusion $u(Z) \subset H$ or equivalently, show $u(Z_s) \subset H_s$ for every $s \in S$. By \Cref{prop:quasi-finite local section} and \ref{rmk:local section}, there exists a section $S' \to Z$ where $S' \to S$ is a quasi-finite and universally equidimensional morphism (\Cref{prop:universally equidimensional morphism 2}) from an integral scheme $S'$ containing $s \in S$ in its image. The composition $S' \to Z \to G$ defines an $S'$-section of $G$:
			\[\begin{tikzcd}
				S' \arrow[r, "h"] \arrow[rd, "\id"'] & G' \arrow[r] \arrow[d] & G \arrow[d] \\
				& S' \arrow[r] & S
			\end{tikzcd}.\]
			Since $G' \to G$ is universally equidimensional whence open, the main component of $G'$ dominates the main component of $G$; in other words, the main component of $G'$ is the reduction $M'_{\red}$ of the base change $M' = M \times_S S'$. By \Cref{lem:main component is equidimensional along section}, this proves $M' \to S'$ has a constant fiber dimension $d$ along $h(S')$, or equivalently $M \to S$ has a constant fiber dimension $d$ along $h(S')$. By varying the section $S' \to Z$ passing through closed points of $Z_s$, this shows the image of $u : Z \to M$ is equidimensional over $S$. That is, $u(Z_s) \subset H_s$.
		\end{proof}
		
		\begin{proof} [Proof of \Cref{prop:main component 2}]
			The identity morphism $e : S \to G$ factors through $H$ by the universal property in \Cref{thm:main component 1}. The inverse morphism $i : G \to G$ similarly restricts to $i : H \to H$. The multiplication morphism $m : G \times_S G \to G$, however, only restricts to
			\[ (H \times_S H)_{\red} \to H ,\]
			since $H \times_S H \to S$ is universally open but the reducedness of $H \times_S H$ is uncertain.
			
			\bigskip
			
			(1) The implication (a) $\Longrightarrow$ (b) is clear from Cartier's theorem. To show (b) $\Longrightarrow$ (a), use \Cref{thm:universal open implies flat} to deduce that $H \to S$ is flat. Hence $H \times_S H \to S$ is flat and has reduced fibers by (Tag~035Z). By \cite[Theorem 23.9]{mat:comm-ring}, this implies the total space $H \times_S H$ is reduced and hence the multiplication morphism on $H$ above is well-defined. This makes $H \to S$ a subgroup space of $G$. It is smooth because it is universally equidimensional (\Cref{prop:smoothness criterion of group space 1}).
			
			\bigskip
			
			(2) Let $\Delta \to S$ be a non-degenerate trait. Since $H \to S$ is universally equidimensional, its base change $H_{\Delta} \to \Delta$ is again universally equidimensional. Hence the reduction $H_{\Delta, \red} = (H_{\Delta})_{\red} \to \Delta$ is flat by \Cref{prop:flat over Dedekind}. The above discussion shows that $H_{\Delta, \red}$ has the identity, inverse morphism, and the multiplication up to the non-reducedness ambiguity
			\[ m_{\Delta} : (H_{\Delta} \times_{\Delta} H_{\Delta})_{\red} \to H_{\Delta, \red} .\]
			But $H_{\Delta, \red} \to \Delta$ was flat, so $H_{\Delta, \red} \times_{\Delta} H_{\Delta, \red} \to \Delta$ is again flat and hence $H_{\Delta, \red} \times_{\Delta} H_{\Delta, \red}$ is reduced by \Cref{prop:flat over Dedekind}. This shows $m_{\Delta}$ is indeed the multiplication morphism on $H_{\Delta, \red}$.
		\end{proof}

	\subsection{Automorphism space, Picard space, and examples}
		\subsubsection{The automorphism group space} \label{sec:automorphism space}
			Given an algebraic space $f : X \to S$, its automorphism sheaf is an fppf sheaf of groups
			\[ \SheafAut_f : (\Sch/S)^{op} \to \Group, \qquad \big[ T \to S \big] \mapsto \{ g : X_T \to X_T : \mbox{automorphism over } T \} .\]
			
			\begin{theorem} [Automorphism space] \label{thm:automorphism space}
				Let $f : X \to S$ be a proper flat algebraic space over a locally Noetherian scheme. Then
				\begin{enumerate}
					\item $\SheafAut_f$ is representable by a separated and locally finite type group space $\Aut_f \to S$.
					\item There exists a canonical isomorphism of fppf $\mathcal O_S$-modules
					\[ \Lie \Aut_f = f_* T^{\fppf}_f .\]
				\end{enumerate}
			\end{theorem}
			
			Here $T^{\fppf}_f = \Hom^{\fppf}_X (\Omega_f, \mathcal O_X)$ is the fppf dual \eqref{eq:fppf dual} of the relative cotangent sheaf (that is represented by the tangent space $\TT_f$ \eqref{eq:tangent scheme}).
			
			\begin{proof}
				\cite[Lemma 5.1]{ols06} or (Tag~0D19) shows that the morphism of fppf sheaves assigning a morphism to its graph
				\[ \SheafHom_S (X,X) \hookrightarrow \underline{\operatorname{Hilb}}_{X \times_S X/S}, \qquad \big[ g : X_T \to X_T \big] \mapsto \big[ \Gamma_g \subset (X \times_S X)_T \big] \]
				is (representable by) an open immersion. Since $\SheafAut_f = \SheafHom_S (X,X) \cap t(\SheafHom_S (X,X))$ where $t$ is the transposition involution on $\operatorname{Hilb}_{X \times_S X/S}$, we get an open immersion $\SheafAut_f \hookrightarrow \underline{\operatorname{Hilb}}_{X \times_S X/S}$. The Hilbert sheaf is representable by a separated and locally finite type algebraic space (Tag~0DM7), so this proves the first item. The second item is \cite[Proposition 8.5]{ari-fed16}.
			\end{proof}
			
			We call such a group space $\Aut_f \to S$ the \emph{automorphism (group) space} of $f$. If $f : X \to S$ is a flat projective scheme then the Hilbert functor is representable by a scheme, so $\Aut_f \to S$ is a group scheme with quasi-projective connected components.

		\subsubsection{The Picard group space} \label{sec:Picard space}
			Given an algebraic space $f : X \to S$, its Picard sheaf is $\underline{\Pic}_f = R^1f_* \GG_m$, the higher direct image of an fppf abelian sheaf $\GG_m$. It is equivalently an fppf sheafification of the presheaf
			\[ (\Sch/S)^{op} \to \Group, \qquad \big[ T \to S \big] \mapsto \Pic(X_T) / \Pic(T) .\]
			
			\begin{theorem} [Picard space] \label{thm:Picard space}
				Let $X$ be an algebraic space and $f : X \to S$ a proper flat morphism over a locally Noetherian scheme with $f_* \mathcal O_X^{\fppf} = \mathcal O_S^{\fppf}$.\footnote{A more standard way of saying this is: \emph{$f_* \mathcal O_X = \mathcal O_S$ holds universally} or \emph{$f$ is cohomologycally flat in dimension $0$}.} Then
				\begin{enumerate}
					\item $\underline{\Pic}_f$ is representable by a locally separated, locally finite type, and commutative group space $\Pic_f \to S$.
					\item There exists a canonical isomorphism of fppf $\mathcal O_B$-modules
					\[ \Lie \Pic_f = \big( R^1 f_* \mathcal O_X \big)^{\fppf} .\]
				\end{enumerate}
			\end{theorem}
			
			Here $\big( R^1 f_* \mathcal O_X \big)^{\fppf}$ is the associated fppf sheaf \eqref{eq:associated fppf sheaf} of the coherent sheaf $R^1 f_* \mathcal O_X$.
			
			\begin{proof}
				The first item is \cite[Theorem 7.3]{artin69}, \cite[\S 8.3]{neron}, or (Tag~0D24). The second item is \cite[Theorem 8.4.1]{neron} or \cite[Proposition 1.3]{liu-lor-ray04}.
			\end{proof}
			
			Such a commutative group space $\Pic_f$ is called the \emph{Picard (group) space} of $f$.

		\subsubsection{Examples and counterexamples of group spaces} \label{sec:examples of group spaces}
			\begin{example} [Basic examples] \label{ex:group space basic examples}
				Let $S = \AA^1_x$ be an affine line with coordinate $x$ over a field $k$ of characteristic $0$.
			
			\begin{itemize}
				\item Let $G$ be an affine line with double origin. Then the morphism collapsing the two origins $G \to S$ has a smooth and non-separated group scheme structure. The fibers over any nonzero $s \in S$ is trivial, and the fiber over $0$ is $\ZZ/2$.
				
				\item Let $G = \AA^1 \sqcup \{ 0 \}$. Then $G \to S$ has an obvious group structure, which is separated and non-smooth. Its Lie algebra scheme is smooth. The morphism $G \to S$ is not equidimensional in the sense of \S \ref{sec:equidimensional morphism}.
				
				\item Let $G = (xy = 0) \subset \AA^2_{(x,y)}$. Then the projection $G \to S$ admits a clear group structure, which is separated and non-smooth. Indeed, its fiber dimension jumps at $0 \in S$ so it is not equidimensional.
				
				\item Let $G = ((x+y)y = 0) \subset \AA^2_{(x,y)}$. The projection $G \to S$ is flat but can never admit any group scheme structure. This is because its central fiber $\Spec k[y]/y^2$ is non-reduced, so it cannot admit any algebraic group structure by Cartier's theorem.
				
				\item Fix a set of infinite closed points $Z = \{ i \in \AA^1(k) : i \in \ZZ \}$ and $G = \AA^1 \sqcup Z$. The projection $G \to S$ has a group scheme structure whose fibers over $Z$ are $\ZZ/2$ and trivial elsewhere. Then $G \to S$ cannot be smooth over every Zariski open subset $S_0 \subset S$ (generic smoothness in \Cref{cor:generic smoothness} fails for $G$ when it is not finite type).
				
				\item This time, set $G$ to be an affine line with infinitely many double points over $Z$. Again the projection $G \to S$ has a group scheme structure. It is non-separated over every Zariski open subset $S_0 \subset S$ (generic separatedness in \Cref{prop:generic separatedness and alignment} fails).
			\end{itemize}
		\end{example}
		
		\begin{example} [Equidimensional but non-smooth group scheme]
			Set $S = (xy = 0) \subset \AA^2_{(x,y)}$ and $G = S \sqcup \AA^1_x \sqcup \big( \AA^1_y \setminus (y = 0) \big)$. The morphism $f : G \to S$ admits a group scheme structure with $\ZZ/2$ fibers. It is equidimensional but not universally equidimensional, so it is not smooth by \Cref{prop:smoothness criterion of group space 1}. Similar example works for a nodal cubic curve $S$, in which case the base is integral.
		\end{example}
		
		\begin{example} [Group scheme with non-reduced total space] \label{ex:group scheme with non-reduced total space}
			Let $C = (x^3 = y^2) \subset \AA^2_{(x,y)}$ be a cuspidal cubic. Its tangent scheme $\TT_C = \Spec_C (\Sym^* \Omega_C) \to C$ is a separated commutative group scheme (it is a linear scheme). Writing $u = 3dx$ and $v = 2dy$, the total space $\TT_C$ is a closed subscheme of $\AA^2_{(x,y)} \times \AA^2_{(u,v)}$ cut out by an ideal
			\[ I = (x^3 - y^2, \ x^2u - yv) \ \subset \ k [x,y,u,v] ,\]
			which is not radical; while $xv - yu \notin I$, we have $(xv - yu)^3 \in I$.
		\end{example}
		
		\begin{example} [Weak and strict neutral components]
			Let $\Delta = \Spec R$ be a local Dedekind scheme with residue field $k$ and function field $K$.
			\begin{itemize}
				\item Let $f : S \to \Delta$ be a minimal elliptic surface with a central fiber $S_{t_0}$ of Kodaira type $\mathrm{I}_2$. Its relative Picard scheme $\Pic_f \to \Delta$ is smooth and non-separated \cite[Proposition 8.4.2]{neron}. The scheme $\Pic_f$ has $\NS(S_K) \cong \ZZ$ connected components. Its neutral component $\Pic^{\circ}_f \to \Delta$ is smooth (and non-separated) by \Cref{prop:neutral component} and parametrizes line bundles of multi-degree $(d, -d)$ for all $d \in \ZZ$ on the $\mathrm{I}_2$ central fiber $S_{t_0}$. Its strict neutral component $\Pic^{\circ\circ}_f \to \Delta$ is smooth by \Cref{thm:strict neutral component} and parametrizes line bundles of multi-degree $(0,0)$ on the central fiber $S_{t_0}$.
				
				\item In \cite[Example 10.1.5]{neron}, a smooth, separated and connected group scheme $G \to \Delta$ with generic fiber $G_K = \GG_{m, K}$ and special fiber $G_{t_0} = \ZZ \times \GG_{m, k}$ of infinitely many connected components is constructed (it is the N\'eron model of $\GG_{m, K}$ over $\Delta$). We have $G^{\circ} = G$ because $G$ is connected. However, its strict neutral component is $G^{\circ\circ} = \GG_{m, \Delta}$.
			\end{itemize}
		\end{example}
		
		\begin{example} [$\pi_0$-group]
			Let $f : X \to S$ be a flat projective morphism with $f_* \mathcal O_X^{\fppf} = \mathcal O_S^{\fppf}$, so that the Picard space $\Pic_f \to S$ is well-defined. Assume that $R^1f_* \mathcal O_X$ is locally free, or $\Lie \Pic_f \to S$ is smooth by \Cref{thm:Picard space}. The $\pi_0$-group $\NS_f = \pi_0(\Pic_f) \to S$ is well-defined, parametrizing the N\'eron--Severi groups $\NS(X_s)$ for $s \in S$. It is a locally quasi-finite group space. A well-known fact is that the Picard rank function
			\[ S \to \ZZ, \qquad s \mapsto \rk \NS(X_s) \]
			is wild and in general not a constructible function. This is explained by the last example in \Cref{ex:group space basic examples} (and that $\NS_f$ may have infinitely many disconnected components).
		\end{example}
		
		\begin{example} [Non-aligned group]
			In \cite{hol19}, non-aligned smooth group spaces $G \to S$ over a regular Noetherian base $S$ are constructed. Holmes' construction is the neutral Picard scheme $\Pic^{\circ}_f \to S$ for a flat projective family of non-aligned stable curves $f : \mathcal C \to S$. See the original reference for more details.
		\end{example}
		
		\begin{example} [Automorphism scheme and main component]
			Let $f : S \to \Delta$ be a minimal elliptic surface over a local Dedekind scheme $\Delta$ with a central fiber of Kodaira type $\textnormal{I}_2$:
			\[ S_{t_0} = C_0 \cup C_1 \quad \mbox{for smooth rational curves } C_0 \mbox{ and } C_1 .\]
			Assume $f$ has sections $T_0$ and $T_1$ passing through $C_0$ and $C_1$, respectively.
			
			The automorphism scheme $\Aut_f \to \Delta$ is a separated group scheme. Its generic fiber is an algebraic group $\Aut_{S_K/K}$ whose neutral component is isomorphic to $S_K$ (translation automorphisms of the elliptic curve $S_K$). Its central fiber $\Aut_{f, t_0}$ is a (non-commutative) linear algebraic group
			\[ \Aut_{f, t_0} = \big( \GG_m^{\times 2} \times \ZZ/2 \big) \rtimes \mathfrak S_2 ,\]
			where $\mathfrak S_2$ acts on $\GG_m^{\times 2}$ by swapping factors. In particular, $\Aut_f \to \Delta$ is not equidimensional and hence non-smooth. Nevertheless, thanks to \Cref{prop:main component over Dedekind}, its main component $H \subset \Aut_f$ is a smooth and closed subgroup scheme over $\Delta$. It is commutative because it is generically commutative and separated.
		\end{example}
		
		\begin{example} [Picard scheme and separated quotient]
			Let us keep working on the previous example $f : S \to \Delta$. The Picard scheme $\Pic_f$ and its neutral component $\Pic^{\circ}_f$ are smooth and non-separated group schemes over $\Delta$ (e.g., \cite[\S 8--9]{neron}). We have two short exact sequences of group schemes
			\[ 0 \to \Pic^{\circ}_f \to \Pic_f \to \ZZ \to 0 ,\qquad 0 \to \GG_{m, k} \to \Pic_{f, t_0} \to \ZZ^2 \to 0 ,\]
			each over $\Delta$ and over the closed point $t_0 \in \Delta$, respectively. The discrepancy between $\NS(S_K) = \ZZ$ and $\NS(S_{t_0}) = \ZZ^2$ explains the strict inclusion $\Pic^{\circ\circ}_f \subsetneq \Pic^{\circ}_f$, and that $\Pic^{\circ}_f \to \Delta$ is not of finite type. More specifically, consider a line bundle $\mathcal O_S(C_1)$ on $S$. It is generically trivial on $S_K$ but has a multi-degree $(2, -2) \in \ZZ^2$ on the central fiber $S_{t_0}$ (compute the intersection numbers). That is, it is a section of $\Pic^{\circ}_f$ that cannot be separated to the identity. Consider a line bundle $\mathcal O_S(T_0 - T_1)$. It is a section of $\Pic^{\circ}_f$ with multi-degree $(1, -1) \in \ZZ^2$ on the central fiber $S_{t_0}$.
			
			The group scheme $\Pic^{\circ}_f \to \Delta$ is aligned by \Cref{prop:smooth group space over Dedekind is aligned}, so it admits a separated quotient $\Pic^{\circ}_f \twoheadrightarrow \check P$ by \Cref{prop:separated quotient}. The closure of the identity section $E \subset \Pic^{\circ}_f$ is an \'etale group space over $\Delta$ whose central fiber passes through the component
			\[ (2,-2) \ \in \ \pi_0(\Pic^{\circ}_f)_{t_0} = \{ (d, -d) \in \ZZ^2 : d \in \ZZ \} .\]
			We will later see in \Cref{ex:minimal elliptic surface} that $\check P_{t_0}$ has two connected components, meaning $E$ does not pass through the component $(1, -1)$.
		\end{example}
		
		We conclude with examples where \Cref{thm:main component 1} fails without appropriate hypotheses.
		
		\begin{example} [Main component I]
			Let $C$ be a curve with single node at $p$. Its tangent scheme $\TT_C \to C$ is a commutative group scheme with fiber $(\TT_C)_p = \GG_a^{\times 2}$ over the node but $(\TT_C)_x = \GG_a$ elsewhere. The main component $M$ of $\TT_C$ cannot be a subgroup scheme, since there are two branches of tangent directions at the node $p$; \Cref{thm:main component 1} and \ref{prop:main component over Dedekind} do not apply because $C$ is not unibranch.
		\end{example}
		
		\begin{example} [Main component II]
			Let $C$ be a cuspidal cubic with a cusp at $p = (0,0)$, and consider its tangent scheme $\TT_C \to C$ as in \Cref{ex:group scheme with non-reduced total space}. Let $M \subset \TT_C$ be its main component. At first glance, the family $M \to C$ seems to have $\AA^1$ fibers over every point and hence is a subgroup scheme. This is not the case, as we show its fiber over $p$ is a double line.
			
			The primary decomposition of the ideal $I$ in \Cref{ex:group scheme with non-reduced total space} is \cite{macaulay2}
			\[ (x^3 - y^2, \ x^2u - yv, \ yu - xv, \ xu^2 - v^2) \ \cap \ (x^3, \ xy, \ y^2, \ x^2u - yv) .\]
			The main component $M \subset \AA^4_{(x,y,u,v)}$ is cut out by the first ideal (which is a radical ideal). Setting $x=y=0$, the fiber $M_p \subset \AA^2_{(u,v)}$ is defined by $(v^2)$, i.e., it is a double line. Its equidimensional subscheme $H \subset M$ in \Cref{thm:main component 1} is $H = M$. This shows that the assumption on reduced fibers in \Cref{prop:main component 2} is necessary.
		\end{example}

\section{N\'eron models} \label{sec:Neron models}
	Let $\Delta$ be a Dedekind (=$1$-dimensional regular Noetherian) scheme with function field $K$. Given a smooth variety $N_K \to \Spec K$, the N\'eron model theory in \cite{neron} seeks for its nicest smooth and separated model $N \to \Delta$ of finite type satisfying a certain universal property, the N\'eron mapping property. \cite[Definition 1.1]{hol19} introduced a generalization of this notion to higher dimensional bases $S$. The separatedness assumption is subsequently dropped in \cite{hol-mol-ore-poi23}.
	
	Let $S$ be an algebraic space and $S_0 \subset S$ a scheme-theoretically dense open subspace. Given an algebraic space $N_0 \to S_0$, its \emph{extension} over $S$ is an algebraic space $N \to S$ equipped with a fixed identification $N_{|S_0} \cong N_0$.
	
	\begin{definition} [{N\'eron model; \cite[Definition 6.1]{hol-mol-ore-poi23}}] \label{def:Neron model}
		Let $S$ be an algebraic space, $S_0 \subset S$ its scheme-theoretically dense open subspace, and $N_0 \to S_0$ a smooth algebraic space over it. An extension $N \to S$ of $N_0 \to S_0$ is called a \emph{N\'eron extension} or \emph{N\'eron model} if:
		\begin{enumerate}
			\item $N \to S$ is a smooth algebraic space.
			\item (N\'eron mapping property) For every smooth algebraic space $Z \to S$, the restriction map
			\[ \operatorname{Mor}_S (Z, N) \to \operatorname{Mor}_{S_0} (Z_0, N_0), \qquad u \mapsto u_0 = u_{|S_0} \]
			is bijective.
		\end{enumerate}
		For simplicity, we will often say ``$N$ is N\'eron under the extension of the base $S_0 \subset S$'' or even ``$N$ is N\'eron under $S_0 \subset S$''.
	\end{definition}
	
	\begin{remark} \label{rmk:more assumptions in NMP}
		\begin{enumerate}
			\item N\'eron model is unique up to unique isomorphism once it exists. It is a universal extension retaining only the smoothness of a given algebraic space, so it may fail to be separated, surjective, or of finite type (but it is always locally of finite type because it is smooth).
			
			\item It is enough to check the N\'eron mapping property axiom for quasi-affine smooth morphisms $Z \to S$ of finite type. To prove this, cover an arbitrary smooth algebraic space $Z \to S$ by finite type affine schemes $\Spec A \to S$ for which the N\'eron mapping property hold by assumption. Two \'etale local charts intersect on a finite type quasi-affine scheme $\Spec A \times_Z \Spec B \to S$, so two morphisms $\Spec A \to N$ and $\Spec B \to N$ agreeing over $S_0$ uniquely patch over $S$ by the N\'eron mapping property.
			
			\item If $N_0 \to S_0$ is a smooth algebraic space and $j : S_0 \to S$ is an open immersion, then the N\'eron model represents the lisse-\'etale sheaf $j_* N_0$ on $S$.
			
			\item \Cref{table:comparison of terminology} compares the terminology with others in literature.
			\begin{table}[h]
				\begin{tabular}{|c|c|}
					\hline
					\rule[-1ex]{0pt}{2.5ex} Original terminology & In terms of \Cref{def:Neron model} \\
					\hline
					\hline
					\rule[-1ex]{0pt}{2.5ex} N\'eron model in \cite{neron} & Finite type N\'eron model over $\Delta$ \\
					\hline
					\rule[-1ex]{0pt}{2.5ex} N\'eron lft-model in \cite[\S 10]{neron} & N\'eron model over $\Delta$ \\
					\hline
					\rule[-1ex]{0pt}{2.5ex} N\'eron model in \cite{hol19} & Separated N\'eron model of an abelian scheme \\
					\hline
				\end{tabular}
				
				\caption{Comparison of \Cref{def:Neron model} with others in the literature. Here $\Delta$ is a Dedekind scheme.}
				\label{table:comparison of terminology}
			\end{table}
		\end{enumerate}
	\end{remark}
	
	The base $S$ in \Cref{def:Neron model} is allowed to be an algebraic space for technical flexibility. However, one can reduce to the case when $S$ is a scheme using \Cref{prop:Neron is local in smooth topology}. Throughout this section, we assume $S_0 \subset S$ is a (scheme-theoretically) dense open subspace and $N_0 \to S_0$ is a smooth algebraic space. The subscript $-_0$ always indicates the fiber product $- \times_S S_0$.

	\subsection{Basic properties}
		Basic properties of N\'eron models will be discussed in this subsection. The following two statements are immediate from the universal property of fiber product.
		
		\begin{proposition} [Product]
			Let $N$ and $M$ be smooth algebraic spaces over $S$ that are N\'eron under $S_0 \subset S$. Then $N \times_S M$ is N\'eron under $S_0 \subset S$. \qed
		\end{proposition}
		
		\begin{proposition} [Product over two bases] \label{prop:product of Neron models}
			Let $S$ and $S'$ be smooth algebraic spaces over $T$, and $S_0 \subset S$ and $S'_0 \subset S'$ their dense open subspaces. Let $N \to S$ be N\'eron under $S_0 \subset S$ and similarly $N' \to S'$ as well. Then $N \times_T N' \to S \times_T S'$ is N\'eron under $S_0 \times_T S'_0 \subset S \times_T S'$. \qed
		\end{proposition}
		
		\begin{proposition} [Sequence of base extensions]
			Let $S_0 \subset S_1 \subset S$ be a sequence of dense open subspaces.
			\begin{enumerate}
				\item Let $N$ be N\'eron under $S_1 \subset S$ and $N_1 = N_{|S_1}$ N\'eron under $S_0 \subset S_1$. Then $N$ is N\'eron under $S_0 \subset S$.
				\item Let $N$ be N\'eron under $S_0 \subset S$ and $N_1$ N\'eron under $S_0 \subset S_1$. Then $N$ is N\'eron under $S_1 \subset S$.
				\item Let $N$ be a separated N\'eron model under $S_0 \subset S$. Then $N$ is N\'eron under $S_1 \subset S$ and $N_1 = N_{|S_1}$ is N\'eron under $S_0 \subset S_1$.
			\end{enumerate}
		\end{proposition}
		\begin{proof}
			The first item is clear. To prove the second item, let $Z \to S$ be a smooth algebraic space and consider a sequence of restriction maps $N(Z) \to N(Z_1) \to N(Z_0)$. The second map and the composition map are bijective, so the first map is bijective. In the case of third item, both maps are injective because $Z_0 \subset Z_1 \subset Z$ are dense and $N$ is separated (\Cref{lem:density respects flat base change} and \ref{lem:morphism to separated space is determined by dense subspace}). The composition is bijective, so both are bijective.
		\end{proof}
		
		The following proposition claims that the construction of the N\'eron model is a local property on the base in smooth topology.
		
		\begin{proposition} \label{prop:Neron is local in smooth topology}
			Let $N_0 \to S_0$ be a smooth algebraic space and $S' \to S$ a smooth morphism of algebraic spaces. Consider $S_0' = S' \times_S S_0 \to S_0$ and a cartesian diagram
			\[\begin{tikzcd}
				N_0' \arrow[r] \arrow[d] & N_0 \arrow[d] \\
				S_0' \arrow[r] & S_0
			\end{tikzcd}.\]
			\begin{enumerate}
				\item \textnormal{(Smooth base change)} If $N_0$ has a N\'eron model $N \to S$, then $N_0'$ has a N\'eron model $N' = N \times_S S' \to S'$.
				\item \textnormal{(Smooth descent)} Assume further $S' \to S$ is surjective. If $N_0'$ has a N\'eron model $N' \to S'$, then $N_0$ has a N\'eron model $N \to S$. Moreover, $N \times_S S' \cong N'$.
			\end{enumerate}
		\end{proposition}
		\begin{proof}
			The first item is a routine diagram chasing using the universal property of fiber product. The second item is proved in \cite[Lemma 6.1]{hol19}, but since their definition is slightly different (N\'eron models are assumed to be separated) we provide a proof that does not use separatedness.
			
			\bigskip
			
			Let us first construct an algebraic space $N \to S$ by descending the N\'eron model $N' \to S'$ along $S' \to S$. Consider $S' \to S$ as a covering in smooth topology, and define its powers $S'' = S' \times_S S'$ and $S''' = S' \times_S S' \times_S S'$ with projections $p_i : S'' \to S'$ and $p_{jk} : S''' \to S''$. Write
			\[ N'S' = N' \times_S S' = p_1^* N' ,\qquad S'N' = S' \times_S N' = p_2^* N' ,\qquad \cdots .\]
			To descend, we need to construct an isomorphism (descent datum) of algebraic spaces $\varphi : N'S' \to S'N'$ over $S''$ making the cocyle diagram commute over $S'''$:
			\begin{equation} \label{diag:cocyle condition of descent datum}
			\begin{tikzcd}[column sep=tiny]
				N'S'S' \arrow[rr, "p_{13}^* \varphi"] \arrow[rd, "p_{12}^* \varphi"'] & & S'S'N' \\
				& S'N'S' \arrow[ru, "p_{23}^* \varphi"']
			\end{tikzcd}.
			\end{equation}
			The given algebraic space $N_0 \to S_0$ induces a trivial descent datum $\varphi_0 : N'S'_0 \to S'N'_0$ over $S_0''$. The morphism $S'N' \to S''$ (resp. $N'S' \to S''$) is N\'eron under $S_0'' \subset S''$ since it is a smooth base change of the N\'eron model $N' \to S'$. Hence $\varphi_0$ uniquely extends to an isomorphism $\varphi : N'S' \to S'N'$ over $S''$. The commutativity of \eqref{diag:cocyle condition of descent datum} is from its commutativity over $S_0'''$ together with the N\'eron mapping property of $S'S'N' \to S'''$ under $S_0''' \subset S'''$. This constructs a descent datum $\varphi$ and hence an algebraic space $N \to S$.
			
			To prove the N\'eron mapping property of such $N \to S$, say $Z \to S$ is a smooth algebriace space and $u_0 : Z_0 \to N_0$ is a morphism over $S_0$. Pulling back to $S'$, the morphism $u'_0 : Z'_0 \to N'_0$ uniquely extends to $u' : Z' \to N'$ by the N\'eron mapping property of $N' \to S'$. To (uniquely) descend it to $u : Z \to N$, it is enough to show the two pullbacks
			\[ p_1^*u', p_2^*u' : Z'' \to N'' \]
			agree (now that we know $N \to S$ exists, we may identify $N'' \coloneq N'S' = S'N'$ via $\varphi$). Since $u'_0$ already descends to $u_0$ over $S_0$ by assumption, we have $p_1^*u'_0 = p_2^*u'_0$ as morphisms of algebraic spaces over $S_0''$. Such an equality extends to $p_1^*u' = p_2^*u'$ by the uniqueness part of the N\'eron mapping property of $N'' \to S''$.
		\end{proof}
		
		\begin{remark}
			Here are some remarks on \Cref{prop:Neron is local in smooth topology}.
			\begin{enumerate}
				\item Smooth descent property is likely to fail if we insist everything to be a scheme in \Cref{def:Neron model}, as not every descent datum is effective for schemes (e.g., \cite[\S 6.7]{neron}).
				
				\item The proposition fails for flat and locally finite type morphisms $S' \to S$. In other words, the construction of the N\'eron model is not an fppf local property on the base (\cite[\S 7.2]{neron} or \Cref{ex:counterexample for codimension 1 Neron}). Nevertheless, there are the following variants:
				
				\begin{itemize}
					\item \cite[Proposition 2.4]{liu-tong16} Let $N \to S$ be a smooth algebraic space and $S' \to S$ an fppf morphism. If the base change $N'$ is \emph{separated} N\'eron under $S'_0 \subset S'$, then $N$ is separated N\'eron under $S_0 \subset S$.
					
					\item \cite[Proposition 7.6.6]{neron} Let $f : S' \to S$ be a finite flat morphism and $N'$ a separated N\'eron model under $S_0' \subset S'$. Then $f_* N' \to S$ is a quasi-separated N\'eron model under $S_0 \subset S$.
				\end{itemize}
				
				For the first item, simply notice that the last paragraph of the proof of \Cref{prop:Neron is local in smooth topology} holds for fppf morphisms $S' \to S$, but with the uniqueness of the N\'eron mapping property replaced by the separatedness of the morphism. For the second item, start with \cite[Theorem 1.5]{ols06} that shows the fppf sheaf $f_* N'$ is representable by a quasi-separated and locally finite type algebraic space (need $f$ to be proper flat). The adjunction
				\[ \operatorname{Mor}_S (Z, f_* N') = \operatorname{Mor}_{S'} (Z', N') \]
				for affine schemes $Z$ shows $f_* N' \to S$ is formally smooth (need $f$ to be affine). The adjunction for smooth $Z \to S$ shows the N\'eron mapping property.
			\end{enumerate}
		\end{remark}
		
		\begin{proposition} [Composition]
			Let $f : S \to T$ be \'etale and N\'eron under $T_0 \subset T$ and $g : N \to S$ be N\'eron under $S_0 = S_{|T_0} \subset S$. Then the composition $f \circ g : N \to T$ is N\'eron under $T_0 \subset T$.
		\end{proposition}
		\begin{proof}
			Let $Z \to T$ be a smooth algebraic space and $u_0 : Z_0 \to N_0$ be a morphism. The composition
			\[\begin{tikzcd}[sep=scriptsize]
				v_0 : Z_0 \arrow[r, "u_0"] & N_0 \arrow[r, "g_0"] & S_0
			\end{tikzcd}\]
			uniquely extends to $v : Z \to S$ by the N\'eron mapping property of $S \to T$. Since $Z \to T$ is smooth and $S \to T$ is \'etale, $v : Z \to S$ is smooth (e.g., \cite[Lemma 2.2.9]{neron}). Apply the N\'eron mapping property of $N \to S$ to conclude.
		\end{proof}

	\subsection{Review: abelian schemes over Dedekind schemes}
		Existence of a N\'eron model is a difficult problem. By the pioneering work of N\'eron \cite{neron:original_article} and Raynaud \cite{raynaud:neron}, we know every abelian torsor over a Dedekind scheme has a N\'eron model. Let us quickly review this fact following \cite[\S 1.3]{neron}.
		
		\begin{theorem} [N\'eron, Raynaud] \label{thm:classical Neron theorem}
			Let $\Delta$ be a Dedekind scheme and $\Delta^* \subset \Delta$ its dense open subscheme. Then
			\begin{enumerate}
				\item Every abelian scheme $P^* \to \Delta^*$ admits a quasi-projective N\'eron model $P \to \Delta$.
				\item Every $P^*$-torsor $N^* \to \Delta^*$ admits a quasi-projective N\'eron model $N \to \Delta$.
			\end{enumerate}
		\end{theorem}
		
		The reader should keep in mind that the N\'eron model $N \to \Delta$ in the second item is not necessarily surjective: it is surjective if and only if it is a $P$-torsor. See \cite[Corollary 6.5.3]{neron} and its following paragraph, or \Cref{prop:Neron torsor}. N\'eron's existence proof is constructive:
		
		\begin{theorem} [N\'eron] \label{thm:Neron algorithm}
			Let $\Delta = \Spec R$ be a local Dedekind scheme with function field $K$ and $\Delta^* = \Spec K \subset \Delta$ a dense open subscheme. Then the N\'eron model of an abelian variety $\nu_K : P^* \to \Delta^*$ can be constructed as follows.
			\begin{enumerate}
				\item[\textnormal{Step 1}] Fix an arbitrary flat projective model $\nu^{(1)} : P^{(1)} \to \Delta$ of $\nu_K$ from a regular scheme $P^{(1)}$.
				
				\item[\textnormal{Step 2}] Consider its smooth locus
				\[ \nu^{(2)} : P^{(2)} = P^{(1)} \setminus \Sing(\nu^{(1)}) \longrightarrow \Delta .\]
				This is a weak N\'eron model of $\nu_K$.
				
				\item[\textnormal{Step 3}] Fix a relative canonical divisor of $\nu^{(2)}$ of the form $K_{\nu^{(2)}} = \sum_{i=1}^m a_i F_i$ where $F_1, \cdots, F_m$ are irreducible components of the central fiber $P^{(2)}_0$. Define $a = \min \{ a_1, \cdots, a_m \}$ and set
				\[ \nu^{(3)} : P^{(3)} = P^{(2)} \setminus \bigcup_{a_i > a} F_i \longrightarrow \Delta .\]
				
				\item[\textnormal{Step 4}] $\nu^{(3)}$ admits a birational group law extending the group law of $\nu_K$. Its group scheme solution $\nu : P \to \Delta$ contains $P^{(3)}$ as an open subscheme and is the N\'eron model of $\nu_K$.
			\end{enumerate}
		\end{theorem}
		
		\begin{remark}
			\begin{enumerate}
				\item For the notion of a weak N\'eron model, see \cite[Definition 1.1.1 and 3.5.1]{neron} or its generalization \Cref{def:weak Neron model}. The notion of a birational group law and its group scheme solution can be found in \S 5 in loc. cit. but we will not use them in this article.
				
				\item If $\nu^{(2)}$ in Step 2 has a group structure, then it is already N\'eron. See Theorem 7.1.1 in loc. cit. or its generalization \Cref{thm:weak Neron model of group space}.
			\end{enumerate}
		\end{remark}
		
		\begin{example} [Minimal elliptic surface] \label{ex:minimal elliptic surface}
			Let $f : S \to \Delta$ be a relatively minimal elliptic surface over a local Dedekind scheme $\Delta = \Spec R$ \emph{with a fixed section}. Let $f_K : S_K \to \Delta^* = \Spec K$ be the generic fiber, an elliptic curve over $K$. Since $f$ is projective and $S$ is regular, we can run the construction \ref{thm:Neron algorithm} with $\nu^{(1)} = f$. Fixing a trivial relative canonical divisor $K_f = 0$, Step 2 and 3 yield
			\[ \nu^{(3)} : S' = S \setminus \Sing f \to \Delta .\]
			Step 4 gives a N\'eron model $N \supset S'$ that is potentially larger than $S'$. The valuative criterion of properness for $f$ together with \Cref{prop:section passes through regular point} shows that every \'etale local section of $S \to \Delta$ factors through $S'$, whence the equality $S' = N$. This is essentially the method in \cite[Proposition 1.5.1]{neron}.
			
			Here are two more ways of constructing the N\'eron model that will be discussed in this article. The first way is to start with the Picard space $\Pic_f \to \Delta$. It is a smooth group space by Proposition 8.4.2 in loc. cit. Its neutral component $\Pic^{\circ}_f \to \Delta$ admits a separated quotient $\check P \to \Delta$ by \Cref{prop:smooth group space over Dedekind is aligned}, which turns out to be the N\'eron model of $S_K$. See Theorem~9.5.4 in loc. cit or \Cref{lem:check P_1}.
			
			The second way is to start with the automorphism scheme $\Aut_f \to \Delta$. Its main component $P \to \Delta$ is a smooth, commutative, and quasi-projective group space by \Cref{prop:main component over Dedekind}. We will show it is also a N\'eron model of $S_K$ in \Cref{thm:Neron model of fibration codimension 1}. Therefore, all three constructions $S' = \check P = P$ give us the N\'eron model for elliptic fibrations. For higher-dimensional Lagrangian fibrations, none of the equalities will hold in general.
		\end{example}
		
		\begin{example} \label{ex:counterexample for codimension 1 Neron}
			Let $\Delta = \Spec R$ be the local Dedekind scheme with function field $K$ and $P^* \to \Delta^* = \Spec K$ be an abelian scheme over $K$. It has a quasi-projective N\'eron model $P \to \Delta$ by \Cref{thm:classical Neron theorem}. Assume that $P$ is not semiabelian. By the semistable reduction theorem (e.g., \cite[\S 7.4]{neron}), there exists a finite ramified morphism $\tilde \Delta = \Spec \tilde R \to \Delta$ with function field $\tilde K$ such that the N\'eron model of
			\[ \tilde P^* = P_{\tilde K} \to \tilde \Delta^* = \Spec \tilde K \]
			is a semiabelian scheme $\tilde P^n \to \tilde \Delta$. The base change $\tilde P = P \times_{\Delta} \tilde \Delta \to \tilde \Delta$ is still a smooth group scheme (with a non-semiabelian central fiber), so by the N\'eron mapping property we have a homomorphism
			\[ \tilde P \to \tilde P^n .\]
			This map is an isomorphism over $\tilde \Delta^*$ but contracts the central fiber of $\tilde P$.
		\end{example}

	\subsection{Quasi-projective \'etale schemes} \label{sec:quasi-projective etale scheme}
		In this section, we study the N\'eron model of the simplest smooth morphisms $N_0 \to S_0$: separated, finite type, and \'etale morphisms. Grothendieck's formulation of Zariski's main theorem (Tag~082K) plays a central role in this case. The following is a direct consequence of Zariski's main theorem.
		
		\begin{lemma} \label{lem:quasi-projective etale morphism}
			Let $f : X \to S$ be a morphism between algebraic spaces and assume $S$ is Noetherian. Then the following are equivalent.
			\begin{enumerate}
				\item $f$ is separated, finite type, and \'etale.
				\item $f$ is quasi-affine, finite type, and \'etale.
				\item $f$ is quasi-projective \'etale. \qed
			\end{enumerate}
		\end{lemma}
		
		Assuming $S$ is a scheme, we call such an algebraic space $X$ a quasi-projective \'etale scheme (over $S$). The main result of this subsection is the following.
		
		\begin{theorem} \label{thm:Neron model of quasi-projective etale scheme}
			Let $N_0 \to S_0$ be a quasi-projective \'etale scheme over a normal Noetherian scheme $S$. Then
			\begin{enumerate}
				\item $N_0 \to S_0$ admits a quasi-projective N\'eron model $N \to S$.
				\item Every quasi-projective \'etale extension $M \to S$ of $N_0 \to S_0$ is an open subscheme of $N$.
			\end{enumerate}
		\end{theorem}
		
		\begin{lemma} \label{lem:Neron model of quasi-projective etale scheme}
			Let $N^a \to S$ be a quasi-projective \'etale scheme over a normal Noetherian scheme $S$. Then
			\begin{enumerate}
				\item There exists a maximal quasi-projective \'etale scheme $N$ containing $N^a$ as an open subscheme.
				\item For every dense open subscheme $S_0 \subset S$, $N$ is N\'eron under $S_0 \subset S$.
			\end{enumerate}
		\end{lemma}
		\begin{proof}
			(1) Zariski's main theorem factorizes the given morphism into an open immersion and finite morphism $N^a \hookrightarrow T \to S$. Since $N^a$ is normal, we may assume $T$ is normal and $N^a \hookrightarrow T$ is dense. Define $N \subset T$ by the \'etale locus of the morphism $T \to S$. We claim such $N \to S$ is the desired maximal quasi-projective \'etale scheme containing $N^a$.
			
			Consider an arbitrary quasi-projective \'etale scheme $M \to S$ containing $N^a$ as an open subscheme. We have a commutative diagram over $S$:
			\[\begin{tikzcd}[column sep=tiny]
				& M \arrow[rd, dashed, "j"] \\
				N^a \arrow[rr, hook] \arrow[ru, hook] & & T
			\end{tikzcd}.\]
			We first show the rational map $j : M \dashrightarrow T$ is a morphism. Since $M$ is normal and $T \to S$ is proper, $j$ is defined on codimension $1$ points in $M$ by the valuative criterion. But since $T \to S$ is affine, the undefined locus of $j$ must be of pure codimension $1$ by \cite[Corollaire 20.4.12]{EGAIV4}, forcing $j$ to be defined everywhere.
			
			Now $j : M \to T$ is a quasi-finite (because $M \to S$ is quasi-finite) and birational morphism to a normal scheme $T$. Zariski's main theorem shows that it is an open immersion. Hence it factors through $M \xhookrightarrow{j} N \subset T$, implying the desired maximality of $N$.
			
			\bigskip
			
			(2) Let $Z \to S$ be a smooth algebraic space and $u_0 : Z_0 \to N_0$ a morphism over $S_0$. The rational map $u : Z \dashrightarrow N \subset T$ extends to a morphism $u : Z \to T$ by the same argument as above, so it remains to show $u(Z) \subset N$. Let $s \in S$ be a point and $z \in Z_s$ a closed point. Since the claim is local on $S$, shrinking $S$ \'etale locally we may assume there exists a section $e : S \to Z$ passing through $z$. The composition
			\[\begin{tikzcd}
				f : S \arrow[r, "e"] & Z \arrow[r, "u"] & T
			\end{tikzcd} \]
			is a section of $T$ passing through $u(z) \in T_s$. Over $S_0$, $u_0$ lands in $N_0$ by assumption and hence $f_0 : S_0 \to N_0 \subset T_0$ is a section of an \'etale morphism $N_0 \to S_0$. This means $f_0$ is an open immersion. Since $T$ is normal, by Zariski's main theorem the section $f : S \to T$ is an open immersion over $S$. But it is also closed because it is a section of a separated morphism $T \to S$, proving $S$ is a connected component of $T$. In particular, $T \to S$ is \'etale at $u(z)$, or equivalently $u(z) \in N$.
		\end{proof}
		
		\begin{proof} [Proof of \Cref{thm:Neron model of quasi-projective etale scheme}]
			(1) The composition $N_0 \to S_0 \subset S$ is a quasi-projective \'etale morphism, so \Cref{lem:Neron model of quasi-projective etale scheme} claims there exists a maximal quasi-projective \'etale scheme $\tilde N \to S$ containing $N_0$. Though $\tilde N \to S$ is N\'eron under $S_0 \subset S$, it may be larger than the desired N\'eron model of $N_0$. To make $\tilde N$ smaller, consider a closed subset $C_0 = \tilde N_0 \setminus N_0$ over $S_0$. Take its Zariski closure $C = \overline{C_0}$ in $\tilde N$ and define
			\[ N = \tilde N \setminus C .\]
			By construction, $N \to S$ is an open subscheme of $\tilde N$ and an extension of $N_0 \to S_0$.
			
			To show the N\'eron mapping property of $N$, let $Z \to S$ be a smooth algebraic space and $u_0 : Z_0 \to N_0$ a morphism. Then $u_0$ extends to $u : Z \to \tilde N$ by the N\'eron mapping property of $\tilde N$. It remains to prove $u(Z) \subset N$. Note that $u$ is smooth because $Z \to S$ is smooth and $\tilde N \to S$ is \'etale \cite[Lemma 2.2.9]{neron}. Therefore $u(Z) \subset \tilde N$ is open. Taking the Zariski closure of $C_0 \subset \tilde N_0 \setminus u_0(Z_0)$ in $\tilde N$, we obtain a sequence of inclusions (the superscript $(-)^-$ means Zariski closure)
			\[ C \subset \big( \tilde N_0 \setminus u_0(Z_0) \big)^- = \big( \tilde N_0 \setminus u(Z) \big)^- \subset \big( \tilde N \setminus u(Z) \big)^- = \tilde N \setminus u(Z) .\]
			In other words, $u(Z) \subset \tilde N \setminus C = N$. The N\'eron mapping property of $N$ follows.
			
			\bigskip
			
			(2) Extend the identity map $M_0 \to N_0$ to a morphism $M \to N$ by the N\'eron mapping property. Such a morphism is quasi-finite and birational to a normal Noetherian scheme $N$, so again Zariski's main theorem concludes.
		\end{proof}

	\subsection{Group algebraic spaces and torsors}
		N\'eron models are mostly used for group spaces and torsors. Let us collect some results about them.
		
		\begin{proposition} \label{prop:Neron model of homomorphism}
			Let $G_0 \to S_0$ be a smooth group space and $G \to S$ its N\'eron model. Then
			\begin{enumerate}
				\item $G$ has a unique group structure extending that on $G_0$. If $G_0$ is commutative, then so is $G$.
				\item Let $H \to S$ be a smooth group space. Then the N\'eron extension of a homomorphism $f_0 : H_0 \to G_0$ to $f : H \to G$ is a homomorphism. \qed
			\end{enumerate}
		\end{proposition}
		
		We may therefore say ``$G$ is a group space N\'eron under $S_0 \subset S$'' for simplicity.
		
		\begin{proposition} \label{prop:extension of etale homomorphism}
			Let $G \to S$ be a smooth group space and $H$ a group space N\'eron under $S_1 \subset S$. Assume that $S$ is $\mathrm{(S_2)}$ and $S_1 \subset S$ has codimension $\ge 2$ complement. Then
			\begin{enumerate}
				\item The extension of an \'etale homomorphism $f_1 : G_1 \to H_1$ is an \'etale homomorphism $f : G \to H$.
				\item If $G$ is separated, then the extension of an open immersion homomorphism $f_1 : G_1 \to H_1$ is an open immersion homomorphism $f : G \to H$.
			\end{enumerate}
		\end{proposition}
		\begin{proof}
			(1) The extension $f$ is a homomorphism by \Cref{prop:Neron model of homomorphism}. To show it is \'etale, it is enough to show
			\[ f_* : \Lie G \to \Lie H \]
			is an isomorphism, or equivalently a homomorphism $f^* : C_H \to C_G$ between their conormal sheaves is an isomorphism. Both $C_G$ and $C_H$ are locally free sheaves on an $\mathrm{(S_2)}$ scheme $S$, so Hartogs' theorem says $f^*$ is an isomorphism iff its restriction to $S_1$ is. The codimension $\ge 2$ assumption is necessary; see \Cref{ex:counterexample for codimension 1 Neron}.
			
			\bigskip
			
			(2) If $G$ is separated and $f_1$ is an open immersion, then $\ker f$ is an \'etale and separated group scheme that is trivial over $S_1$. Hence it is trivial by \Cref{lem:etale separated group space is trivial}. Every \'etale monomorphism is an open immersion (Tag~025G).
		\end{proof}
		
		\begin{proposition} \label{prop:torsor is Neron iff group is}
			Let $G \to S$ be a smooth group space and $N \to S$ a $G$-torsor. Then $G$ is N\'eron under $S_0 \subset S$ if and only if $N$ is.
		\end{proposition}
		\begin{proof}
			The question is \'etale local on the base $S$ by \Cref{prop:Neron is local in smooth topology}, so we may assume $N \to S$ has a section after shrinking $S$ \'etale locally. Then $N$ is a trivial $G$-torsor, or $N \cong G$.
		\end{proof}
		
		\begin{proposition} \label{prop:Neron torsor}
			Let $G_0 \to S_0$ be a smooth group space and $N_0 \to S_0$ a $G_0$-torsor.
			\begin{enumerate}
				\item Assume $G \to S$ and $N \to S$ are N\'eron models of $G_0$ and $N_0$. Then $N$ is a $G$-torsor if and only if $N \to S$ is surjective.
				\item Assume $N_0$ has a surjective N\'eron model $N \to S$. Then $G_0$ has a N\'eron model $G \to S$ (and thus $N$ is a $G$-torsor).
			\end{enumerate}
		\end{proposition}
		\begin{proof}
			(1) The implication $(\Longrightarrow)$ is clear, so let us prove $(\Longleftarrow)$. The $G_0$-action $G_0 \times_{S_0} N_0 \to N_0$ uniquely extends to a $G$-action $G \times_S N \to N$ by the N\'eron mapping property. To show this defines a torsor, assume $N \to S$ admits a section $e : S \to N$ after shrinking $S$ \'etale locally. The orbit map $\phi_e : G \to N$ (e.g., see \S \ref{sec:group scheme action}) is an isomorphism over $S_0$ since it trivializes the $G_0$-torsor $N_0$. The isomorphism extends over $S$, as both $G$ and $N$ are N\'eron models. (In fact, this argument proves $N$ is a $G$-torsor over the image of $N \to S$.)
			
			\bigskip
			
			(2) Since $N \to S$ is surjective, we can take a covering $S' = \bigsqcup_{\alpha} U_{\alpha} \to S$ in \'etale topology so that $N' \to S'$ admits a section $e : S' \to N'$. Use its restriction $e_0 : S'_0 \to N'_0$ to trivialize the $G'_0$-torsor $N'_0$ and yield an isomorphism of algebraic spaces $\phi_{e_0} : G'_0 \to N'_0$. Since $N'_0$ has a N\'eron model $N' \to S'$ by \Cref{prop:Neron is local in smooth topology}, $G'_0$ admits the same N\'eron model $G' = N' \to S'$. Descend it to a N\'eron model $G \to S$ of $G_0 \to S_0$ using \Cref{prop:Neron is local in smooth topology}. As the N\'eron model of a group space, $G$ is a group space as well.
		\end{proof}
		
		\begin{remark}
			We do not know whether the following statement is true: let $N_0$ be a $G_0$-torsor. If $G_0$ has a N\'eron model $G$, then so does $N_0$.
		\end{remark}
		
		\begin{proposition} \label{prop:Neron torsor restriction}
			Let $G$ be a group space N\'eron under $S_0 \subset S$. Then the restriction map of pointed sets
			\[ H^1_{\acute et} (S, G) \to H^1_{\acute et} (S_0, G) \]
			is injective. Moreover, $[N_0] \in H^1_{\acute et}(S_0, G)$ is contained in its image if and only if $N_0 \to S_0$ admits a surjective smooth extension $N^a \to S$.
		\end{proposition}
		\begin{proof}
			To prove the injectiveness, let $N \to S$ and $M \to S$ be $G$-torsors whose restrictions over $S_0$ are isomorphic. By \Cref{prop:torsor is Neron iff group is}, they are both N\'eron under $S_0 \subset S$ and hence isomorphic by the uniqueness of the N\'eron model.
			
			For the second claim, consider the given extension $S' \coloneq N^a \to S$ as a covering in smooth topology. Its restriction $S'_0 = N_0 \to S_0$ trivializes the torsor (by definition of torsor):
			\[\begin{tikzcd}
				G'_0 \arrow[r] \arrow[d] & N_0 \arrow[d] \\
				S'_0 \arrow[r] & S_0
			\end{tikzcd}.\]
			Note that $G'_0 \to S'_0$ admits a N\'eron model $G' \to S'$. Such a N\'eron model descends to a N\'eron model $N \to S$ of $N_0 \to S_0$ by \Cref{prop:Neron is local in smooth topology}. The morphism $N \to S$ is surjective by descent, so it is a $G$-torsor by \Cref{prop:Neron torsor}.
		\end{proof}
		
		\begin{lemma} \label{lem:K-torsor is Neron}
			Let $K \to T$ be a group space N\'eron under $T_0 \subset T$, $S \to T$ a smooth algebraic space, and $N \to S$ a $K$-torsor. Then $N \to S$ satisfies a ``strong N\'eron mapping property'' in the following sense:
			\begin{itemize}
				\item For every algebraic space $Z \to S$ that is smooth over $T$, the restriction map
				\[ \operatorname{Mor}_S (Z, N) \to \operatorname{Mor}_{S_0} (Z_0, N_0) ,\qquad u \mapsto u_0 = u_{|S_0} \]
				is bijective.
			\end{itemize}
		\end{lemma}
		\begin{proof}
			Universal property of fiber product constructs a diagram
			\[\begin{tikzcd}
				Z \arrow[r, dashed, "u_0"] \arrow[rd, "\id"'] & N_Z \arrow[r] \arrow[d] & N \arrow[d] \\
				& Z \arrow[r] & S
			\end{tikzcd}.\]
			The base change $N_Z \to Z$ is a torsor under a group space $K_Z \to Z$, which is N\'eron under $Z_0 \subset Z$ since $Z \to T$ is smooth (\Cref{prop:Neron is local in smooth topology}). Hence $N_Z \to Z$ is N\'eron under $Z_0 \subset Z$ by \Cref{prop:torsor is Neron iff group is}, making $u_0 : Z_0 \to N_{Z_0}$ uniquely extends to $u : Z \to N_Z \to N$.
		\end{proof}

	\subsection{Exact sequences}
		The goal of this subsection is to prove the following.
		
		\begin{theorem} \label{thm:les of Neron models}
			Let $1 \to K \to G^a \to H$ be a left exact sequence of smooth group spaces over $S$. If $K$ and $H$ are N\'eron under $S_0 \subset S$, then there exists a commutative diagram of left exact sequences
			\[\begin{tikzcd}
				1 \arrow[r] & K \arrow[r] \arrow[d, phantom, "\parallel"] & G^a \arrow[r] \arrow[d, hook] & H \arrow[d, phantom, "\parallel"] \\
				1 \arrow[r] & K \arrow[r] & G \arrow[r] & H
			\end{tikzcd},\]
			where $G \to S$ is a N\'eron model of $G^a_0 \to S_0$ that contains $G^a$ as an open subgroup space.
		\end{theorem}
		
		Let us state two immediate corollaries of this theorem.
		
		\begin{corollary} \label{cor:Neron model of open subgroup}
			Let $H$ be a group space N\'eron under $S_0 \subset S$. Then every open subgroup space $G_0 \subset H_0$ over $S_0$ admits a N\'eron model $G \subset H$ over $S$, an open subgroup space of $H$.
		\end{corollary}
		\begin{proof}
			The strict neutral component $H^{\circ\circ} \subset H$ is an open subgroup space by \Cref{thm:strict neutral component}. The union $G^a = G_0 \cup H^{\circ\circ} \subset H$ is an open subgroup space of $H$. Apply \Cref{thm:les of Neron models} to $G^a \subset H$.
		\end{proof}
		
		\begin{corollary} \label{cor:Neron under quasi-projective etale homomorphism}
			Let $S$ be a normal Noetherian scheme and $H$ a group space N\'eron under $S_0 \subset S$. Assume that a smooth commutative group space $G^a \to S$ admits a quasi-projective \'etale homomorphism
			\[ f^a : G^a \to H .\]
			Then there exists a N\'eron model $G \to S$ of $G^a_0 \to S_0$ containing $G^a$ as an open subgroup space. Moreover, $f^a$ extends to a quasi-projective \'etale homomorphism $f : G \to H$.
		\end{corollary}
		\begin{proof}
			Set $K^a = \ker f^a$, a quasi-projective, \'etale, and commutative group space over $S$. By \Cref{thm:Neron model of quasi-projective etale scheme}, $K^a_0 \to S_0$ admits a quasi-projective commutative N\'eron model $K \to S$ that contains $K^a$ as an open subgroup space. Define a pushout
			\[ G^b = \big( G^a \times_S K \big) / K^a \qquad\mbox{where}\quad K^a \hookrightarrow G^a \times_S K \quad\mbox{by}\quad x \mapsto (x, -x) ,\]
			a smooth group space over $S$. Such a new group space sits in a commutative diagram
			\[\begin{tikzcd}
				0 \arrow[r] & K^a \arrow[r] \arrow[d, hook] & G^a \arrow[r, "f^a"] \arrow[d, hook] & H \arrow[d, phantom, "\parallel"] \\
				0 \arrow[r] & K \arrow[r] & G^b \arrow[r, "f^b"] & H
			\end{tikzcd}.\]
			Apply \Cref{thm:les of Neron models} to the second row.
		\end{proof}
		
		\begin{remark} \label{rmk:finite type separated implies quasi-projective morphism}
			If $G^a \to S$ is separated and of finite type (e.g., when it has connected fibers by \Cref{thm:smooth group space with connected fibers}), then every \'etale homomorphism $f^a : G^a \to H$ is automatically quasi-projective.
		\end{remark}
		
		\begin{proposition} \label{prop:ses of Neron models}
			Let $G$, $H$, and $K$ be smooth group spaces over $S$.
			\begin{enumerate}
				\item Let $1 \to K \to G \to H$ be a left exact sequence. If $G$ and $H$ are N\'eron under $S_0 \subset S$, then so is $K$.
				\item Let $1 \to K \to G \to H \to 1$ be a short exact sequence. If $K$ and $H$ are N\'eron under $S_0 \subset S$, then so is $G$.
			\end{enumerate}
		\end{proposition}
		\begin{proof}
			(1) is clear. (2) Let $Z \to S$ be a smooth algebraic space. Consider the restriction maps between cohomology long exact sequences (e.g., \cite[Proposition III.4.5]{milne:etale}):
			\[\begin{tikzcd}
				1 \arrow[r] & K(Z) \arrow[r] \arrow[d, phantom, "\parallel"] & G(Z) \arrow[r] \arrow[d] & H(Z) \arrow[r] \arrow[d, phantom, "\parallel"] & H^1 (Z, K) \arrow[d, hook] \\
				1 \arrow[r] & K(Z_0) \arrow[r] & G(Z_0) \arrow[r] & H(Z_0) \arrow[r] & H^1 (Z_0, K)
			\end{tikzcd}.\]
			The first and third vertical arrows are isomorphisms by the N\'eron mapping property. The last vertical arrow is an injective map of pointed sets by \Cref{prop:Neron torsor restriction}. The five lemma-type diagram chasing shows the second vertical arrow is an isomorphism (when group spaces are commutative, this is the actual five lemma).
		\end{proof}
		
		\begin{proof} [Proof of \Cref{thm:les of Neron models}]
			Let us set up notations. Let $f^a : G^a \to H$ be the given homomorphism with kernel $K$ and write $I^a = \im f^a$ for its image, an open subgroup space of $H$. The lower index $-_0$ will indicate the base change $- \times_S S_0$ as usual. For simplicity, denote the restriction of $f^a : G^a \to I^a$ over $S_0$ by $f_0 : G_0 \to I_0$ without the upper index $-^a$. We divide the proof into four steps.
			
			\bigskip
			
			(Step 1: constructing an algebraic space $I$) Let us first enlarge $I^a$ to an open subspace $I \subset H$: this will be the image of the extended homomorphism $f : G \to H$. Consider a collection of \'etale local sections of $H \to S$
			\[ \{ s_{\alpha} : U_{\alpha} \to H : s_{\alpha, 0} = (s_{\alpha})_{|U_{\alpha, 0}} \mbox{ lifts to } t_{\alpha, 0} : U_{\alpha, 0} \to G_0 \} .\]
			The open subgroup space $I^a$ acts on $H$ freely by translation, so the orbit map $\phi_{s_{\alpha}} : I^a_{\alpha} \hookrightarrow H_{\alpha}$ of $s_{\alpha}$ is an open immersion of algebraic spaces over $U_{\alpha}$ (\S \ref{sec:group scheme action}). Their collection defines an \'etale (and hence open) morphism
			\begin{equation} \label{eq:I'}
				I' \coloneq \bigsqcup_{\alpha} I^a_{\alpha} \to H \qquad\mbox{by combining}\quad \begin{tikzcd}[column sep=scriptsize]
					I^a_{\alpha} \arrow[r, hook, "\phi_{s_{\alpha}}"] & H_{\alpha} \arrow[r] & H
				\end{tikzcd}.
			\end{equation}
			Define an open subspace $I \subset H$ by its image. That is, $I \subset H$ is the ``union of $I^a$-orbits for all $G_0$-liftable local sections''. It contains $I^a$ as an open subspace because the identity section $e : S \to H$ is $G_0$-liftable. Since the image of $f_0 : G_0 \to H_0$ was $I_0$, we have $I_{|S_0} = I_0$ and the notation is consistent.
			
			\bigskip
			
			(Step 2: constructing the morphism $f : G \to I$) Recall the definition of $I'$ and the surjective \'etale morphism $I' \to I$ in \eqref{eq:I'}. Collect the pullbacks of $f^a : G^a \to I^a$ to define a morphism
			\[ f' : G' \coloneq \bigsqcup_{\alpha} G^a_{\alpha} \to I' = \bigsqcup_{\alpha} I^a_{\alpha} .\]
			Since $f^a$ is a $K$-torsor, so is $f'$. In particular, $f' : G' \to I'$ is N\'eron under $I'_0 \subset I'$ by \Cref{prop:torsor is Neron iff group is}. On the other hand, recall that $s_{\alpha, 0} : U_{\alpha, 0} \to H_0$ lifts to $t_{\alpha, 0} : U_{\alpha, 0} \to G_0$. Collecting the orbit maps $\phi_{t_{\alpha}} : G^a_{\alpha, 0} \to G_0$ of $t_{\alpha, 0}$ for each $\alpha$, we obtain a cartesian diagram
			\[\begin{tikzcd}
				G'_0 \arrow[r] \arrow[d, "f'_0"] & G_0 \arrow[d, "f_0"] \\
				I'_0 \arrow[r] & I_0
			\end{tikzcd}.\]
			Now \Cref{prop:Neron is local in smooth topology} descends the N\'eron model $f' : G' \to I'$ of $f'_0 : G'_0 \to I'_0$ to the N\'eron model $f : G \to I$ of $f_0 : G_0 \to I_0$. Note that $f$ is a $K$-torsor.
			
			\bigskip
			
			(Step 3: N\'eron mapping property of $G \to S$) At this point, we have an algebraic space $f : G \to I$, which is both a $K$-torsor and an extension of the given homomorphism $f_0 : G_0 \to I_0$. Let $Z \to S$ be a smooth algebraic space and $u_0 : Z_0 \to G_0$ an arbitrary morphism. The following diagram will be useful:
			\[\begin{tikzcd}
				Z \arrow[r, dashed, "u_0"] \arrow[rrdd, bend right=25] & G \arrow[d, "f"] \\
				& I \arrow[r, phantom, "\subset"] & H \arrow[d] \\
				& & S
			\end{tikzcd}.\]
			The composition $Z_0 \to H_0$ uniquely extends to a morphism $Z \to H$ by the N\'eron mapping property of $H$. Let us show that this morphism factors though $I \subset H$. Given an arbitrary closed point $z \in Z_s$ of a fiber over $s \in S$, fix an \'etale local section $U_{\alpha} \to Z$ passing through $z$. Then the composition $s_{\alpha} : U_{\alpha} \to H$ is a section of $H$ that lifts to a section of $G_0$, so by construction in Step 1, it factors through $I$. We have thus constructed a morphism $Z \to I$ and a commutative diagram
			\[\begin{tikzcd}
				Z \arrow[r, dashed, "u_0"] \arrow[rd] & G \arrow[d, "f"] \\
				& I
			\end{tikzcd}.\]
			\Cref{lem:K-torsor is Neron} concludes that $u_0$ uniquely extends to a morphism $u : Z \to G$ (Note: though $f : G \to I$ is a N\'eron model, the morphism $Z \to I$ may not be smooth so this does not follow directly from the N\'eron mapping property of $f$). This proves the N\'eron mapping property of $G \to S$.
			
			\bigskip
			
			(Step 4: completion of the proof) Now that we know the N\'eron mapping property of $G$, its group structure is uniquely extended from that of $G_0 = G^a_0$. The given homomorphism $f_0 : G_0 \to H_0$ extends to a homomorphism $G \to H$ and coincides with the previous morphism $f : G \to I \subset H$ by the uniqueness of the N\'eron mapping property of $H$. This a posteriori shows that $I$ is a subgroup space of $H$. Finally, since $f$ was a $K$-torsor, it has the kernel $K$ as a homomorphism of group spaces.
		\end{proof}

	\subsection{Weak N\'eron models}
		A weaker version of the N\'eron mapping property can be still interesting to some extent. In this subsection, we generalize the notion of a weak N\'eron model in \cite[\S 3.5]{neron} to higher-dimensional bases.

		\subsubsection{Weak N\'eron models}
			The notion of a weak N\'eron model over a Dedekind scheme is introduced in Definition 1.1.1, Definition 3.5.1 and Proposition 3.5.6 in \cite{neron}. Moreover, Theorem 7.1.1 in loc. cit. shows the equivalence between weak N\'eron models and N\'eron models for nice group schemes.
			
			\begin{definition} [Weak N\'eron model] \label{def:weak Neron model}
				Under the same assumption in \Cref{def:Neron model}, an extension $N \to S$ of $N_0 \to S_0$ is called a \emph{weak N\'eron extension} or \emph{weak N\'eron model} if it satisfies the axiom (1) together with:
				\begin{enumerate}
					\item[\textnormal{(2w)}] (Weak N\'eron mapping property) For every \'etale algebraic space $U \to S$, the restriction map
					\[ \operatorname{Mor}_S (U, N) \to \operatorname{Mor}_{S_0} (U_0, N_0), \qquad u \mapsto u_0 = u_{|S_0} \]
					is bijective.
				\end{enumerate}
			\end{definition}
			
			\begin{remark} \label{rmk:more assumptions in weak NMP}
				\begin{enumerate}
					\item It is enough to check the weak N\'eron mapping property axiom for separated, finite type, and \'etale morphisms $U \to S$ (similar to \Cref{rmk:more assumptions in NMP}). When $S$ is Noetherian, we may assume $U \to S$ is quasi-projective \'etale (\Cref{lem:quasi-projective etale morphism}).
					
					\item Weak N\'eron model in the sense of \cite[Definition 3.5.1, Proposition 3.5.6]{neron} is a finite type weak N\'eron model over a Dedekind scheme in the sense of \Cref{def:weak Neron model}. Separatedness automatically follows from \Cref{cor:weak Neron model over Dedekind is separated} and hence it is a scheme by \Cref{thm:group space is scheme}. The weak N\'eron mapping property axiom in \Cref{def:weak Neron model} is called the \emph{extension property for \'etale points (morphisms)} in Definition 1.1.1 of the book. The weak N\'eron mapping property in the book has a different meaning (in Proposition 3.5.3) but we will not use this notion in this article.
				\end{enumerate}
			\end{remark}
			
			The following is the main result in this subsection, generalizing \cite[Theorem 7.1.1]{neron} to higher-dimensional bases. The rest of this subsection will be devoted to its proof.
			
			\begin{theorem} \label{thm:weak Neron model of group space}
				Let $G \to S$ be a smooth group space of finite type over a normal scheme $S$. Then $G$ is weakly N\'eron if and only if it is N\'eron (under $S_0 \subset S$).
			\end{theorem}
			
			The finite type assumption seems to be necessary; see the paragraph after Proposition 10.1.2 in \cite[p.290]{neron}. Let us start with basic properties of weak N\'eron models. Since N\'eron models are weak N\'eron models, these results apply to N\'eron models as well.
			
			To start, \Cref{prop:Neron is local in smooth topology} and its proof holds verbatim but in \'etale topology instead of smooth topology; we will not repeat the statement. Hence the construction of the weak N\'eron model is an \'etale local property on the base.
			
			\begin{lemma} \label{lem:weak Neron model satisfies uniqueness of NMP}
				Let $N$ be weakly N\'eron under $S_0 \subset S$. Then $N$ satisfies the uniqueness part of the N\'eron mapping property.
			\end{lemma}
			\begin{proof}
				Let $Z \to S$ be a smooth algebraic space and $u_0 : Z_0 \to N_0$ a morphism over $S_0$. Assume $u_0$ admits two extensions $u, v : Z \to N$. Let us first prove $u = v$ pointwise on $Z$. Assume on the contrary that there exists a closed point $z \in Z$ with $u(z) \neq v(z)$. Take an \'etale local section $s : U \to Z$ passing through $z$. Then $u \circ s, v \circ s : U \to N$ defines two different extensions of $u_0 \circ s_0 : U_0 \to N_0$, violating the weak N\'eron mapping property.
				
				Pulling back the diagonal map $N \to N \times_S N$ by $(u, v) : Z \to N \times_S N$ defines a locally closed subspace of $Z$ (recall that $N \to S$ is locally separated). Since it is set-theoretically equal to $Z$ by the first paragraph, it is a closed subscheme of $Z$. But it contains a scheme-theoretically dense open subspace $Z_0 \subset Z$, proving it is an equality. This proves $u = v$ as morphisms.
			\end{proof}
			
			\Cref{lem:weak Neron model satisfies uniqueness of NMP} shows that only the existence part of the N\'eron mapping property is problematic for weak N\'eron models. Moreover, the lemma provides an analogue statement of \Cref{lem:morphism to separated space is determined by dense subspace}. That is, if $u : Z \dashrightarrow N$ is a rational map from a smooth algebraic space to a weak N\'eron model that is defined over $S_0$, then there exists a maximal open subspace $\operatorname{Dom}(u) \subset Z$, the \emph{domain of definition} of $u$, such that $u$ is represented by a morphism $u : \Dom(u) \to N$.
			
			\begin{lemma} \label{lem:Neron respects localization}
				Let $N \to S$ be a smooth algebraic space. Then the following are equivalent.
				\begin{enumerate}
					\item $N$ is N\'eron (resp. weakly N\'eron) under $S_0 \subset S$.
					\item For each point $s \in S$, the base change to a local scheme
					\[ N_{(s)} \to S_{(s)} = \Spec \mathcal O_{S, s} \]
					is N\'eron (resp. weakly N\'eron) under $S_{(s)} \times_S S_0 \subset S_{(s)}$.
				\end{enumerate}
			\end{lemma}
			\begin{proof}
				Same proof of \cite[Proposition 1.2.4]{neron} applies. Use \Cref{rmk:more assumptions in NMP} (resp. \ref{rmk:more assumptions in weak NMP}) to assume the test space $Z \to S$ is of finite type. Only the local finite type assumption of $N \to S$ is needed to apply \cite[Theorem 8.8.2]{EGAIV3}.
			\end{proof}
			
			\begin{proposition} \label{prop:aligned subgroup in weak Neron model is separated}
				Let $G$ be a group space weakly N\'eron under $S_0 \subset S$. If $H \subset G$ is an aligned subgroup such that $H_0 \to S_0$ is separated, then $H \to S$ is separated.
			\end{proposition}
			\begin{proof}
				Let $E_H \subset H$ be the scheme-theoretic image of the identity section of $H$, which is \'etale over $S$ by the alignment of $H$. Let $x \in E_H$ be a closed point. Shrinking $S$ \'etale locally, we may choose a section $s : S \to E_H$ passing through $x$. Since $H$ is separated over $S_0$ whence $s$ is trivial over $S_0$, the weak N\'eron mapping property of $G$ forces $s = e$.
			\end{proof}
			
			The following explains the separatedness assumption in the (weak) N\'eron lft-model in the sense of \cite[\S 10]{neron} is redundant, as observed in \cite[\S 1.1]{hol-mol-ore-poi23}.
			
			\begin{corollary} \label{cor:weak Neron model over Dedekind is separated}
				Let $\Delta$ be a Dedekind scheme and $G_0 \to \Delta_0$ a smooth separated group space. Then its weak N\'eron model $G \to \Delta$ is separated.
			\end{corollary}
			\begin{proof}
				Combine \Cref{prop:smooth group space over Dedekind is aligned} and \Cref{prop:aligned subgroup in weak Neron model is separated}.
			\end{proof}
			
			\begin{proof} [Proof of \Cref{thm:weak Neron model of group space}]
				Assume $G \to S$ is weakly N\'eron under $S_0 \subset S$. Let $Z \to S$ be a smooth algebraic space and $u_0 : Z_0 \to G_0$ a morphism over $S_0$. The extension $u : Z \to G$ is unique by \Cref{lem:weak Neron model satisfies uniqueness of NMP}, so we need to show its existence. Taking a covering $\bigsqcup_i Z_i \to Z$ in \'etale topology where each $Z_i \to S$ is a smooth affine scheme of finite type and with irreducible (or empty) fibers, we may reduce to the case where $Z \to S$ is smooth, of finite type, and has irreducible fibers. Moreover, shrinking $S$ \'etale locally, we may assume $S$ is Noetherian and $Z \to S$ has a section $e : S \to Z$.
				
				Let $\eta \in S \setminus S_0$ be a codimension $1$ point and $\Delta = \Spec \mathcal O_{S, \eta} \to S$ a flat trait. The base change $G_{\Delta} \to \Delta$ is a weak N\'eron model of finite type over a Dedekind scheme by \Cref{lem:Neron respects localization}. Its generic fiber $G_K \to \Delta_0 = \Spec K$ is an algebraic group so it is separated. Therefore, $G_{\Delta} \to \Delta$ is separated by \Cref{cor:weak Neron model over Dedekind is separated} and $G_{\Delta}$ is a scheme by \Cref{thm:group space is scheme}. We may thus apply \cite[Theorem 7.1.1]{neron} to conclude that $G_{\Delta} \to \Delta$ is a N\'eron model. This means the rational map $u_{\Delta} : Z_{\Delta} \dashrightarrow G_{\Delta}$ extends to a morphism. By limit argument (Lemma 1.2.5 in loc. cit. or more generally \cite[Theorem 8.8.2]{EGAIV3}), the morphism $u_{\Delta}$ is represented by $u_U : Z_U \to G_U$ for an open subscheme $U \subset S$, i.e., the rational map $u : Z \dashrightarrow G$ is defined over a Zariski neighborhood $U \subset S$ containing $\eta$. Taking union for every codimension $1$ point $\eta \in S$, this shows $u_0$ extends to a morphism $u_1 : Z_1 \to G_1$ over a dense open subscheme $S_1 \subset S$ with codimension $\ge 2$ complement.
				
				The section $e : S \to Z$ composed with $u_0 : Z_0 \to G_0$ defines an $S_0$-section $\varepsilon_0 : S_0 \to G_0$, which extends to an $S$-section $\varepsilon : S \to G$ by the weak N\'eron mapping property. By \Cref{lem:domian of definition contains a section} below, $u : Z \dashrightarrow G$ is defined along $e : S \to Z$. Since $Z \to S$ has connected fibers, $u$ is an \emph{$S$-rational map} in the sense of \S 2.5 in loc. cit. By the theorem of Weil (Theorem 4.4.1 in loc. cit.), $u$ is defined everywhere and extends to a morphism $u : Z \to G$.
			\end{proof}
			
			\begin{lemma} [{\cite[Proposition IX.1.1]{ray:group_schemes}}] \label{lem:domian of definition contains a section}
				Let $Z \to S$ and $N \to S$ be algebraic spaces over an $\mathrm{(S_2)}$ scheme $S$ together with sections $e : S \to Z$ and $\varepsilon : S \to N$. Assume $Z \to S$ is smooth and has connected fibers, and $N \to S$ is either
				\begin{enumerate}[label=\textnormal{(\roman*)}]
					\item separated; or
					\item weakly N\'eron under $S_0 \subset S$.
				\end{enumerate}
				Let $u_0 : Z_0 \to N_0$ be a morphism over $S_0$ and consider it as a rational map $u : Z \dashrightarrow N$ over $S$. Assume
				\begin{enumerate}[label=\textnormal{(\alph*)}]
					\item There exists a dense open subscheme $S_0 \subset S_1 \subset S$ with codimension $\ge 2$ complement such that $e(S_1) \subset \Dom(u)$.
					\item $u(e(S_1)) \subset \varepsilon(S)$.
				\end{enumerate}
				Then $e(S) \subset \Dom(u)$.
			\end{lemma}
			\begin{proof}
				We simply observe that our lemma is a special case of Raynaud's original statement. Since $S$ is $\mathrm{(S_2)}$, the codimension $\ge 2$ complement $S_1$ contains all depth $\le 1$ points, which is Raynaud's original assumption. Raynaud also required the following more technical assumption $\mathrm{(b')}$: for all $z \in e(S) \cap \Dom(u)$, there exists a representative $u^V : V \to N$ of $u$ defined on a Zariski open neighborhood $z \in V \subset \Dom(u)$ such that $u^V (e(S) \cap V) \subset \varepsilon(S)$. Our assumption (b) implies $\mathrm{(b')}$ by \Cref{lem:morphism to separated space is determined by dense subspace} in case (i) and \Cref{lem:weak Neron model satisfies uniqueness of NMP} in case (ii).
			\end{proof}

		\subsubsection{Weak N\'eron models with free actions}
			It is possible to prove a variant of \Cref{thm:weak Neron model of group space} where the assumption on a group space is weakened to that on an algebraic space with free group action. The following is the main result of this subsection.
			
			\begin{theorem} \label{thm:weak Neron model of torsor}
				Let $N \to S$ be a smooth, surjective, and finite type algebraic space over a normal scheme $S$. Assume the following:
				\begin{enumerate}[label=\textnormal{(\roman*)}]
					\item There exists a smooth group space $G^a \to S$ acting on $N$ over $S$.
					\item The action is free over $S$ and transitive over $S_0$.
				\end{enumerate}
				Then $N$ is weakly N\'eron if and only if it is N\'eron (under $S_0 \subset S$).
			\end{theorem}
			\begin{proof}
				Replacing $S$ by its covering \'etale topology (weak N\'eron model version of \Cref{prop:Neron is local in smooth topology}), we may assume $N \to S$ has a section. Such a section trivializes the $G^a_0$-torsor $N_0$, so we may assume $N_0 = G^a_0$ and write $G = N$ for notational simplicity. Our goal is to show $G$ is N\'eron under $S_0 \subset S$. Thanks to \Cref{thm:weak Neron model of group space}, it is enough to show $G \to S$ is a group space.
				
				The identity section $e_0 : S_0 \to G^a_0$ uniquely extends to $e : S \to G$ by the weak N\'eron mapping property of $G$. Let us define the group operation of $G$. Since the $G^a$-action on $G$ is free, every \'etale local section $s : U \to G$ has an open immersion orbit map $\phi_s : G^a_U \hookrightarrow G_U$. Collect sufficiently many \'etale local sections to obtain a covering in \'etale topology
				\[ G' \coloneq \bigsqcup_{\alpha} G^a_{\alpha} \to G ,\qquad g^a \longmapsto g^a.s_{\alpha} .\]
				Set $G'' = G' \times_G G'$ and write $p_1, p_2 : G'' \to G'$ the two projections. Define a morphism
				\[ m' : G' \times_S G \to G ,\qquad (g^a_{\alpha}, h) \longmapsto g^a_{\alpha} . (s_{\alpha} \cdot h) .\]
				Since $G^a_0 = G_0 \to S_0$ is a group space, the morphism $m'_0$ over $S_0$ is nothing but the base change of $m_0 : G_0 \times_S G_0 \to G_0$ under $G' \to G$. That is, $m'$ satisfies the trivial descent condition
				\[ m' \circ (p_1 \times \id) = m' \circ (p_2 \times \id) : G'' \times_S G \to G \]
				over $S_0$. Now $G \to S$ is weakly N\'eron under $S_0 \subset S$, so it satisfies the \emph{uniqueness part} of the N\'eron mapping property by \Cref{lem:weak Neron model satisfies uniqueness of NMP}. This forces $m' \circ (p_1 \times \id) = m' \circ (p_2 \times \id)$ over $S$. Descending $m'$, we obtain a desired multiplication morphism $m : G \times_S G \to G$.
				
				The inverse morphism $G \to G$ is constructed similarly. First, construct $s^{-1}_{\alpha}$ by taking the inverse of $s_{\alpha, 0}$ over $S_0$ and then extending it using the weak N\'eron mapping property. Define a morphism $i' : G' \to G$ by $g^a_{\alpha} \longmapsto s_{\alpha}^{-1} \cdot (g^a_{\alpha})^{-1}$ using the multiplication morphism defined above and finally descend. The group axioms are verified by the uniqueness of the N\'eron mapping property.
			\end{proof}
			
			Let us end this section with the following lemma, which will be used in the future argument. It is a variant of \Cref{prop:extension of etale homomorphism}.
			
			\begin{lemma} \label{lem:Neron extension is an open immersion}
				Let $G \to S$ be a group space N\'eron under $S_0 \subset S$ and $G^a \subset G$ an open subgroup space with $G^a_0 = G_0$. Let $Z \to S$ be a smooth algebraic space extension of $G_0 \to S_0$ that admits a free $G^a$-action. Then
				\begin{enumerate}
					\item The isomorphism $u_0 : Z_0 \to G_0$ extends to a $G^a$-equivariant \'etale morphism $u : Z \to G$.
					\item If we further assume $S$ is normal and $Z \to S$ is separated and of finite type, then the extension $u : Z \to G$ is an open immersion.
				\end{enumerate}
			\end{lemma}
			\begin{proof}
				(1) The equivariance of the extended morphism $u : Z \to G$ is from the N\'eron mapping property. For each closed point $z \in Z$, fix a section $e : S \to Z$ passing through $z$ after shrinking $S$ \'etale locally. The $G^a$-equivariance of $u$ implies the orbit maps of the sections $e$ and $u \circ e$ commute:
				\[\begin{tikzcd}
					G^a \arrow[r, hook, "\phi_e"] \arrow[rd, hook, "\phi_{u \circ e}"'] & Z \arrow[d, "u"] \\
					& G
				\end{tikzcd}.\]
				Since the $G^a$-action on $Z$ and $G$ are free, both $\phi_e$ and $\phi_{u \circ e}$ are open immersions. This proves $u$ is \'etale at $z$.
				
				\bigskip
				
				(2) The previous argument shows $u : Z \to G$ is \'etale and birational. The image $u(Z) \subset G$ is an open subspace since $u$ is \'etale. We claim $u : Z \to u(Z)$ is an isomorphism. Shrinking, we may assume $S$ is normal Noetherian. Now $u : Z \to u(Z)$ is quasi-finite separated and $u(Z)$ is normal Noetherian, so Zariski's main theorem concludes that the birational map $u : Z \to u(Z)$ is an isomorphism.
			\end{proof}

\section{The $\delta$-regular action on a Lagrangian fibration} \label{sec:translation automorphism scheme}
	Let us from now on exclusively focus on the study of Lagrangian fibrations of symplectic varieties. We will work over $\CC$. In this first section, we construct a $\delta$-regular action on the Lagrangian fibration. Results in \S \ref{sec:group spaces} will be used frequently but results in \S \ref{sec:Neron models} will not be used in this section.
	
	Given a morphism $\pi : X \to B$, recall that a (relative) automorphism $f$ of $X$ over $B$ is an automorphism $f : X \to X$ respecting $\pi$, or $\pi = \pi \circ f$.
	
	\begin{definition} \label{def:translation automorphism}
		Let $\pi : X \to B$ be a flat morphism with a dense open subscheme $B_0 \subset B$ over which $\pi$ has abelian variety fibers. A \emph{translation automorphism} of $\pi : X \to B$ is an automorphism $f$ of $X$ over $B$ such that for every point $b \in B_0$, the automorphism of the fiber $f_b : X_b \to X_b$ is a translation of an abelian variety.
	\end{definition}
	
	Here is how we interpret translation automorphisms: if $f$ is a translation automorphism and $X_b$ is a singular fiber, then the automorphism $f_b : X_b \to X_b$ is a limit of translation automorphisms on nearby abelian variety fibers. We can now state the first main theorem of this section.
	
	\begin{theorem} \label{thm:translation automorphism scheme}
		Let $X$ be a smooth symplectic variety and $\pi : X \to B$ a projective Lagrangian fibration to a smooth variety $B$. Assume
		\begin{equation} \label{eq:pi'}
			\pi' : X' \coloneq X \setminus \Sing(\pi) \to B
		\end{equation}
		is surjective. Then there exists a smooth, commutative, and quasi-projective group scheme $P^a \to B$ such that for every \'etale morphism $U \to B$, we have
		\begin{equation} \label{eq:etale local section of translation automorphism scheme}
			P^a (U) = \{ f : X_U \to X_U : f \mbox{ is a translation automorphism over } U \} .
		\end{equation}
		The group scheme $P^a \to B$ is an extension of an abelian scheme $P_0 \to B_0$ and acts faithfully on $X$ over $B$.
	\end{theorem}
	
	\begin{remark}
		$\pi'$ in \eqref{eq:pi'} is surjective if and only if every fiber of $\pi$ has at least one reduced irreducible component. Since $\pi$ is flat, shrinking $B$ to a Zariski open subset $\pi(X')$ will always make $\pi'$ surjective.
	\end{remark}
	
	\begin{definition} \label{def:translation automorphism scheme}
		We call the group scheme $P^a \to B$ in \Cref{thm:translation automorphism scheme} the \emph{translation automorphism scheme} of $\pi$.
	\end{definition}
	
	It is easy to generalize \eqref{eq:etale local section of translation automorphism scheme} to every flat and locally finite type morphism $U \to B$. However, we caution the reader that this is not the case for an arbitrary morphism $U \to B$. For example, if $\{ b \} \hookrightarrow B$ is a closed point over which the fiber $X_b$ singular, then $X_b \to \{ b \}$ has no nearby abelian variety fibers so the notion of translation automorphism does not make sense.
	
	Under a more restrictive assumption that every fiber of $\pi$ is integral, the theorem was proved in \cite[Theorem 2]{ari-fed16}. Their idea was to construct $P^a$ as a locally closed subgroup scheme of the relative automorphism scheme $\Aut_{\pi}$. We will follow their idea of constructing $P^a$ in $\Aut_{\pi}$ but use different techniques to realize this: $P^a$ will be the equidimensional locus of the main component of $\Aut_{\pi}$ (\S \ref{sec:main component}).
	
	\bigskip
	
	The $P^a$-action on $X$ was faithful, so $(X, P^a)$ automatically satisfies the axioms of weak abelian fibration (\S \ref{sec:delta-regularity}). It is also $\delta$-regular.
	
	\begin{theorem} \label{thm:delta-regularity}
		Notations as in \Cref{thm:translation automorphism scheme}. Then the translation automorphism scheme $P^a \to B$ is $\delta$-regular and the Lagrangian fibration $\pi : X \to B$ is a $\delta$-regular abelian fibration in the sense of Ng\^o.
	\end{theorem}
	
	The rest of this section is devoted to the proof of these two theorems.

	\subsection{Smooth subalgebra of the automorphism Lie algebra} \label{sec:Lie algebra of P}
		Let $X$ be a smooth symplectic variety and $\pi : X \to B$ its Lagrangian fibration to a smooth variety $B$. The Lagrangian fibration $\pi$ is flat because it is equidimensional (use \cite[Theorem 1]{mat00} an miracle flatness). Consider the cotangent bundle
		\[ \mathbb T^*_B = \Spec_B \big( \Sym^* T_B \big) \to B .\]
		Over the complement of the discriminant locus $B_0 \subset B$, there is a well-known isomorphism
		\[ \mathbb T^*_{B_0} \to \Lie (\Aut_{\pi_0}) \]
		induced from the symplectic form $\sigma$. The goal of this subsection is to extend this map over the entire base $B$. It is an infinitesimal version of \Cref{thm:translation automorphism scheme} and is a generalization of \cite[Lemma 8.4(i)]{ari-fed16}.
		
		To state the theorem, consider the relative automorphism scheme $\Aut_{\pi} \to B$ of the flat proper morphism $\pi : X \to B$ (\S \ref{sec:automorphism space}), which is a separated group scheme over $B$. Its Lie algebra scheme is a scheme $\Lie (\Aut_{\pi}) \to B$ representing its Lie algebra sheaf (\S \ref{sec:Lie algebra scheme}). The \emph{main component} of $\Lie (\Aut_{\pi})$ is a Zariski closure of the irreducible component containing its zero section (\S \ref{sec:main component}).
		
		\begin{theorem} \label{thm:smooth Lie algebra subscheme}
			The main component of the Lie algebra scheme of $\Aut_{\pi}$ is the cotangent bundle of $B$. More precisely, the symplectic form $\sigma$ induces an injective Lie algebra scheme homomorphism
			\begin{equation} \label{eq:Lie algebra scheme homomorphism}
				\mathbb T^*_B \hookrightarrow \Lie \big( \Aut_{\pi} \big),
			\end{equation}
			which is an isomorphism over $B_0$.
		\end{theorem}
		
		\begin{remark}
			\Cref{thm:smooth Lie algebra subscheme} does not need the surjectivity assumption of $\pi'$ in \eqref{eq:pi'}.
		\end{remark}
		
		To tackle the theorem, we first consider the following commutative diagram of coherent sheaves on $X$ with exact rows and columns. This diagram has appeared several times in literature when $X$ is a compact hyper-K\"ahler manifold (e.g., \cite{mat05}, \cite{sawon08}, and \cite{lehn11, lehn16}) but it also holds in our generality.
		\begin{equation} \label{diag:cotangent}
			\begin{tikzcd}
				& 0 \arrow[d] & 0 \arrow[d] & \tau \arrow[d] \\
				0 \arrow[r] & \pi^* \Omega_B \arrow[r] \arrow[d] & \Omega_X \arrow[r] \arrow[d, "\sigma"] & \Omega_{\pi} \arrow[r] \arrow[d] & 0 \\
				0 \arrow[r] & \Omega_{\pi}^{\vee} \arrow[r] \arrow[d] & T_X \arrow[r] \arrow[d] & \pi^* T_B \arrow[r] \arrow[d] & T^1_{\pi} \arrow[r] & 0 \\
				& \tau & 0 & T^1_{\pi}
			\end{tikzcd}
		\end{equation}
		
		The first row of the diagram is the cotangent sequence of $\pi : X \to B$, which is injective since it is generically injective and $\pi^* \Omega_B$ is locally free. The second row is the coherent sheaf dual of the first using the notation
		\[ \Omega_{\pi}^\vee = \SheafHom_{\mathcal O_X} (\Omega_{\pi}, \mathcal O_X), \qquad T^1_{\pi} = \SheafExt^1_{\mathcal O_X} (\Omega_{\pi}, \mathcal O_X) .\]
		The second column is an isomorphism induced from the symplectic form $\sigma$. Since $\pi$ is Lagrangian, $\sigma : \Omega_X^{\otimes 2} \to \mathcal O_X$ vanishes on the subsheaf $\pi^* \Omega_B^{\otimes 2}$ and induces a homomorphism $\Omega_{\pi} \to \pi^* T_B$ in the third column. Its cokernel is isomorphic to $T^1_{\pi}$ by the snake lemma. Its kernel $\tau$ is the torsion part of $\Omega_{\pi}$ and is interesting in its own right, but we will not need this in our discussion. Finally, the first column is induced by diagram chasing; it is injective and has cokernel $\tau$.
		
		Before using the diagram, let us prove the following lemma.
		
		\begin{lemma} \label{lem:pi is a universal fibration}
			The adjunction map $\mathcal O_B^{\fppf} \to \pi_* \mathcal O_X^{\fppf}$ is an isomorphism of fppf $\mathcal O_B$-modules.\footnote{In other words, $\pi_* \mathcal O_X = \mathcal O_B$ universally, or $\pi$ is cohomologically flat in dimension $0$.}
		\end{lemma}
		\begin{proof}
			Consider an affine test scheme $f : S = \Spec A \to B$ and cartesian diagram
			\[\begin{tikzcd}
				X_S \arrow[r, "f"] \arrow[d, "\pi"] & X \arrow[d, "\pi"] \\
				S \arrow[r, "f"] & B
			\end{tikzcd}.\]
			We need to show a bijection $H^0 (S, \mathcal O_S) = H^0 (X_S, \mathcal O_{X_S})$ or a coherent sheaf isomorphism $\mathcal O_S = \pi_* \mathcal O_{X_S}$. Since $\pi$ is flat, cohomology base change theorem for quasi-coherent sheaves (Tag~08IB) claims the adjunction map
			\[ Lf^* (R\pi_* \mathcal O_X) \to R\pi_* (Lf^* \mathcal O_X) \]
			is a quasi-isomorphism. In the case of Lagrangian fibrations, we further have $R\pi_* \mathcal O_X = \bigoplus_{k=0}^n \Omega_B^k [-k]$ by Matsushita's theorem in \cite[Theorem 1.3]{mat05} and \cite[\S 4]{sch23}. The $0$-th cohomology of the quasi-isomorphism is the desired isomorphism.
		\end{proof}
		
		Consider the third column $\Omega_{\pi} \to \pi^* T_B \to T^1_{\pi} \to 0$ of \eqref{diag:cotangent}. Apply the fppf dual functor $\Hom^{\fppf}(-, \mathcal O_X)$ and subsequently $\pi_*$ to obtain a left exact sequence of fppf $\mathcal O_B$-modules
		\[\begin{tikzcd}
			0 \arrow[r] & \pi_* \Hom^{\fppf} (T_{\pi}^1, \mathcal O_X) \arrow[r] & \Omega_B^{\fppf} \arrow[r, "\sigma"] & \pi_* T_{\pi}^{\fppf}
		\end{tikzcd}.\]
		Notice that we need \Cref{lem:pi is a universal fibration} to apply the projection formula $\Omega_B^{\fppf} = \pi_* \pi^* \Omega_B^{\fppf}$. Since $\Omega_B^{\fppf}$ and $\pi_* T_{\pi}^{\fppf}$ are represented by linear schemes $\TT^*_B$ and $\Lie \Aut_{\pi}$ (\eqref{eq:tangent scheme} and \Cref{thm:automorphism space}), the latter homomorphism $\sigma$ is the desired homomorphism in \Cref{thm:smooth Lie algebra subscheme}: the content of the theorem is the vanishing of the fppf $\mathcal O_B$-module $\pi_* \Hom^{\fppf}(T_{\pi}^1, \mathcal O_X)$. This question is equivalent to the vanishing of a higher direct image coherent sheaf by the following.
		
		\begin{lemma}
			The fppf $\mathcal O_B$-module $\pi_* \Hom^{\fppf} (T_{\pi}^1, \mathcal O_X)$ vanishes if and only if the coherent sheaf $R^n \pi_* T^1_{\pi}$ on $B$ vanishes.
		\end{lemma}
		\begin{proof}
			The sheaf $\pi_* \Hom^{\fppf} (T_{\pi}^1, \mathcal O_X)$, as the kernel of a linear scheme homomorphism \eqref{eq:Lie algebra scheme homomorphism}, is representable by a linear scheme over $B$ (\Cref{prop:linear scheme}) and thus its vanishing can be checked pointwise over closed points of $B$ by Nakayama's lemma. Over a closed point $b \in B$, we have
			\[ \big( \pi_* \Hom^{\fppf} (T^1_{\pi}, \mathcal O_X) \big) (b) = \Hom^{\fppf} (T^1_{\pi}, \mathcal O_X) (X_b) = \Hom_{\mathcal O_{X_b}} (T^1_{\pi|b}, \mathcal O_{X_b}) .\]
			
			Since $\pi_b : X_b \to \{ b \}$ is an lci scheme with trivial canonical bundle (by adjunction formula), its dualizing complex is isomorphic to $\mathcal O_{X_b}[n]$. Combining the Serre duality, cohomology base change (Tag~08IB), and convergence of Grothendieck spectral sequences, we have
			\[ \Hom_{\mathcal O_{X_b}} (T^1_{\pi|b}, \mathcal O_{X_b}) = H^n (X_b, T^1_{\pi|b})^{\vee} = \big( R^n\pi_* T^1_{\pi} \otimes_{\mathcal O_B} k(b) \big)^{\vee} .\]
			Again by Nakayama's lemma, $R^n\pi_* T^1_{\pi} \otimes_{\mathcal O_B} k(b) = 0$ for all closed points $b \in B$ if and only if $R^n\pi_* T^1_{\pi} = 0$. This completes the proof.
		\end{proof}
		
		Therefore, \Cref{thm:smooth Lie algebra subscheme} boils down to showing the vanishing of the higher direct image coherent sheaf $R^n\pi_* T^1_{\pi}$. We prove this using the symmetry property of the decomposition theorem in \cite{sch23}.

		\subsubsection{Vanishing of the higher direct image}
			Our goal is reduced to showing the vanishing of the coherent sheaf $R^n\pi_* T^1_{\pi}$. We use Hodge modules to achieve this.
			
			\begin{proposition} \label{prop:vanishing of higher direct image sheaf}
				Let $T^1_{\pi} = \SheafExt^1_{\mathcal O_X} (\Omega_{\pi}, \mathcal O_X)$ be a coherent sheaf on $X$. Then
				\[ R^n\pi_* T^1_{\pi} = 0 .\]
			\end{proposition}
			
			Denote the (coherent sheaf) Verdier dual functors of the smooth varieties $X$ and $B$ by
			\[ \DD_X = R\SheafHom_{\mathcal O_X} (-, \omega_X[2n]), \qquad \DD_B = R\SheafHom_{\mathcal O_B} (-, \omega_B[n]) .\]
			The symplectic structure on $X$ gives an isomorphism $\omega_X \cong \mathcal O_X$.
			
			\begin{lemma}
				The desired higher direct image sheaf lies in a four-term exact sequence of coherent sheaves on $B$
				\begin{equation} \label{eq:four term exact sequence}
					\begin{multlined}
						\mathcal H^{n-1} \Big[ \DD_B \big( R\pi_* \Omega_X [2n-1] \big) \Big]
						\longrightarrow \mathcal H^{n} \Big[ \DD_B (R\pi_* \mathcal O_X [2n] \otimes \Omega_B ) \Big] \qquad \\
						\longrightarrow R^n\pi_* T^1_{\pi}
						\longrightarrow \mathcal H^{n} \Big[ \DD_B \big( R\pi_* \Omega_X[2n-1] \big) \Big] .
					\end{multlined}
				\end{equation}
			\end{lemma}
			\begin{proof}
				We first claim there exists an isomorphism
				\[ R^n\pi_* T^1_{\pi} \ \cong \ \mathcal H^{-n+1} \big[ \DD_B R\pi_* \Omega_{\pi} \big] .\]
				To prove this, recall from \eqref{diag:cotangent} that $\Omega_{\pi}$ had a length $1$ locally free resolution $\pi^* \Omega_B \to \Omega_X$. Hence from the convergence of the spectral sequence we have
				\begin{align*}
					R^n\pi_* T^1_{\pi} &= R^n \pi_* \SheafExt^1_{\mathcal O_X} (\Omega_{\pi}, \mathcal O_X) \\
					&= \mathcal H^{n+1} \big[ R\pi_* R\SheafHom_{\mathcal O_X} (\Omega_{\pi}, \mathcal O_X) \big] \quad & \mbox{(spectral sequence)}\\
					&= \mathcal H^{-n+1} \big[ R\pi_* \DD_X \Omega_{\pi} \big] \quad & (\omega_X^{\bullet} = \mathcal O_X [2n]) \\
					&= \mathcal H^{-n+1} \big[ \DD_B R\pi_* \Omega_{\pi} \big] \quad & \mbox{(Grothendieck duality)}.
				\end{align*}
				
				Again, start from the short exact sequence $0 \to \pi^* \Omega_B \to \Omega_X \to \Omega_{\pi} \to 0$. Applying $\DD_B \circ R\pi_*$, we obtain an exact triangle
				\[\begin{tikzcd}
					\DD_B R\pi_* \Omega_{\pi} \arrow[r] & \DD_B R\pi_* \Omega_X \arrow[r] & \DD_B \big( R\pi_* \mathcal O_X \otimes \Omega_B \big) \arrow[r, "+1"] & { }
				\end{tikzcd}.\]
				The desired sequence \eqref{eq:four term exact sequence} is its cohomology long exact sequence.
			\end{proof}
			
			To understand the exact sequence \eqref{eq:four term exact sequence}, let us use some facts on the decomposition theorem for $\pi : X \to B$ in \cite{shen-yin22} and \cite{sch23}. We do not need the full strength of their results, so let us summarize only the necessary parts following \cite[\S 1.2]{sch23}. The decomposition theorem for $\pi : X \to B$ is an isomorphism in the bounded derived category of Hodge modules on $B$
			\begin{equation} \label{eq:decomposition theorem}
				R\pi_* \Big( \QQ_X(n)[2n] \Big) \cong \bigoplus_{i=-n}^n P_i[-i] ,
			\end{equation}
			where $P_i$ is a Hodge module of pure weight $i$ for $i = -n, \cdots, n$. The bound on the index $i$ is determined by the defect of semismallness $\dim X \times_B X - \dim X = n$. The weight $i$ Hodge module $P_i$ contains a filtered (regular holonomic) $D$-module $(\mathcal P_i, F_{\bullet} \mathcal P_i)$ in its datum, so we may consider its associated filtered de Rham complex
			\[ F_k \operatorname{DR} (\mathcal P_i) = \Big[ F_{k-n} \mathcal P_i \otimes \bigwedge^n T_B \longrightarrow \ \cdots \ \longrightarrow F_{k-1} \mathcal P_i \otimes T_B \longrightarrow F_k \mathcal P_i \Big] ,\]
			that lives in $D^b_{coh}(B)$ and is concentrated in degree $[-n, 0]$.
			
			\begin{definition}
				Following \cite[\S 1.2]{sch23}, we denote the $(-k)$-th graded associated piece of the filtered de Rham complex $(\operatorname{DR} (\mathcal P_i), \, F_{\bullet} \operatorname{DR} (\mathcal P_i))$, up to degree shift, by
				\[ G_{i,k} = \operatorname{gr}_{-k}^F \operatorname{DR} (\mathcal P_i) [-i] .\]
			\end{definition}
			
			The isomorphism \eqref{eq:decomposition theorem} induces an isomorphism of filtered $D$-modules. Applying the functor $\operatorname{gr}^F_{-k} \operatorname{DR}$ to this filtered $D$-module isomorphism yields an isomorphism in $D^b_{coh}(B)$:
			\begin{equation} \label{eq:direct image of Omega}
				R\pi_* \Omega_X^{n+k} [n-k] \cong \bigoplus_{i=-n}^n G_{i,k} \qquad \mbox{for} \quad k = -n, \cdots, n .
			\end{equation}
			The following is originally proved in \cite[Theorem 1.3]{mat05} and later revisited in \cite[\S 4]{sch23}.
			
			\begin{theorem} [Matsushita, Schnell] \label{thm:matsushita}
				We have isomorphisms
				\[ G_{i,n} \cong \begin{cases}
					\Omega_B^i[-i] \quad & \mbox{for} \quad i \ge 0 \\
					0 & \mbox{for} \quad i < 0
				\end{cases} .\]
				In other words, the isomorphism \eqref{eq:direct image of Omega} for $k = n$ is
				\[ R\pi_* \omega_X \cong \bigoplus_{i=0}^n \Omega_B^i[-i] .\]
			\end{theorem}
			
			We can now start calculating the sequence \eqref{eq:four term exact sequence} of interest.
			
			\begin{lemma} \label{lem:four term exact sequence}
				The sequence \eqref{eq:four term exact sequence} up to isomorphism is computed as follows.
				\begin{enumerate}
					\item The first term of the sequence is
					\[ \mathcal H^{n-1} \Big( G_{n-1, n-1} \oplus G_{n, n-1} \Big) \cong \mathcal H^{n-1} \Big( G_{n-1, n-1} \Big) \oplus \Omega_B^{n-1} .\]
					
					\item The second term of the sequence is
					\[ \mathcal H^n \Big( G_{n, n} \otimes T_B \Big) \cong \Omega_B^{n-1} .\]
					
					\item The homomorphism between the first two terms, with respect to the above isomorphisms, contains the identity map $\Omega_B^{n-1} \to \Omega_B^{n-1}$. In particular, it is surjective.
					
					\item The last term of the sequence vanishes.
				\end{enumerate}
			\end{lemma}
			
			\begin{proof} [Proof of \Cref{prop:vanishing of higher direct image sheaf}]
				From \Cref{lem:four term exact sequence}, the first homomorphism in \eqref{eq:four term exact sequence} is surjective and the fourth term vanishes. This proves the third term $R^n \pi_* T^1_{\pi}$ vanishes as well (and completes the proof of \Cref{thm:smooth Lie algebra subscheme}).
			\end{proof}
			
			\begin{proof} [Proof of \Cref{lem:four term exact sequence}]
				(1) and (4) Consider a sequence of isomorphisms
				\begin{align*}
					\mathcal H^p \Big[ \DD_B \big( R\pi_* \Omega_X[2n-1] \big) \Big] &\cong \mathcal H^p \Big[ \bigoplus_{i=-n}^n \DD_B G_{i, -n+1} \Big] & \mbox{(by \eqref{eq:direct image of Omega})} \\
					&\cong \mathcal H^p \bigoplus_{i=-n}^n G_{i,n-1} & \mbox{(by \cite[(43.1)]{sch23})} .
				\end{align*}
				By \cite[Lemma 44]{sch23}, the complex $G_{i,k}$ is concentrated in degrees $\le \min (i, k, i+k)$, so none of the $G_{i, n-1}$'s can have a nontrivial $n$-th cohomology. This proves (4). For (1), set $p = n-1$ so that the only terms with potentially nonvanishing $(n-1)$-th cohomology are $G_{n-1, n-1}$ and $G_{n, n-1}$. The latter cohomology is
				\[ \mathcal H^{n-1} G_{n, n-1} \cong \Omega^{n-1}_B \]
				from either an explicit computation as in (3) below, or more conveniently from \Cref{thm:matsushita} together with the main result of Schnell in \cite[Theorem 24]{sch23}.
				
				\bigskip
				
				(2) Recall again from \eqref{eq:direct image of Omega} that we have
				\[ R\pi_* \mathcal O_X [2n] \cong \bigoplus_{i=-n}^n G_{i,-n}, \qquad G_{i,-n} \cong \Omega_B^{n+i}[n+i] .\tag{\eqref{eq:direct image of Omega} \& Matsushita}\]
				We therefore get
				\begin{align*}
					& \mathcal H^n R\SheafHom (R\pi_* \mathcal O_X[2n] \otimes \Omega_B, \omega_B^{\bullet}) \\
					\cong \ & \mathcal H^n R\SheafHom (\bigoplus_{i=-n}^n G_{i,-n} \otimes \Omega_B, \omega_B^{\bullet})
					= \mathcal H^n \bigoplus_{i=-n}^n R\SheafHom (G_{i,-n} \otimes \Omega_B, \omega_B^{\bullet}) \\
					\cong \ & \mathcal H^n \bigoplus_{i=-n}^n R\SheafHom (\Omega_B^{n+i}, \Omega_B^{n-1})[-i]
					= R \SheafHom (\mathcal O_B, \Omega_B^{n-1}) = \Omega_B^{n-1} .
				\end{align*}
				
				\smallskip
				
				(3) To analyze the homomorphism, we need to be more careful about applying the isomosphisms above. Recall how we have defined the homomorphism. We considered a map $\pi^* \Omega_B \to \Omega_X$ and applied the functor $R\pi_*$ to get
				\begin{equation} \label{eq:derived pushforward map}
					R\pi_* \mathcal O_X \otimes \Omega_B \to R\pi_* \Omega_X .
				\end{equation}
				This homomorphism is the $\Omega_B$-contraction homomorphism from the filtered $D$-module structure on $\pi_+ (\omega_X, F_{\bullet} \omega_X)$ (e.g., \cite[(14)]{sch23}). As a result, the $\Omega_B$-contraction respects the decomposition \eqref{eq:direct image of Omega} and acts componentwise:
				\[ G_{i,-n} \otimes \Omega_B \to G_{i, -n+1}[1] .\]
				The map \eqref{eq:derived pushforward map} can be decomposed as a direct sum of these maps for $i = -n, \cdots, n$ together with an isomorphism \eqref{eq:direct image of Omega}.
				
				If we set $i = -n$, then we can explicitly calculate this map $G_{-n, -n} \otimes \Omega_B \to G_{-n, -n+1}[1]$. Since the fibers of $\pi$ are all connected, standard argument in decomposition theorem implies that the lowest weight part of the decomposition theorem $R\pi_* \QQ_X(n)[2n] \cong \bigoplus_{i=-n}^n P_i[-i]$ is
				\[ P_{-n} = \QQ_B(n)[n] ,\]
				the weight $-n$ Hodge module associated to the trivial VHS $\QQ_B(n)$ on $B$. Its filtered $D$-module part $(\mathcal P_{-n}, F_{\bullet} \mathcal P_{-n})$ is described by
				\[ \mathcal P_{-n} = \omega_B ,\qquad \cdots = F_{-2} = F_{-1} = 0 ,\qquad F_0 = F_1 = \cdots = \omega_B .\]
				Hence one can explicitly compute the gradings on the de Rham complex $\operatorname{DR}(\mathcal P_{-n})$ as
				\[ G_{-n, -n} = \mathcal O_B[2n], \qquad G_{-n, -n+1} = \Omega_B[2n-1] .\]
				The homomorpihsm $G_{-n,-n} \otimes \Omega_B \to G_{-n, -n+1}[1]$ is the contraction homomorphism coming from the $D$-module structure, which in this case is an identity $\mathcal O_B[2n] \otimes \Omega_B \to \Omega_B[2n]$. This completes the proof of \Cref{lem:four term exact sequence}.
			\end{proof}

	\subsection{Smooth subgroup scheme of the automorphism scheme}
		We can now upgrade the infinitesimal statement \Cref{thm:smooth Lie algebra subscheme} to our desired \Cref{thm:translation automorphism scheme}.
		
		\begin{proof} [Proof of \Cref{thm:translation automorphism scheme}]
			Let us write $\Aut$ for $\Aut_{\pi}$ to ease notation. \Cref{thm:smooth Lie algebra subscheme} showed the main component of the Lie algebra scheme $\Lie \Aut$ is smooth. We can thus use \Cref{thm:main component 1} to yield the largest universally equidimensional and locally closed subscheme $P^a \subset \Aut$. To show $P^a \to B$ is a smooth subgroup scheme, it is enough to show that it has reduced fibers by \Cref{prop:main component 2}.
			
			Since the question is \'etale local on $B$ and $\pi'$ is surjective, we may assume $\pi : X \to B$ has a section $e : B \to X$ (\Cref{prop:etale local section}). Such $e$ is a section of $X'$ because $X$ and $B$ are smooth (\Cref{prop:section passes through regular point}). Note that $\Aut$ acts on $X'$, so we can consider an orbit map (e.g., \S \ref{sec:group scheme action})
			\[ \phi = \phi_e : \Aut \to X' ,\qquad f \mapsto f.e .\]
			We claim its restriction
			\begin{equation} \label{eq:orbit map}
				\phi_{|P^a} : P^a \subset \Aut \to X'
			\end{equation}
			is an open immersion. This will prove $P^a_b$ is reduced for every point $b \in B$ and hence completes the proof of the theorem.
			
			Let us first prove the following lemma.
			
			\begin{lemma}
				Let $\Delta \to B$ be a non-degenerate trait and $\pi'_{\Delta} : X'_{\Delta} \to \Delta$ the base change of $\pi'$. Let $e : \Delta \to X'_{\Delta}$ be a section and $\phi : \Aut_{\Delta} \to X'_{\Delta}$ its orbit map. Then $\phi$ restricted to the main component
				\begin{equation} \label{eq:orbit map small}
					\phi_{|H^{\Delta}} : H^{\Delta} \subset \Aut_{\Delta} \to X_{\Delta}'
				\end{equation}
				is an open immersion.
			\end{lemma}
			\begin{proof}
				By \Cref{prop:main component over Dedekind}, the main component $H^{\Delta}$ of $\Aut_{\Delta}$ is a smooth and closed subgroup scheme over $\Delta$ whose Lie algebra is
				\[ \Lie H^{\Delta} = (\TT^*_B)_{|\Delta} \]
				by \Cref{thm:smooth Lie algebra subscheme}. Consider the induced $H^{\Delta}$-action on $X'_{\Delta}$. To prove its orbit map $H^{\Delta} \to X_{\Delta}'$ is an open immersion, it is enough to show the $H^{\Delta}$-action on $X'_{\Delta}$ is free, or equivalently the stabilizer scheme $\St'_{\Delta} \to X'_{\Delta}$ is trivial by \Cref{prop:trivial stabilizer is free}.
				
				Recall that \eqref{eq:Lie algebra scheme homomorphism} was an fppf $\mathcal O_B$-module homomorphim $\Omega^{\fppf}_B \hookrightarrow \pi_* T_{\pi}^{\fppf}$, which was obtained by first taking the right exact sequence $\Omega_{\pi} \to \pi^* T_B \to T^1_{\pi} \to 0$ in \eqref{diag:cotangent}, taking its fppf dual $0 \to \Hom^{\fppf} (T^1_{\pi}, \mathcal O_B) \to \pi^* \Omega^{\fppf}_B \to T_{\pi}^{\fppf}$, and then subsequently taking $\pi_*$. Note that $T^1_{\pi}$ is supported outside of $X'$ (it is supported on $\Sing(\pi) = X \setminus X'$, thanks to \Cref{prop:fitting2}). Now pullback the original right exact sequence first to $X'_{\Delta}$ and obtain
				\[ \Omega_{\pi'_{\Delta}} \to (\pi'_{\Delta})^* T_B \to 0 \to 0 .\]
				Taking its fppf dual, we get a left exact sequence
				\[ 0 \to 0 \to (\pi'_{\Delta})^* \big( (\Omega_B^{\fppf})_{|\Delta} \big) \to T^{\fppf}_{\pi'_{\Delta}} ,\]
				which models the infinitesimal $H^{\Delta}$-action on $X'_{\Delta}$ because $\Lie H^{\Delta} = (\Omega^{\fppf}_B)_{|\Delta}$. In other words, this sequence is \Cref{prop:infinitesimal action} and hence $\Lie \St'_{\Delta} = 0$.
			\end{proof}
			
			We can now complete the proof of the theorem. To show \eqref{eq:orbit map} is an open immersion, it is enough to show it is injective on closed points by \Cref{prop:Zariski main theorem} (Zariski's main theorem). Let $b \in B$ be a closed point and $\Delta \to B$ a non-degenerate \emph{quasi-finite} (or \emph{unramified}) trait through $b$. The quasi-finiteness says the closed point of $\Delta$ has residue field $\CC$. Since \eqref{eq:orbit map small} is an open immersion, $\phi$ induces an injective map on closed points $(H^{\Delta})_b(\CC) \hookrightarrow X'_b(\CC)$. Since $(P^a_{\Delta})_{\red} \subset H^{\Delta}$ by \Cref{prop:main component 2}(2a), $\phi$ induces an injective map on closed points
			\[ \phi : P^a_b(\CC) \subset (H^{\Delta})_b(\CC) \hookrightarrow X'_b(\CC) .\]
			Vary $b \in B$ to conclude that $\phi : P^a(\CC) \to X'(\CC)$ is injective.
		\end{proof}

	\subsection{$\delta$-regularity of the group scheme}
		We show the $\delta$-regularity of $P^a \to B$ in this subsection. For the notion of $\delta$-regularity and weak abelian fibrations, see \S \ref{sec:delta-regularity}. We have already seen that the coherent sheaf
		\[ T^1_{\pi} = \SheafExt^1_{\mathcal O_X} (\Omega_{\pi}, \mathcal O_X) \]
		plays a particular role in the proof of \Cref{thm:translation automorphism scheme} and \ref{thm:smooth Lie algebra subscheme}. A posteriori, we can endow it with the following geometric meaning.
		
		\begin{proposition} \label{prop:Lie algebra of stabilizer scheme}
			Let $\St \to X$ be the stabilizer group scheme of the $P^a$-action on $X$ over $B$. Then $T^1_{\pi}$ is the conormal sheaf of the identity section $X \to \St$, i.e., $\Lie \St$ is a linear scheme associated to the coherent sheaf $T^1_{\pi}$.
		\end{proposition}
		\begin{proof}
			Take the third column $\Omega_{\pi} \to \pi^* T_B \to T^1_{\pi} \to 0$ of \eqref{diag:cotangent}. Its fppf dual is a left exact sequence of linear schemes
			\[ 0 \to \VV_{T^1_{\pi}} \to \pi^* \TT^*_B \to \TT_{\pi} .\]
			The translation automorphism scheme $P^a \subset \Aut_{\pi}$ satisfies
			\[ \Lie P^a = \TT^*_B \ \subset \ \Lie \Aut_{\pi} = \pi_* \TT_{\pi} \]
			via the inclusion in \Cref{thm:smooth Lie algebra subscheme}, which was again induced from the adjoint map of $\sigma : \pi^* \TT^*_B \to \TT_{\pi}$ by the construction in \S \ref{sec:Lie algebra of P}. This means the last homomorphism above is precisely $\pi^* (\Lie P^a) \to \TT_{\pi}$ induced from the $P^a$-action. \Cref{prop:infinitesimal action} concludes $\Lie \St = \VV_{T^1_{\pi}}$.
		\end{proof}
		
		We call the coherent sheaf $T^1_{\pi}$ the \emph{stabilizer sheaf} of the $P^a$-action on $X$. It contains infinitesimal information of the stabilizers of the $P^a$-action, e.g., dimension of the stabilizers.
		
		\begin{proposition} \label{prop:action is free on X'}
			The $P^a$-action on $X'$ is free.
		\end{proposition}
		\begin{proof}
			Since the (Fitting) support of $T^1_{\pi}$ is $\Sing(\pi)$, the Lie algebra $\Lie (\St) \to X$ vanishes on $X'$. Use \Cref{prop:trivial stabilizer is free} to conclude.
		\end{proof}
		
		\begin{proof} [Proof of \Cref{thm:delta-regularity}]
			The following is essentially \cite[\S 2]{ngo11} and \cite[Proposition 8.9]{ari-fed16}. The group scheme $P^a$ is a subgroup scheme of $\Aut_{\pi}$ by construction. Hence $P^a$ acts on $X$ faithfully and $(X, P^a)$ is automatically a weak abelian fibration (e.g., \Cref{ex:weak abelian fibration}). We may thus use \Cref{prop:delta-regularity} to check the $\delta$-regularity of $P^a$. By \Cref{prop:Lie algebra of stabilizer scheme}, it is enough to show
			\[ \codim \pi (\Fitt_{i-1} T^1_{\pi}) \ge i .\]
			By the diagram \eqref{diag:cotangent} and \Cref{prop:fitting2}, we have an identity
			\[ \Fitt_{i-1} T^1_{\pi} = \Fitt_{n+i-1} \Omega_{\pi} ,\]
			so the claim follows from \Cref{lem:image of the fitting subscheme}.
		\end{proof}

\section{Codimension $1$ behavior of a Lagrangian fibration} \label{sec:Neron model of fibration codimension 1}
	Let again $\pi : X \to B$ be a Lagrangian fibration of a smooth symplectic variety over a smooth variety $B$. To use the results in \S \ref{sec:translation automorphism scheme}, we always assume $\pi' : X' \to B$ is surjective. Consider the $\delta$-function $\delta : B \to \ZZ$ of the $\delta$-regular group scheme $P^a$ as in \eqref{eq:delta-function} and define a decreasing sequence of Zariski closed subsets
	\[ D_i = \{ b \in B : \delta(b) \ge i \} \ \subset B \qquad\mbox{for}\quad i = 1, \cdots, n .\]
	Note that $D_1$ is precisely the discriminant locus of $\pi$, so every irreducible component of $D_1$ has codimension $1$ by \cite[Proposition 3.1]{hwang-ogu09}. The $\delta$-regularity result in \Cref{thm:delta-regularity} is the inequality $\codim D_i \ge i$. We will more often use their complements
	\begin{equation} \label{eq:delta-loci}
		B_i = \{ b \in B : \delta(b) \le i \} = B \setminus D_{i+1} \qquad\mbox{for}\quad i = 0, \cdots, n ,
	\end{equation}
	which is an increasing sequence of Zariski open subsets.
	
	In this section, we focus on the behavior of $\pi$ over the first nontrivial $\delta$-locus $B_1$. We somewhat vaguely call it a study of the \emph{codimension $1$ behavior} of $\pi$. Results in \S \ref{sec:group spaces} and \S \ref{sec:Neron models} will be freely used. Recall that the Lagrangian fibration outside of the discriminant locus $\pi_0 : X_0 \to B_0$ is a $P_0$-torsor for a uniquely determined abelian scheme $P_0 = \Aut^{\circ}_{\pi_0} \to B_0$ (see \cite[\S 8]{ari-fed16} or \cite[\S 3]{kim25}). The goal of this section is to show the morphism $\pi_1 : X_1 = \pi^{-1}(B_1) \to B_1$ is still manageable. Our results are highly motivated by the work \cite{hwang-ogu09}, which is a generalization of Kodaira's work on singular fibers of minimal elliptic fibrations.
	
	\begin{theorem} \label{thm:Neron model of fibration codimension 1}
		Keep the notations and assumptions of \Cref{thm:translation automorphism scheme}. Let $B_0$ and $B_1$ be its first two $\delta$-loci in $B$ as in \eqref{eq:delta-loci}. Then
		\begin{enumerate}
			\item $X'_1 \to B_1$ is a N\'eron model of $X_0 \to B_0$.
			\item $P^a_1 \to B_1$ is a N\'eron model of $P_0 \to B_0$.
		\end{enumerate}
		As a consequence, $X'_1$ is a $P^a_1$-torsor over $B_1$.
	\end{theorem}
	
	\begin{remark}
		Let $b \in B$ be a closed point and $\Delta \to B$ a non-degenerate \emph{unramified} trait through $b$. \cite{hwang-ogu09, hwang-ogu11} studied the behavior of the fiber $X_b$ under an assumption that $X_{\Delta} = X \times_B \Delta$ is regular. One can show in this situation that (we omit the details)
		\begin{enumerate}
			\item $b \in B_1$ (construct a surjection $\pi^* \Omega_{\Delta} \to T^1_{\pi_{\Delta}} = (T^1_{\pi})_{|\Delta}$ similarly to \eqref{diag:cotangent} and show that the stabilizer sheaf $T^1_{\pi_{\Delta}}$ has fiber dimensions $\le 1$); and
			\item $X'_{\Delta} \to \Delta$ is a N\'eron model (use \Cref{thm:Neron algorithm} and imitate the first paragraph of \Cref{ex:minimal elliptic surface}).
		\end{enumerate}
		\Cref{thm:Neron model of fibration codimension 1} is a higher-dimensional base formulation of these results.
	\end{remark}
	
	We will see that the N\'eron mapping property of $X'$ is intimately related to the behavior of \emph{birational translation automorphisms} of $X$ (\Cref{def:birational translation automorphism}). This means we should study birational translation automorphisms, or more generally arbitrary birational maps between two Lagrangian fibered symplectic varieties. To state the main theorem to this direction, let us first note that the translation \emph{automorphism} scheme $P^a$ is not an invariant under birational maps between Lagrangian fibrations. To reconcile this, we consider its strict neutral component (\S \ref{sec:strict neutral component})
	\[ P^{\circ\circ} \subset P^a .\]
	
	\begin{theorem} \label{thm:birational map of fibration codimension 1}
		Let $\pi : X \to B$ and $\tau : Y \to B$ be Lagrangain fibrations of smooth symplectic varieties with surjective $\pi'$ and $\tau'$. Let $f : X \dashrightarrow Y$ be a birational map over $B$. Then
		\begin{enumerate}
			\item $f$ induces an isomorphism $P_X^{\circ\circ} = P_Y^{\circ\circ} (= P^{\circ\circ})$ and it is $P^{\circ\circ}$-equivariant. In particular, the $\delta$-loci \eqref{eq:delta-loci} of $\pi$ and $\tau$ coincide.
			\item $f$ is an isomorphism over the first $\delta$-locus $B_1$. In particular, $P^a_X$ and $P^a_Y$ are isomorphic over $B_1$.
		\end{enumerate}
	\end{theorem}
	
	\begin{remark}
		Two remarks about \Cref{thm:birational map of fibration codimension 1}.
		\begin{enumerate}
			\item The birational map $f : X \dashrightarrow Y$ is allowed to be non-symplectic.
			\item $B_1 \subset B$ always has a codimension $\ge 2$ complement by $\delta$-regularity but the inequality may be strict. For example, we may have $B_1 = B$ in certain cases, and the theorem for these cases claims that every birational map $f$ is an isomorphism.
		\end{enumerate}
	\end{remark}
	
	The rest of this section is devoted to the proof of the two main theorems.

	\subsection{Flops between Lagrangian fibrations}
		We start with an observation that the N\'eron mapping property of $\pi' : X' = X \setminus \Sing(\pi) \to B$ in \eqref{eq:pi'} and \emph{birational translation automorphisms} of $\pi$ are closely related.
		
		\begin{definition} \label{def:birational translation automorphism}
			In the setting of \Cref{def:translation automorphism}, a \emph{birational translation automorphism} of $\pi : X \to B$ is a birational automorphism $f : X \dashrightarrow X$ over $B$ such that for every point $b \in B_0$, the automorphism $f_b : X_b \to X_b$ is a translation of an abelian variety.
		\end{definition}
		
		Notice that every birational automorphism of $X$ over $B$ is necessarily defined on $X_0$ (since abelian varieties do not contain rational curves, e.g., \cite[Lemma 3.26]{kim25}). Hence $f_b : X_b \to X_b$ is indeed a regular automorphism for every $b \in B_0$ and the definition makes sense.
		
		\begin{lemma} \label{lem:equivalence between Neron and birational automorphism}
			Assume $\pi'$ in \eqref{eq:pi'} is surjective. Then the following are equivalent.
			\begin{enumerate}
				\item $X'$ is N\'eron under $B_0 \subset B$.
				\item For every quasi-projective \'etale morphism $U \to B$, a birational translation automorphism $f : X_U \dashrightarrow X_U$ over $U$ induces a regular translation automorphism $f : X'_U \to X'_U$ over $U$.
			\end{enumerate}
		\end{lemma}
		\begin{proof}
			$(\Longrightarrow)$ A birational automorphism $f$ induces a regular automorphism $f : X_{U_0} \to X_{U_0}$ over $U_0$ because it has abelian variety fibers. It subsequently extends to an automorphism $f : X'_U \to X'_U$ over $U$ by the N\'eron mapping property of $X'$.
			
			\bigskip
			
			$(\Longleftarrow)$ Since $\pi'$ is surjective and $P^a$ is acting freely on $X'$, it is enough to show that $X'$ is weakly N\'eron under $B_0 \subset B$ by \Cref{thm:weak Neron model of torsor}. By \Cref{prop:Neron is local in smooth topology} and the surjectivity of $\pi' : X' \to B$, up to taking a covering of $B$ in \'etale topology we may assume $\pi'$ has a section $e : B \to X'$. Let $U \to B$ be a quasi-projective \'etale morphism (see \Cref{rmk:more assumptions in weak NMP}) and $s_0 : U_0 \to X'_{U_0}$ a morphism. The difference $s_0 - e_0$ defines a translation automorphism of $X_{U_0}$, or a birational translation automorphism
			\[ f : X_U \dashrightarrow X_U .\]
			By assumption, $f$ defines an automorphism $f : X'_U \to X'_U$ over $U$. In particular, $s_0$ uniquely extends to a morphism $s = f \circ e : U \to X'_U$. This shows the weak N\'eron mapping property of $X'$.
		\end{proof}
		
		This reduces the study of N\'eron mapping property of $X'$ to the study of all birational translation automorphisms locally on $B$. In the rest of this subsection, we focus on proving the following \cref{prop:birational map between fibration codimension 1}, which will be a preparatory step to the proof of \Cref{thm:birational map of fibration codimension 1}.
		
		Let us explain the terminology beforehand. The smooth part of a Lagrangian fibration $\pi' : X' \to B$ is \emph{surjective over codimension $1$ points in $B$} if the open subset $\pi(X') \subset B$ has codimension $\ge 2$ complement. Equivalently, this means every codimension $1$ (geometric) fiber $X_{\eta} = \pi^{-1}(\eta)$ has at least one integral component. A birational map $f : X \dashrightarrow Y$ over $B$ is \emph{isomorphic over codimension $1$ points in $B$} if there exists a dense open subscheme $U \subset B$ of codimension $\ge 2$ complement such that $f_U : X_U \dashrightarrow Y_U$ is defined everywhere and is an isomorphism.
		
		\begin{proposition} \label{prop:birational map between fibration codimension 1}
			Let $\pi : X \to B$ and $\tau : Y \to B$ be Lagrangian fibrations of smooth symplectic varieties. Assume $\pi'$ is surjective over codimension $1$ points in $B$. Then every birational map $f : X \dashrightarrow Y$ over $B$ is isomorphic over codimension $1$ points in $B$.
		\end{proposition}
		
		Observe the following before getting into its proof.
		
		\begin{lemma} \label{lem:surjective over codimension 1}
			If $\pi' : X' \to B$ is surjective over codimension $1$ points in $B$, then so is $\tau' : Y' = Y \setminus \Sing(\tau) \to B$.
		\end{lemma}
		\begin{proof}
			Assume on the contrary that there exists a codimension $1$ point $\eta \in B$ with nowhere reduced fiber $Y_{\eta}$. Taking closure, we have a prime divisor $D = \overline{\{\eta\}} \subset B$ such that $Y_D = \tau^{-1}(D)$ is a nowhere reduced divisor in $Y$. But $X$ and $Y$ are both relative minimal models over $B$ (i.e., they have terminal singularities, proper over $B$, and have relatively nef canonical divisors), so $f$ is isomorphic in codimension $1$ points by \cite[Corollary 3.54]{kol-mori}. This means $X_D = \pi^{-1}(D)$ is a nowhere reduced divisor in $X$ as well, contradicting the assumption that $\pi'$ is surjective over codimension $1$ points in $B$.
		\end{proof}

		\subsubsection{Flops between Lagrangian fibrations}
			Let $\pi : X \to B$ and $\tau : Y \to B$ be Lagarngian fibrations of smooth symplectic varieties, and let $f : X \dashrightarrow Y$ an arbitrary birational map over $B$ that is possibly non-symplectic. Since $X$ (resp. $Y$) is a relative minimal model over $B$ (i.e., $X$ has terminal singularities, $\pi$ is proper, and $K_X$ is $\pi$-nef), \cite[Theorem 1]{kaw08} shows $f$ is a composition of a finite sequence of $\varepsilon L$-flops where $L$ is a Cartier divisor on $X$ and $0 < \varepsilon \ll 1$ is a rational number (in particular, $(X, \varepsilon L)$ is a terminal pair):
			\begin{equation} \label{diag:flops}
				\begin{tikzcd}[column sep=tiny]
					X = X^{(1)} \arrow[rr, dashed, "f^{(1)}"] \arrow[rd, "g^{(1)}"'] & & X^{(2)} \arrow[rr, dashed, "f^{(2)}"] \arrow[dl, "h^{(1)}"] \arrow[rd] & & \ \cdots \ \arrow[rr, dashed, "f^{(t)}"] \arrow[dl] \arrow[dr, "g^{(t)}"'] & & X^{(t+1)} \arrow[rr, "f^{(t+1)}", "\cong"'] \arrow[dl, "h^{(t)}"] & \ & Y \\
					& \bar X^{(1)} & & \bar X^{(2)} & & \bar X^{(t)}
				\end{tikzcd}.
			\end{equation}
			The diagram is defined over $B$. Specifying the last isomorphism $f^{(t+1)}$ will be important in the case $Y = X$.
			
			\begin{lemma}
				In \eqref{diag:flops}, the following holds for all $i$.
				\begin{enumerate}
					\item $X^{(i)}$ is a smooth symplectic variety.
					\item $\bar X^{(i)}$ is a (singular) symplectic variety.
					\item $X^{(i)} \to B$ (resp. $\bar X^{(i)} \to B$) is a Lagrangian fibration.
					\item Assume that $\pi' : X' \to B$ is surjective over codimension $1$ points in $B$. Then so is $(\pi^{(i)})' : (X^{(i)})' \to B$.
				\end{enumerate}
			\end{lemma}
			\begin{proof}
				Using induction on $i$, we may focus on the case $i = 1$ and (to ease notation) temporarily denote a single flop $f$ by
				\[\begin{tikzcd}
					X \arrow[r, "g"] & \bar X & Y \arrow[l, "h"']
				\end{tikzcd}.\]
				Flops are isomorphic in codimension $1$, so the symplectic form $\sigma$ on $X$ defines a symplectic form on a codimension $2$ complement of $Y$. Such a symplectic form extends to a regular $2$-form on the common resolution of $X$ and $Y$, so this proves $Y$ is a symplectic variety. Similarly, $\bar X$ is a symplectic variety. Note that $Y$, as a flop of a smooth variety, is $\QQ$-factorial and terminal by \cite[Proposition 3.37, Corollary 3.42]{kol-mori}. Hence \cite[Corollary 31]{nam08} applies and shows $Y$ is smooth.
				
				The varieties $X$, $\bar X$, and $Y$ are isomorphic generically over $B$ so the third item is clear. The last item is simply \Cref{lem:surjective over codimension 1} applied to $X \dashrightarrow Y$.
			\end{proof}
			
			This reduces the study of a birational map $f : X \dashrightarrow Y$ to a single flop $f^{(i)} : X^{(i)} \dashrightarrow X^{(i+1)}$ in \eqref{diag:flops}. We may thus assume $f : X \dashrightarrow Y$ is a single $\varepsilon L$-flop between \emph{smooth} symplectic varieties. Denote the flopping diagram by
			\begin{equation} \label{diag:single flop}
			\begin{tikzcd}
				X \arrow[rr, dashed, "f"] \arrow[rd, "g"'] & & Y \arrow[ld, "h"] \\
				& \bar X
			\end{tikzcd}
			\end{equation}
			where everything is Lagrangian fibered over $B$. Since $\pi'$ and $\tau'$ are surjective over codimension $1$ points in $B$, we may use the results in \S \ref{sec:translation automorphism scheme} up to taking a codimension $\ge 2$ complement of $B$. Let us study this diagram step by step.
			
			\begin{lemma} \label{lem:flop contracts dimension at least 2}
				In \eqref{diag:single flop}, let $W \subset \Exc(g)$ be an irreducible component of the exceptional locus of $g$ and $\bar W = g(W)$ its image in $\bar X$. Then $\dim W - \dim \bar W \ge 2$.
			\end{lemma}
			\begin{proof}
				Since $g : X \to \bar X$ is a symplectic resolution, \cite[Theorem 1.2(i)]{wie03} applies and shows
				\[ \dim W - \dim \bar W \ge \codim W .\]
				But $\codim W \ge 2$ because $g$ was a flopping contraction. This lemma will be generalized in \Cref{prop:exceptional locus of flop}.
			\end{proof}
			
			The following is a result in \cite[\S 2--3]{hwang-ogu09} (especially Proposition 3.8). Let us recover their results to illustrate the use of \Cref{thm:translation automorphism scheme} and \ref{thm:delta-regularity}. See \Cref{rmk:Hwang-Oguiso's result} as well.
			
			\begin{proposition} [Hwang--Oguiso] \label{prop:Hwang-Oguiso's result}
				There exists a Zariski open subset $U \subset B$ with codimension $\ge 2$ complement with the following properties for every closed point $b \in U$:
				\begin{enumerate}
					\item The normalization $X_b^{\nu}$ of the fiber $X_b$ is a smooth projective variety, possibly disconnected.
					\item Assume that a connected component $T$ of $X_b^{\nu}$ contains a rational curve $C$. Then $C$ is isomorphic to $\PP^1$ and there exists a $\PP^1$-bundle morphism
					\[ f : T \to A \]
					to an abelian variety $A$ of dimension $n-1$, realizing $C$ as one of its fibers.
				\end{enumerate}
			\end{proposition}
			\begin{proof}
				We prove $U = B_1$ satisfies the desired property. The statements are trivial when $X_b$ is an abelian variety, so let us assume $b \in B_1 \setminus B_0$. The algebraic group $G = P^{\circ\circ}_b$ acts on $X_b$ by \Cref{thm:translation automorphism scheme} and hence its normalization $X_b^{\nu}$ as well. The $\delta$-value of $G$ is precisely $\delta(G) = 1$ by \Cref{thm:delta-regularity}, so the Chevalley short exact sequence of $G$ is
				\[ 0 \to \GG \to G \to A \to 0 ,\]
				where $\GG = \GG_m$ or $\GG_a$ is a $1$-dimensional connected and commutative linear algebraic group and $A$ is an abelian variety of dimension $n-1$. Moreover, the stabilizer of each point $x \in X_b$ is either trivial or $\GG$ again by the $\delta$-regularity in \Cref{prop:delta-regularity} (it is trivial if and only if $x \in X_b'$ by \Cref{prop:action is free on X'}). Consider the normalization morphism
				\[ \nu : X_b^{\nu} \to X_b ,\]
				which is automatically $G$-equivariant. The stabilizer dimension of each point in $X_b^{\nu}$ is $\le 1$. But the $G$-subvariety $\Sing(X_b^{\nu})$ has dimension $\le n-2$ by normality, so this forces $\Sing(X_b^{\nu}) = \emptyset$ and proves the smoothness of $X_b^{\nu}$.
				
				Take a connected component $T$ of $X_b^{\nu}$, a smooth projective connected variety. If $\nu(T) \subset X_b$ is a reduced component, then the $G$-action on $T$ is generically free, so $T$ can be considered as a compactification $j : G \hookrightarrow T$. Consider a rational map
				\[\begin{tikzcd}[column sep=normal]
					f : T \arrow[r, dashed, "j^{-1}"] & G \arrow[r] & A
				\end{tikzcd}.\]
				Any rational map from a smooth variety to an abelian variety is everywhere defined, so $f$ is a morphism. This morphism is $G$-equivariant, so every closed fiber of $f$ is isomorphic to each other. In particular, $f$ is smooth projective by generic smoothness. Since it is a compactification of a $\GG$-torsor $G \to A$, it must be a $\PP^1$-bundle. As a result, $C \cong \PP^1$.
				
				If $\nu(T) \subset X_b$ is a non-reduced component, then $\GG \subset G$ acts trivially on the entire $T$. Hence $T$ has an induced $A$-action with finite stabilizers. Take a finite abelian subgroup $K \subset A$ containing all stabilizer subgroups and consider the quotient $\bar T = T / K$ with a free $\bar A = A/K$-action. Since $\bar T$ is normal and $\bar A$ acts on it freely, similar argument to above shows that $\bar T$ is smooth. Take a free quotient
				\[ g : \bar T \to C' = \bar T/ \bar A ,\]
				where $C'$ is a $1$-dimensional smooth proper algebraic space, whence a smooth projective curve. Since the fibers of $g$ are abelian varieties, $g$ induces a surjective morphism $\bar C \to C'$ from a rational projective curve $\bar C$ in $\bar T$. This forces $C' \cong \PP^1$. But $H^1_{\acute et}(\PP^1, \bar A) = 0$, so $g$ is necessarily a trivial $\bar A$-torsor and hence $\bar T \cong \bar A \times \PP^1$.
				
				To conclude, compose $T \to \bar T$ with $\pr_1 : \bar T \to \bar A$ and take its Stein factorization:
				\[\begin{tikzcd}
					T \arrow[r] \arrow[d, "f"] & \bar T \arrow[d] \\
					A' \arrow[r] & \bar A
				\end{tikzcd}.\]
				The $A$-action on $A'$ forces it to be smooth. Moreover, since $A' \to \bar A$ is finite and $A$-equivariant, it is an \'etale covering and hence $A'$ is an abelian variety. Finally, since $f : T \to A'$ is $A$-equivariant, it is a smooth projective morphism whose fibers are connected and containing a rational curve. The claim follows.
			\end{proof}
			
			\begin{remark} \label{rmk:Hwang-Oguiso's result}
				In \Cref{prop:Hwang-Oguiso's result}, the singular fiber $X_b$ and hence its normalization $X_b^{\nu}$ may not contain any rational curves \cite[Proposition~6.1]{mat01} \cite[Theorem~1.1]{hwang-ogu11}. This arises when $X_b$ is irreducible and non-reduced. In this case, $(X_b)_{\red} \to A$ is an ellpitic curve fiber bundle over an abelian variety of dimension $n-1$ (e.g., bielliptic surface). See \cite[Proposition~3.7--8]{hwang-ogu09} and \cite[Type~$\mathrm I_0$ in Table~4]{mat01}.
			\end{remark}
			
			\begin{proof} [Proof of \Cref{prop:birational map between fibration codimension 1}]
				Using the decomposition \eqref{diag:flops}, we may reduce to the case when $f$ is a single $\varepsilon L$-flop of $X$ as in \eqref{diag:single flop}. Let $W \subset \Exc(g)$ be an irreducible component of the exceptional locus of $g$ and $\bar W = g(W)$ its image. Then $f$ is isomorphic outside of $\pi(W) \subset B$. Hence, it is enough to show $\codim \pi(W) \ge 2$.
				
				Let $\bar x \in \bar W$ be a general point and set $b = \bar \pi (\bar x)$. Since $\bar X$ is symplectic and hence has canonical singularities, \cite[Corollary 1.5]{hacon-mcker07} shows the fiber $W_{\bar x}$ of $W \to \bar W$ is rationally chain connected. From \Cref{lem:flop contracts dimension at least 2}, we have $\dim W_{\bar x} \ge 2$. But $W_{\bar x}$ is contained in the fiber $X_b$, so this means $X_b$ contains a $2$-dimensional rationally chain connected subvariety. This is impossible for $b \in U$ where $U \subset B$ is a Zariski open subset of codimension $\ge 2$ complement defined in \Cref{prop:Hwang-Oguiso's result}. Hence $b \notin U$, proving $\pi(W) \cap U = \emptyset$ and in particular $\codim \pi(W) \ge 2$.
			\end{proof}

	\subsection{Codimension $1$ behavior of birational maps}
		\Cref{prop:birational map between fibration codimension 1}, together with Raynaud's extension \cref{thm:Raynaud extension of homomorphism}, implies the following proposition. This proposition in turn strengthens the discussions in the previous subsection.
		
		\begin{proposition} \label{prop:birational map is equivariant}
			Let $\pi : X \to B$ and $\tau : Y \to B$ be Lagrangain fibrations of smooth symplectic varieties with surjective $\pi'$ and $\tau'$. Let $f : X \dashrightarrow Y$ be a birational map over $B$. Then
			\begin{enumerate}
				\item $f$ induces an isomorphism $P^{\circ\circ}_X \cong P^{\circ\circ}_Y$. Let us denote them by $P^{\circ\circ}$.
				\item In \eqref{diag:flops}, $X^{(i)}$ and $\bar X^{(i)}$ for all $i$ admit $P^{\circ\circ}$-actions over $B$. As a consequence, $(X^{(i)}, P^{\circ\circ})$ and $(\bar X^{(i)}, P^{\circ\circ})$ are $\delta$-regular abelian fibrations.
				\item $g^{(i)}$ and $h^{(i)}$ are equivariant.
			\end{enumerate}
		\end{proposition}
		\begin{proof}
			(1) By \Cref{prop:birational map between fibration codimension 1}, there exists a codimension $\ge 2$ complement $U \subset B$ such that $f_U : X_U \to Y_U$ over $U$ is an isomorphism. This induces an isomorphism of smooth group spaces $P^{\circ\circ}_{X|U} \cong P^{\circ\circ}_{Y|U}$ over $U$. The isomorphism uniquely extends to an isomorphism $P^{\circ\circ}_X \cong P^{\circ\circ}_Y$ over $B$ by \Cref{thm:Raynaud extension of homomorphism}.
			
			\bigskip
			
			(2) and (3) Again let us simply write $X = X^{(i)}$, $\bar X = \bar X^{(i)}$, and $Y = X^{(i+1)}$. Since $g : X \to \bar X$ is an isomorphism over $U$ by \Cref{prop:birational map between fibration codimension 1}, there exists a homomorphism of group spaces $P^{\circ\circ}_{X|U} \to \Aut_{\bar \pi_U}$ over $U$. This uniquely extends to a homomorphism $P^{\circ\circ}_X \to \Aut_{\bar \pi}$ over $B$ again by \Cref{thm:Raynaud extension of homomorphism}, proving that $\bar X$ admits a $P^{\circ\circ}$-action. To prove the morphism $g$ is $P^{\circ\circ}$-equivariant, consider the diagram
			\[\begin{tikzcd}
				P^{\circ\circ} \times_B X \arrow[r] \arrow[d, "\id \times g"] & X \arrow[d, "g"] \\
				P^{\circ\circ} \times_B \bar X \arrow[r] & \bar X
			\end{tikzcd}.\]
			The diagram commutes over $U$ because $g$ was an isomorphism over it. Since $\bar X \to B$ is separated and $(P^{\circ\circ} \times_B X)_{|U} \subset P^{\circ\circ} \times_B X$ is dense by \Cref{lem:density respects flat base change}, the diagram commutes over the entire $B$ by \Cref{lem:morphism to separated space is determined by dense subspace}. Same discussions hold for $Y$ and $h : Y \to \bar X$.
			
			Finally, recall that the $P^{\circ\circ}$-stabilizers of points in $X$ are affine. Since $g : X \to \bar X$ was $P^{\circ\circ}$-equivariant, this shows that for each $b \in B$, there exists a point $\bar x \in \bar X_b \setminus \bar W_b = X_b \setminus W_b \neq \emptyset$ with affine stabilizer. By \cite[Lemma 5.16]{ari-fed16} (or the proof of \Cref{lem:delta-regularity}), this shows that every point $x \in \bar X_b$ has affine stabilizers and hence $(\bar X, P^{\circ\circ})$ is a weak abelian fibration.
		\end{proof}
		
		\begin{proposition} \label{prop:exceptional locus of flop}
			In \eqref{diag:single flop}, let $W \subset \Exc(g)$ be an irreducible component of the exceptional locus of $g$ and $\bar W = g(W)$ its image in $\bar X$. Write $d = \codim_X W$ and $e = \codim_{\bar X} \bar W$. Then
			\begin{enumerate}
				\item $e = 2d$.
				\item $W$ is an irreducible component of $\pi^{-1}(D_d)$ and $\pi(W)$ is an irreducible component of $D_d$ with codimension $d$.
				\item Every fiber of the morphism $W \to \bar W$ is rationally chain connected and has dimension $\ge d$.
				\item The normalization $\bar W^{\nu}$ of $\bar W$ is a symplectic variety of dimension $2(n-d)$. The morphism $\pi$ induces a Lagrangian fibration $\bar W^{\nu} \to \pi(W)^{\nu}$.
			\end{enumerate}
			Same results hold for $h$.
		\end{proposition}
		\begin{proof}
			(1) In \Cref{lem:flop contracts dimension at least 2}, we already used the inequality $\dim W - \dim \bar W \ge \codim W$ in \cite[Theorem 1.2(i)]{wie03}, i.e., the inequality $e - d \ge d$. Let us prove the opposite inequality $e - d \le d$.
			
			The image $\pi(W)$ is an irreducible subvariety of $B$. Since $\pi$ is equidimensional, we have $\codim \pi(W) \le \codim W = d$. Take the \emph{largest} integer $k$ with $\pi(W) \subset D_k$. We then have $k \le \codim D_k \le \codim \pi(W)$ by the $\delta$-regularity of the Lagrangian fibration. Combining the inequalities, we get
			\begin{equation} \label{eq:exceptional locus of flop:inequality}
				k \le \codim D_k \le \codim \pi(W) \le \codim W = d
			\end{equation}
			
			A general fiber of the morphism $W \to \bar W$ has dimension $e - d$. Fixing a general closed point $b \in \pi(W)$, this shows $\dim W_b = \dim \bar W_b + (e-d)$. Note that $b \in D_k \setminus D_{k+1}$ because $b \in \pi(W)$ is chosen to be general and $k$ is assumed to be maximal. Since $(\bar X, P^{\circ\circ})$ is a $\delta$-regular abelian fibration by \Cref{prop:birational map is equivariant}, $\bar W_b \subset \bar X_b$ has stabilizer dimension $\le k$ at every point. Hence $\dim \bar W_b \ge n-k$. We have a sequence of inequalities
			\[ (n-k) + (e-d) \le \dim \bar W_b + (e-d) = \dim W_b \le n .\]
			This means $e-d \le k$ and hence $e - d \le d$ by \eqref{eq:exceptional locus of flop:inequality}. Therefore, $e = 2d$.
			
			\bigskip
			
			(2) All the intermediate inequalities above are now equalities, so we have $k = d$, $\codim \pi(W) = d$, and that $\pi(W)$ is an irreducible component of $D_d$. Since $\dim W = 2n-d$, $\dim \pi(W) = n-d$, and $\pi$ is equidimensional of relative dimension $n$, this forces $W$ to be an irreducible component of $\pi^{-1}(\pi(W))$. (3) Apply \cite[Corollary 1.5]{hacon-mcker07} to $g : X \to \bar X$ (note that $\bar X$ is symplectic so it has canonical singularities). (4) This is simply a copy of \cite[Theorem 3.1]{mat15}, based on the previous result of \cite[Theorem 2.5]{kal06}.
		\end{proof}
		
		We can finally complete the proof of \Cref{thm:Neron model of fibration codimension 1} and \ref{thm:birational map of fibration codimension 1}.
		
		\begin{proof} [Proof of \Cref{thm:birational map of fibration codimension 1}]
			We have already proved in \Cref{prop:birational map is equivariant} that $P^{\circ\circ}_X \cong P^{\circ\circ}_Y$ and $f$ is equivariant. Hence the $\delta$-loci of $\pi$ and $\tau$ coincide. It only remains to prove $f$ is isomorphic over $B_1$. To do so, reduce it to the case when $f$ is a single flop as in \eqref{diag:single flop}. Let $W \subset \Exc(g)$ be an irreducible component of codimension $d \ge 2$. By \Cref{prop:exceptional locus of flop}, we have $\pi(W) \subset D_d$. In particular, $\pi(W) \cap B_1 = \emptyset$ and hence $g$ is isomorphic over $B_1$. Similarly, $h$ is isomorphic over $B_1$ and hence the result follows.
		\end{proof}
		
		\begin{proof} [Proof of \Cref{thm:Neron model of fibration codimension 1}]
			(1) To show $X_1'$ is N\'eron under $B_0 \subset B_1$, thanks to \Cref{lem:equivalence between Neron and birational automorphism}, it is enough to prove that for every quasi-projective \'etale morphism $U \to B_1$, a birational translation automorphism $f : X_U \dashrightarrow X_U$ over $U$ is a translation automorphism $f : X_U' \to X_U'$ over $U$. The $\delta$-regular loci are invariant under base change, so this means $U_1 = U$. Hence, the birational automorphism $f$ is in fact a regular automorphism on the entire domain $X_U$ by \Cref{thm:birational map of fibration codimension 1}.
			
			\bigskip
			
			(2) Thanks to \Cref{thm:weak Neron model of group space}, we only need to show $P^a_1$ is weakly N\'eron under $B_0 \subset B_1$. Let $U \to B_1$ be a quasi-projective \'etale morphism and $f : U_0 \to P_0$ a $U_0$-section. Since $P_0 \to B_0$ represents the translation automorphisms of $X_0$ by \Cref{thm:translation automorphism scheme}, $f$ induces a birational translation automorphism $f : X_U \dashrightarrow X_U$ over $U$. By \Cref{thm:birational map of fibration codimension 1} and $U_1 = U$, this uniquely extends to a translation automorphism of $X_U$, which in turn is equivalent to a section $U \to P^a_1$. This proves the weak N\'eron mapping property of $P^a_1$. Finally, by \Cref{prop:Neron torsor}, $X'_1$ is a $P^a_1$-torsor because $\pi'$ is surjective.
		\end{proof}
		
		We conclude this section with the following lemma, which will be used later on in the proof of \Cref{prop:equivalence of Neron mapping property of X'} and in \S \ref{ex:Beauville-Mukai system}. Recall from \Cref{prop:exceptional locus of flop} that every irreducible component $W \subset \Exc(g)$ of the exceptional locus of a flop has codimension $\le n$. If $\codim W = n$ then $X \to \bar X$ contracts $W$ to a point, so \cite[Theorem 1.1]{wie-wis03} and \cite[Theorem 1.2]{ame-ver24} conclude that $W$ is isomorphic to $\PP^n$.
		
		\begin{lemma} \label{lem:Mukai flop of Pn}
			Let $\pi : X \to B$ be a Lagrangian fibration of a smooth symplectic algebraic space $X$ of dimension $2n$. Assume that a closed fiber $X_b$ contains $\PP^n$. Let $f : X \dashrightarrow Y$ be the Mukai flop of such $\PP^n$ and let $\check \PP^n$ be the flopped locus. Then the following are equivalent.
			\begin{enumerate}
				\item The scheme-theoretic component of $X_b$ corresponding to $\PP^n$ is reduced.
				\item The scheme-theoretic component of $Y_b$ corresponding to $\check \PP^n$ is reduced.
			\end{enumerate}
		\end{lemma}
		\begin{proof}
			Before starting the proof, let us first notice that $Y$ admits a Lagrangian fibration $\tau : Y \to B$ with $\pi = \tau \circ f$. Moreover, since $X_b$ and $Y_b$ are lci schemes, they admit unique scheme-theoretic irreducible decompositions (the primary decomposition). Consider the Muaki flop diagram
			\[\begin{tikzcd}[column sep=tiny]
				& \tilde X \arrow[ld, "\epsilon"'] \arrow[rd, "\delta"] \\
				X \arrow[rr, dashed, "f"] \arrow[rd, "\pi"'] & & Y \arrow[ld, "\tau"] \\
				& B
			\end{tikzcd},\]
			where $\epsilon$ and $\delta$ are blowups of $\PP^n \subset X$ and $\check \PP^n \subset Y$. They share the same exceptional divisor $E \subset \tilde X$ so that both $\PP^n \longleftarrow E \to \check \PP^n$ are $\PP^{n-1}$-bundles ($E \subset \PP^n \times \check \PP^n$ is the universal hyperplane). If $\PP^n \subset X_b$ is a reduced component, then $E = \epsilon^{-1}(\PP^n) \subset \tilde X_b = \epsilon^{-1}(X_b)$ is a reduced component. By $E = \delta^{-1}(\check \PP^n) \subset \tilde X_b = \delta^{-1} (Y_b)$, this implies that $\check \PP^n \subset Y_b$ is a reduced component. The converse holds as well since $X$ is the Mukai flop of $\check \PP^n \subset Y$.
		\end{proof}

\section{The N\'eron model of a Lagrangian fibration} \label{sec:Neron model of fibration}
	This section extends the codimension $1$ behavior of Lagrangian fibrations studied in \S \ref{sec:Neron model of fibration codimension 1} to the deeper $\delta$-strata in $B$. Recall from \Cref{thm:Neron model of fibration codimension 1} that both the translation automorphism scheme $P^a$ and the smooth locus $X'$ are N\'eron under $B_0 \subset B_1$, and as result $X'_1$ is a $P_1$-torsor over $B_1$. When $n = 1$, i.e., when $\pi : X \to B$ is a minimal elliptic fibration without multiple fibers, we have $B_1 = B$ and this recovers the original result of N\'eron--Raynaud in the introduction and \Cref{ex:minimal elliptic surface}. For higher dimensional Lagrangian fibrations, both $P^a$ and $X'$ may fail to be N\'eron under the full extension of the base $B_0 \subset B$. In fact, it is unclear from our previous discussions whether $P^a_1 \to B_1$ and $X'_1 \to B_1$ may even admit a N\'eron model under the further extension $B_1 \subset B$. We need a new input to tackle this problem: the new ingredient will be the \emph{dual side} of the story.
	
	Recall that the translation automorphism scheme $P^a$ restricted over $B_0$ was an abelian scheme $\nu_0 : P_0 \to B_0$. Consider its dual abelian scheme $\check P_0 \to B_0$ (e.g., \cite[\S I.1]{fal-chai:abelian}), which by definition is its neutral Picard scheme $\check P_0 = \Pic^{\circ}_{\nu_0} \to B_0$. Since $X_0$ is a $P_0$-torsor, it has a same neutral Picard schemes to $P_0$. We may thus define the dual abelian scheme by
	\[ \check P_0 = \Pic^{\circ}_{\pi_0} \to B_0 ,\]
	which is a restriction of the full neutral Picard space $\Pic^{\circ}_{\pi} \to B$ of the Lagrangian fibration. A projective abelian scheme $P_0$ and its dual $\check P_0$ admit polarizations
	\[ \lambda_0 : P_0 \to \check P_0 ,\qquad \check \lambda_0 : \check P_0 \to P_0 .\]
	These finite \'etale homomorphisms connect many properties of $P_0$ and $\check P_0$, including their N\'eron extendability. More specifically, $P_0$ admits a N\'eron model $P^a_1$ over $B_1$ and $\Pic^{\circ}_{\pi}$ is N\'eron under $B_1 \subset B$ (line bundles on a regular scheme are determined by their codimension $1$ behavior), so the polarizations connect these complementary properties and yield the N\'eron models of both $P_0$ and $\check P_0$.
	
	\begin{theorem} \label{thm:Neron model of abelian schemes}
		Keep the notations and assumptions of \Cref{thm:translation automorphism scheme} and \ref{thm:Neron model of fibration codimension 1}. Then
		\begin{enumerate}
			\item There exists a N\'eron model $P \to B$ of the abelian scheme $P_0 \to B_0$.
			\item There exists a N\'eron model $\check P \to B$ of the dual abelian scheme $\check P_0 \to B_0$.
			\item Every finite \'etale homomorphism $\lambda_0 : P_0 \to \check P_0$ (e.g., polarization) over $B_0$ uniquely extends to a quasi-projective \'etale homomorphism
			\[ \lambda : P \to \check P \]
			over $B$. Same holds for finite \'etale homomorphisms $\check \lambda_0 : \check P_0 \to P_0$.
			\item $P$ and $\check P$ are $\delta$-regular with the same $\delta$-loci.
		\end{enumerate}
	\end{theorem}
	
	Recall the notion of a birational translation automorphism in \Cref{def:birational translation automorphism}. The N\'eron mapping property of $P$ says that for every \'etale morphism $U \to B$, we have
	\begin{equation} \label{eq:etale local section of birational translation automorphism space}
		P(U) = \{ f : X_U \dashrightarrow X_U : f \mbox{ is a birational translation automorphism over } U \} .
	\end{equation}
	This means $P$ can be considered as a moduli space of birational translation automorphisms of $\pi$; compare it with \eqref{eq:etale local section of translation automorphism scheme} and \Cref{def:translation automorphism scheme}.
	
	\begin{definition} \label{def:birational translation automorphism scheme}
		We call the N\'eron model $P \to B$ in \Cref{thm:Neron model of abelian schemes} the \emph{birational translation automorphism space} of $\pi$.
	\end{definition}
	
	By construction, there is an open immersion $P^a \subset P$ of group spaces over $B$. This immersion is an equality over $B_1$ by \Cref{thm:Neron model of fibration codimension 1} but not over $B$ in general. It is unclear from \Cref{thm:Neron model of abelian schemes} whether $P$ is separated (or of finite type), but its numerically trivial open subgroup space $P^{\tau} \subset P$ (\S \ref{sec:etale group space}) is always separated by \Cref{prop:aligned subgroup in weak Neron model is separated}. We obtain a sequence of open immersions of smooth commutative group spaces
	\begin{equation} \label{eq:inclusions of group spaces}
		P^{\circ\circ} \subset P^a \subset P^{\tau} \subset P .
	\end{equation}
	For each inclusion, we present an example for which the inclusion is strict (\S \ref{sec:main examples}). The three groups $P^{\circ\circ}$, $P^{\tau}$, and $P$ are invariant under birational equivalences of $\pi$, but the group $P^a$ is not.
	
	Similar results hold for the N\'eron model of $\pi_0 : X_0 \to B_0$ under the extension $B_0 \subset B$. The N\'eron model always exists, but checking whether it is of finite type or separated is a subtle question.
	
	\begin{theorem} \label{thm:Neron model of fibration}
		Keep the notations and assumptions of \Cref{thm:Neron model of abelian schemes}. Then
		\begin{enumerate}
			\item There exists a N\'eron model $X^n \to B$ of $\pi_0 : X_0 \to B_0$. It contains $X'$ as a $P^a$-invariant open subscheme.
			\item $X^n = \bigcup_{Y_U} Y'_U$ where $U \to B$ runs through every quasi-projective \'etale morphism and $\tau : Y_U \to U$ a Lagrangian fibered smooth symplectic variety that is birational to $X_U$.
		\end{enumerate}
		Consequently, $X^n$ is a $P$-torsor over $B$.
	\end{theorem}
	
	To talk about the finite type and separatedness of the N\'eron models $P$ and $X^n$, notice first that $P$ is of finite type (resp. separated) if and only if $X^n$ is. It is relatively easy to see the implications: $X' = X^n$ $\Longrightarrow$ $X^n$ is of finite type $\Longrightarrow$ $X^n$ is separated. This motivates us to consider the following.
	
	\begin{definition} \label{def:torsion locus}
		Define the following three subsets of $B$:
		\begin{align*}
			B_{\tor} &= \{ b \in B : X'_b = X^n_b \} ,\\
			B_{\ft} &= \mbox{maximal Zariski open subset of } B \mbox{ over which } X^n \mbox{ is of finite type} ,\\
			B_{\sep} &= \mbox{maximal Zariski open subset of } B \mbox{ over which } X^n \mbox{ is separated} .
		\end{align*}
	\end{definition}
	
	Clearly, $B_{\tor} = B$ (resp. $B_{\ft} = B$ and $B_{\sep} = B$) if and only if $X'$ is a $P$-torsor (resp. $X^n$ is of finite type and separated).
	
	\begin{proposition} \label{prop:subloci of B}
		Keep the notations and assumptions of \Cref{thm:Neron model of fibration}. Then
		\begin{enumerate}
			\item $B_{\tor}$ is the complement of a finite union of certain irreducible components of the $\delta$-loci $D_i$'s. In particular, it is Zariski open in $B$.
			
			\item $B_1 \subset B_{\tor} \subset B_{\ft} \subset B_{\sep} \subset B$.
			
			\item Assume $\dim X = 2n = 4$. Then the following are equivalent:
			\begin{enumerate}
				\item $P^{\tau} = P$, i.e., every (local) birational translation automorphism is numerically trivial.
				\item $X' = X^n$, i.e., $X'$ is a $P$-torsor.
				\item $X^n \to B$ is of finite type.
				\item $X^n \to B$ is separated.
			\end{enumerate}
			In particular, $B_{\tor} = B_{\ft} = B_{\sep}$.
			
			\item $B_1 \subset B_{\tor}$ and $B_{\sep} \subset B$ may be strict.
		\end{enumerate}
	\end{proposition}
	
	We expect \Cref{prop:subloci of B}(3) to hold in arbitrary dimensions, and therefore the sequence \eqref{eq:inclusions of group spaces} will be useful to understand the behavior of the N\'eron models $X^n$ and $P$. We provide three examples in \S \ref{sec:main examples} for which we can calculate all four groups in \eqref{eq:inclusions of group spaces}. In the first two examples, we have $P^{\tau} = P$ and moreover $X'$ is a $P$-torsor. In the last example, we have $P^{\tau} \subsetneq P$ and moreover $X'$ is not a $P$-torsor.
	
	Though $X'$ fails to be a $P$-torsor in general, one may still hope whether there \emph{exists} a smooth group space $G \to B$ making $X'$ a $G$-torsor. Unfortunately \S \ref{ex:Beauville-Mukai system} provides a counterexample for this statement, too:
	
	\begin{proposition} \label{prop:X' is not G-torsor}
		There is an example of a Lagrangian fibration $\pi : X \to B$ with surjective $\pi'$ such that $X'$ is not a $G$-torsor for every smooth group space $G \to B$.
	\end{proposition}
	
	The first two subsections will be devoted to the proof of the above results. Examples and applications will be presented in the next two sections \S \ref{sec:main examples}--\ref{sec:applications}. Again we will freely use the results in \S \ref{sec:group spaces} and \S \ref{sec:Neron models}. Notations and assumptions in \Cref{thm:translation automorphism scheme} and \ref{thm:Neron model of fibration codimension 1} will be used throughout.

	\subsection{The N\'eron model of the automorphism and Picard space}
		A line bundle on a regular scheme is determined by its restriction to a codimension $\ge 2$ complement. This can be interpreted as the N\'eron mapping property of the Picard space and is used in \cite[Theorem 9.5.4]{neron} and \cite[Theorem 6.2]{hol19} in the context of constructing the N\'eron model of the relative Jacobian of a flat projective family of curves.
		
		\begin{lemma} \label{lem:Picard space exists}
			There exists a well-defined relative Picard space $\Pic_{\pi} \to B$ with $\Lie \Pic_{\pi} \cong \TT^*_B$. As a result, the neutral relative Picard space $\Pic^{\circ}_{\pi} \to B$ is a smooth group space.
		\end{lemma}
		\begin{proof}
			The existence of a group space $\Pic_{\pi}$ (\Cref{thm:Picard space}) is from the flat properness of $\pi$ together with $\pi_* \mathcal O_X^{\fppf} = \mathcal O_B^{\fppf}$ (\Cref{lem:pi is a universal fibration}). Its Lie algebra scheme can be computed from Matsushita's result $R^1\pi_*\mathcal O_X \cong \Omega_B$ \cite[Theorem 1.3]{mat05} \cite[\S 4]{sch23} (or \Cref{thm:matsushita}). The smoothness of the Lie algebra scheme implies that its neutral component $\Pic^{\circ}_{\pi}$ is a smooth group space by \Cref{prop:neutral component}. We do not know whether $\Pic_{\pi}$ is smooth.
		\end{proof}
		
		\begin{proposition} \label{prop:Neron mapping property of neutral Picard space}
			Assume $\pi' : X' \to B$ is surjective. Then the neutral Picard space $\Pic_{\pi}^{\circ}$ is N\'eron under $B_1 \subset B$.
		\end{proposition}
		\begin{proof}
			It is enough to check the N\'eron mapping property of $\Pic^{\circ}_{\pi}$ \'etale locally in $B$ by \Cref{prop:Neron is local in smooth topology}. Since $\pi'$ is surjective, we may assume $\pi : X \to B$ has a section. Let $Z \to B$ be a smooth separated scheme of finite type (\Cref{rmk:more assumptions in NMP}). Both $Z$ and $X_Z = X \times_B Z$ are smooth varieties with dense open subsets $Z_1 \subset Z$ and $X_{Z_1} \subset X_Z$ with codimension $\ge 2$ complements. Hence the restriction homomorphisms
			\[ \Pic(X_Z) \to \Pic(X_{Z_1}), \qquad \Pic(Z) \to \Pic(Z_1) \]
			are isomorphisms. Since $\pi$ has a section, there exist isomorphisms $\Pic_{\pi}(Z) \cong \Pic(X_Z) / \Pic Z$ and $\Pic_{\pi}(Z_1) \cong \Pic(X_{Z_1}) / \Pic Z_1$ by \cite[Theorem 9.2.5]{kle:picard} (or Tag~0D28). This yields an isomorphism
			\[ \Pic_{\pi}(Z) \to \Pic_{\pi}(Z_1) ,\]
			which is the N\'eron mapping property of $\Pic_{\pi}$ under $B_1 \subset B$. Its neutral component $\Pic^{\circ}_{\pi}$ is a connected component of $\Pic_{\pi}$ so it satisfies the N\'eron mapping property, too.
		\end{proof}
		
		\begin{remark}
			In fact, the above proof shows: if a connected component $\Pic^{\alpha}_{\pi}$ of $\Pic_{\pi}$ dominates $B$, then it is N\'eron under $B_1 \subset B$.
		\end{remark}
		
		The neutral Picard space $\Pic^{\circ}_{\pi}$ is in general \emph{not} N\'eron under the extension $B_0 \subset B_1$. However, we can force it to be N\'eron by taking its separated quotient (this is the strategy in \cite[Theorem 9.5.4]{neron} and \cite[Theorem 6.2]{hol19}). Recall from the introduction of this section that we have the equality $\check P_0 = \Pic^{\circ}_{\pi_0}$.
		
		\begin{lemma} \label{lem:check P_1}
			The dual abelian scheme $\check P_0 \to B_0$ admits a separated N\'eron model $\check P_1 \to B_1$.
		\end{lemma}
		\begin{proof}
			We construct $\check P_1 \to B_1$ as a separated quotient of $\Pic^{\circ}_{\pi_1} \to B_1$. Take an arbitrary polarization $\check \lambda_0 : \check P_0 \to P_0$, a finite \'etale homomorphism of abelian schemes over $B_0$. By the N\'eron mapping property of $P^a$ under $B_0 \subset B_1$ (\Cref{thm:Neron model of fibration codimension 1}), the polarization uniquely extends to a homomorphism
			\[ \check \lambda_1 : \Pic^{\circ}_{\pi_1} \to P^a_1 .\]
			If $\check \lambda_1$ is \'etale, then $\Pic^{\circ}_{\pi_1} \to B_1$ is aligned by \Cref{prop:aligned under etale homomorphism}.\footnote{We caution the reader that $\Pic^{\circ}_{\pi}$ is not aligned over the entire base $B$ in general.} We can therefore consider its separated quotient (\Cref{prop:separated quotient})
			\[ \Pic^{\circ}_{\pi_1} \to \check P_1 .\]
			
			Recall from \Cref{thm:smooth Lie algebra subscheme} and \ref{lem:Picard space exists} that we have $\Lie P^a \cong \Lie \Pic^{\circ}_{\pi} \cong \TT^*_B$. To show $\check \lambda_1$ is \'etale, it is thus enough to show $\Lie \check \lambda_1 : \TT^*_{B_1} \to \TT^*_{B_1}$ is an isomorphism. Equivalently, we may show the coherent sheaf homomorphism $\check \lambda_1^* : T_{B_1} \to T_{B_1}$ is an isomorphism. Choose a polarization $\lambda_0 : P_0 \to \check P_0$ of $P_0$ so that the compositions
			\begin{equation} \label{eq:composition of polarization and dual polarization}
			\begin{tikzcd}
				P_0 \arrow[r, "\lambda_0"] & \check P_0 \arrow[r, "\check \lambda_0"] & P_0
			\end{tikzcd} ,\qquad \begin{tikzcd}
				\check P_0 \arrow[r, "\check \lambda_0"] & P_0 \arrow[r, "\lambda_0"] & \check P_0
			\end{tikzcd}
			\end{equation}
			are both multiplication endomorphisms by $m$ for a positive integer $m$ (e.g., dual polarization in \cite[\S 14.4]{bir-lange:abelian}). Replacing $\lambda_0$ by $2\lambda_0$ if necessary, we may assume $\lambda_0$ is a polarization obtained by a line bundle $L$ on $X$ (that is $\pi$-ample over $B_0$) \cite[\S I.1.6]{fal-chai:abelian}. Let us explicitly define an extension of $\lambda_0$ over $B$ by
			\begin{equation} \label{eq:polarization morphism}
				\lambda : P^a \to \Pic^{\circ}_{\pi} ,\qquad f \mapsto [f^* L \otimes L^{-1}] .
			\end{equation}
			Though $\lambda$ may \emph{not} be a group homomorphism (see \Cref{rmk:cubish line bundle}), it sends the identity section to the identity section. Hence we can consider the homomorphism between their conormal sheaves $\lambda^* : T_B \to T_B$. Over $B_0$, \eqref{eq:composition of polarization and dual polarization} says the compositions $\lambda_0^* \circ \check \lambda_0^*$ and $\check \lambda_0^* \circ \lambda_0^*$ are both multiplication by $m$ endomorphism of $T_{B_0}$. Hence both $\lambda_1^* \circ \check \lambda_1^*$ and $\check \lambda_1^* \circ \lambda_1^*$ are forced to be $m \cdot \id$ on $T_{B_1}$ because $T_{B_1}$ is locally free, proving that $\check \lambda_1^*$ is an isomorphism.
			
			It remains to show the N\'eron mapping property of $\check P_1 \to B_1$. The uniqueness part is clear from its separatedness, so let us prove the existence using the same strategy to the proof of \Cref{prop:Neron mapping property of neutral Picard space}. Assume without loss of generality that $\pi$ has a section and let $Z \to B_1$ be a smooth separated scheme of finite type. The base change $X_Z$ is a smooth variety so the localization exact sequence yields surjective homomorphisms $\Pic(X_Z) \twoheadrightarrow \Pic(X_{Z_0})$ and $\Pic Z \twoheadrightarrow \Pic Z_0$. This yields a surjective restriction map $\Pic^{\circ}_{\pi_1}(Z) \twoheadrightarrow \check P_0 (Z_0)$. But this map factors through
			\[ \Pic^{\circ}_{\pi_1}(Z) \to \check P_1(Z) \to \check P_0 (Z_0) ,\]
			so the latter homomorphism $\check P_1(Z) \to \check P_0 (Z_0)$ is a surjection.
		\end{proof}
		
		\begin{proof} [Proof of \Cref{thm:Neron model of abelian schemes}]
			(1) The translation automorphism scheme $P^a \to B$ is of finite type. Hence there exists a sufficiently divisible integer $m$ such that the multiplication endomorphism $[m]$ on $P^a$ has the image in its strict neutral component:
			\begin{equation} \label{eq:Neron extension of P0:m}
				[m] : P^a \to P^{\circ\circ} .
			\end{equation}
			
			Take an arbitrary polarization $\lambda_0 : P_0 \to \check P_0$ and extend it to an \'etale homomorphism $\lambda_1 : P^{\circ\circ}_1 \to \check P_1$ using the N\'eron mapping property of $\check P_1$ (\Cref{lem:check P_1}). Since $P^{\circ\circ}$ has connected fibers, the image of $\lambda_1$ is contained in the strict neutral component $\check P^{\circ\circ}_1$. But recall that $\check P_1$ was the separated quotient of $\Pic^{\circ}_{\pi_1}$ by construction, so we have $\check P^{\circ\circ}_1 = \Pic^{\circ\circ}_{\pi_1} \subset \Pic^{\circ}_{\pi_1}$ from \Cref{prop:separated quotient}. This means we have a homomorphism $\lambda_1 : P^{\circ\circ}_1 \to \Pic^{\circ}_{\pi_1}$ over $B_1$. Now by the N\'eron mapping property of $\Pic^{\circ}_{\pi}$ under $B_1 \subset B$ (\Cref{prop:Neron mapping property of neutral Picard space}) or Raynaud's extension \cref{thm:Raynaud extension of homomorphism}, this extends to a homomorphism
			\begin{equation} \label{eq:Neron extension of P0:lambda}
				\lambda : P^{\circ\circ} \to \Pic^{\circ}_{\pi} .
			\end{equation}
			
			The composition of \eqref{eq:Neron extension of P0:m} and \eqref{eq:Neron extension of P0:lambda} is
			\begin{equation} \label{eq:Neron extension of P0:lambda m}
				f = \lambda \circ [m] : P^a \to \Pic^{\circ}_{\pi} ,
			\end{equation}
			which is a quasi-projective \'etale homomorphism of smooth commutative group spaces (\Cref{rmk:finite type separated implies quasi-projective morphism}). \Cref{cor:Neron under quasi-projective etale homomorphism} shows $P^a_1 \to B_1$ has a N\'eron model $P \to B$ that contains $P^a$ as an open subgroup scheme. Since $P^a$ was already N\'eron under $B_0 \subset B_1$, this shows $P$ is N\'eron under $B_0 \subset B$.
			
			\bigskip
			
			(2) Take a polarization $\check \lambda_0 : \check P_0 \to P_0$. By the N\'eron mapping property of $P$, it uniquely extends to an \'etale homomorphism $\lambda : \Pic^{\circ\circ}_{\pi} \to P$. Since $\Pic^{\circ\circ}_{\pi} \to B$ is separated and of finite type by \Cref{thm:smooth group space with connected fibers}, $\lambda$ is a quasi-projective \'etale homomorphism. Again applying \Cref{cor:Neron under quasi-projective etale homomorphism}, $\check P_0 \to B_0$ has a N\'eron model $\check P \to B$ (whose strict neutral component is $\Pic^{\circ\circ}_{\pi}$). The restriction of $\check P$ over $B_1$ coincides with $\check P_1$ in \Cref{lem:check P_1}.
			
			\bigskip
			
			(3) Extension of $\lambda_0$ to $\lambda : P \to \check P$ is done by the N\'eron mapping property of $\check P$. It is an \'etale homomorphism since its extension over $B_1$ is \'etale as in the proof of \Cref{lem:check P_1} (and then we can use \Cref{prop:extension of etale homomorphism}). Its kernel $K = \ker \lambda$ is N\'eron under $B_0 \subset B$ by \Cref{prop:ses of Neron models}. Since $K_0 = \ker \lambda_0$ is a finite \'etale group scheme over $B_0$, the quasi-projectiveness of $K$ follows from \Cref{thm:Neron model of quasi-projective etale scheme}.
			
			\bigskip
			
			(4) $P$ is $\delta$-regular since it contains a $\delta$-regular group scheme $P^a$ (\Cref{thm:delta-regularity}) as an open subgroup scheme. To show the $\delta$-regularity of $\check P$, it is enough to show their $\delta$-functions agree. Take quasi-finite homomorphisms $\lambda : P \to \check P$ and $\check \lambda : \check P \to P$ such that $\lambda \circ \check \lambda$ and $\check \lambda \circ \lambda$ are multiplication by $m$ endomorphisms. For each $b \in B$, the homomorphism $\lambda_b : P^{\circ\circ}_b \to \check P^{\circ\circ}_b$ is quasi-finite and hence finite by \cite[Lemma 7.3.1]{neron}. Therefore, it induces a finite homomorphism $\lambda_b : (P^{\circ\circ}_b)_{\operatorname{aff}} \to (\check P^{\circ\circ}_b)_{\operatorname{aff}}$ between their maximal affine subgroups. Same argument applies to $\check \lambda_b$ and shows $(P^{\circ\circ}_b)_{\operatorname{aff}}$ and $(\check P^{\circ\circ}_b)_{\operatorname{aff}}$ are isogenous. In particular, their dimensions are the same.
		\end{proof}
		
		\begin{remark} \label{rmk:cubish line bundle}
			Fix a line bundle $L$ on $X$ and define a morphism
			\begin{equation}
				\lambda : P^a \to \Pic^{\circ}_{\pi}, \qquad f \mapsto \big[ f^* L \otimes L^{-1} \big] . \tag{\eqref{eq:polarization morphism} restated}
			\end{equation}
			Over $B_0$, it is a homomorphism $\lambda_0 : P_0 \to \check P_0$ by the theory of abelian schemes. When $\lambda$ is a homomorphism over the full base $B$, we say that $L$ \emph{satisfies the theorem of the square} \cite[\S 6.3]{neron}.\footnote{\cite[Proposition 5.12]{ari-fed16} calls such a line bundle $L$ a \emph{cubish line bundle}.} Since $\Pic^{\circ}_{\pi}$ is N\'eron under $B_1 \subset B$, the morphism $\lambda$ is a homomorphism over $B$ iff it is over $B_1$. We do not know how to check this even over $B_1$.
			
			To overcome this issue, we introduced the multiplication map $[m]$ in the proof of \Cref{thm:Neron model of abelian schemes} and considered instead
			\[ \lambda \circ [m] : P^a \to \Pic^{\circ}_{\pi}, \qquad f \mapsto \big[ (f^*)^{\circ m} (L) \otimes L^{-1} \big] \]
			as in \eqref{eq:Neron extension of P0:lambda m}. This is a homomorphism for every line bundle $L$.
		\end{remark}
		
		We have constructed $\check P$ as a N\'eron model of the separated quotient $\Pic^{\circ}_{\pi_1} \to \check P_1$ over $B_1$. Denote by $\check E_1$ its kernel, or the closure of the identity section of $\Pic^{\circ}_{\pi_1}$, so that we have a short exact sequence $0 \to \check E_1 \to \Pic^{\circ}_{\pi_1} \to \check P_1 \to 0$ over $B_1$. Extend it over $B$ using the N\'eron mapping property to yield a left exact sequence
		\[ 0 \to \check E \to \Pic^{\circ}_{\pi} \to \check P ,\]
		where $\check E$ is a N\'eron model of $\check E_1$ by \Cref{prop:ses of Neron models} (Note: $\check E$ is \emph{not} the closure of the identity section of $\Pic^{\circ}_{\pi}$). We do not know whether this sequence is right exact, but there is an easy situation when this is the case:
		
		\begin{proposition} \label{prop:dual Neron for geometrically integral fiber}
			Assume that $\pi$ has integral fibers over $B_1$. Then $\Pic^{\circ}_{\pi} = \check P$.
		\end{proposition}
		\begin{proof}
			The integral fiber assumption implies that $\Pic_{\pi_1} \to B_1$ is a separated group scheme, thanks to \cite[Theorem 9.4.8]{kle:picard} or \cite[Theorem 8.2.1]{neron}. This means $\check E_1 = 0$ and $\Pic^{\circ}_{\pi_1} = \check P_1$. Since both $\Pic^{\circ}_{\pi}$ and $\check P$ are N\'eron under $B_1 \subset B$, the claim follows from the uniqueness of the N\'eron model.
		\end{proof}

	\subsection{Proof of the main result}
		With the N\'eron model $P$ at hand, we can construct the N\'eron model $X^n$ in two different ways: either a $P$-torsor or the union of smooth loci for all birational models of the Lagrangian fibration.
		
		\begin{proof} [Proof of \Cref{thm:Neron model of fibration}]
			(1) Recall from \Cref{thm:Neron model of fibration codimension 1} that $X'_1$ is a $P_1$-torsor over $B_1$. This yields a cohomology class $[X_1'] \in H^1_{\acute et} (B_1, P)$. The restriction map
			\[ H^1_{\acute et} (B, P) \to H^1_{\acute et} (B_1, P) \]
			is injective and contains $X'_1$ in its image by \Cref{prop:Neron torsor restriction}. Hence $X'_1 \to B_1$ uniquely extends to a $P$-torsor $X^n \to B$. It is a N\'eron model as desired by \Cref{prop:torsor is Neron iff group is}. The equality $X'_1 = X^n_1$ over $B_1$ uniquely extends to a moprhism $X' \to X^n$, which is an equivariant open immersion by \Cref{lem:Neron extension is an open immersion}.
			
			\bigskip
			
			(2) To ease notation, we write $U = B$ and shrink $B$ whenever needed. Say $\tau : Y \to B$ is a Lagrangian fibered smooth symplectic variety and $f : X \dashrightarrow Y$ is a birational map. By \Cref{thm:birational map of fibration codimension 1}, $f_1 : X_1 \to Y_1$ is an isomorphism over $B_1$ and hence induces an isomorphism between their N\'eron models $f : X^n \to Y^n$ over $B$. We thus have an open immersion $f^{-1} : Y' \hookrightarrow Y^n \cong X^n$ and hence an open immersion
			\begin{equation} \label{eq:thm:Neron model of fibration}
				\bigcup_{f : X \dashrightarrow Y} f^{-1}(Y') \subset X^n .
			\end{equation}
			
			Fix a reference section $e : B \to X'$. Let $a : B \to X^n$ be an arbitrary (local) section. Considering $e$ as a section of $X^n$, the difference between the two sections define a section $f = e - a : B \to P$ because $X^n$ is a $P$-torsor. Recalling that $P$ is the birational translation automorphism space, this yields a birational translation automorphism $f : X \dashrightarrow X$. We have constructed a commutative diagram
			\[\begin{tikzcd}
				B \arrow[r, dashed, "a"] \arrow[rd, bend right=20, "e"'] & X' \arrow[d, dashed, "f"] \arrow[r, phantom, "\subset"] & X \arrow[d, dashed, "f"] \\
				& X' \arrow[r, phantom, "\subset"] & X
			\end{tikzcd}.\]
			An alternative interpretation of this diagram is that the birational automorphism $f : X \dashrightarrow Y \coloneq X$ to a Lagrangian fibered symplectic variety $\tau \coloneq \pi : Y \to B$ sends a rational section $a$ of $X'$ to a regular section $e$ of $Y'$. In other words, we have a commutative diagram
			\[\begin{tikzcd}[sep=tiny]
				& X' \arrow[dd, dashed, "f"] \arrow[rr, hook] & & X^n \arrow[dd, "f", "\cong"'] \\
				B \arrow[ru, dashed, "a"] \arrow[rd, "e"'] \\
				& Y' \arrow[rr, hook] & & Y^n
			\end{tikzcd}.\]
			This means the section $a : B \to X^n$ is contained in an open subscheme $f^{-1}(Y') \subset X^n$. Varying the arbitrary section $a$, the open immersion \eqref{eq:thm:Neron model of fibration} is an equality.
		\end{proof}
		
		The following is a more comprehensive version of \Cref{prop:subloci of B}(3). Compare it with \Cref{lem:equivalence between Neron and birational automorphism}.
		
		\begin{proposition} \label{prop:equivalence of Neron mapping property of X'}
			Assume $\dim X = 2n = 4$. Then all the following are equivalent.
			\begin{enumerate}
				\item $X' = X^n$, or equivalently
				\begin{enumerate}
					\item[\textnormal{(1n)}] $X'$ is N\'eron under $B_0 \subset B$.
					\item[\textnormal{(1t)}] $X'$ is a $P$-torsor over $B$.
					\item[\textnormal{(1f)}] $X^n \to B$ is of finite type.
					\item[\textnormal{(1s)}] $X^n \to B$ is separated.
				\end{enumerate}
				
				\item $P^{\tau} = P$, or equivalently
				\begin{enumerate}
					\item[\textnormal{(2f)}] $P \to B$ is of finite type.
					\item[\textnormal{(2s)}] $P \to B$ is separated.
				\end{enumerate}
				
				\item $\check P^{\tau} = \check P$, or equivalently
				\begin{enumerate}
					\item[\textnormal{(3f)}] $\check P \to B$ is of finite type.
					\item[\textnormal{(3s)}] $\check P \to B$ is separated.
				\end{enumerate}
			\end{enumerate}
		\end{proposition}
		\begin{proof}
			(1) $\Longleftrightarrow$ (1n) $\Longleftrightarrow$ (1t) is clear from \Cref{thm:Neron model of fibration}(1) that $X^n$ is a $P$-torsor and $X' \subset X^n$. (1f) $\Longleftrightarrow$ (2f) and (1s) $\Longleftrightarrow$ (2s) are again clear since $X^n$ is a $P$-torsor. (2) $\Longleftrightarrow$ (3) , (2f) $\Longleftrightarrow$ (3f), and (2s) $\Longleftrightarrow$ (3s) follow from the existence of quasi-projective \'etale homomorphisms $\lambda : P \to \check P$ and $\check \lambda : \check P \to P$ in \Cref{thm:Neron model of abelian schemes}. Therefore, we may show the equivalence between (1), (2), (2f), and (2s). (1) $\Longrightarrow$ (2f) follows since $X^n$ is a $P$-torsor. (2f) $\Longrightarrow$ (2) is clear from definition of the numerically trivial subgroup $P^{\tau}$. (2) $\Longrightarrow$ (2s) is a consequence of \Cref{prop:numerically trivial group is aligned} and \ref{prop:aligned subgroup in weak Neron model is separated}.
			
			\bigskip
			
			It remains to show (2s) $\Longrightarrow$ (1), and this is the only part we use the assumption $\dim X = 4$. Assume that $X' \to B$ is not N\'eron under $B_0 \subset B$. \Cref{lem:equivalence between Neron and birational automorphism} implies that (after shrinking $B$) there exists a birational translation automorphism $X \dashrightarrow X$ whose restriction to $X' \dashrightarrow X'$ is not a regular automorphism. Such a birational map decomposes into a sequence of flops as in \eqref{diag:flops}. When $\dim X = 4$, it further decomposes into a sequence of \emph{Mukai flops} by \cite[Theorem 1.2]{wie-wis03}.
			
			We may thus consider a single Mukai flop $f : X \dashrightarrow Y$ arising in the sequence whose induced map $f : X' \dashrightarrow Y'$ is not an isomorphism. The flopping $\PP^2$ and flopped $\check \PP^2$ are contained in fibers $X_b$ and $Y_b$ for a closed point $b \in B$. Since $f : X' \dashrightarrow Y'$ is not an isomorphism, both $\PP^2 \subset X_b$ and $\check \PP^2 \subset Y_b$ are reduced components by \Cref{lem:Mukai flop of Pn}. Consider a section $a : B \to X'$ passing through the reduced component $\PP^2 \subset X_b$ (after shrinking $B$). The closure of the rational section $B \to X \dashrightarrow Y$ has a fiber $\PP^1$ over $b$ since $f$ is a Mukai flop, and its intersection with $Y'$ is nonempty because $\check \PP^2 \subset Y_b$ is a reduced component. As a result, the closure of the rational section $B \dashrightarrow Y'$ has a $1$-dimensional fiber over $b$.
			
			Finally, the varieties $X'$ and $Y'$ admit open immersions to $P$ by \Cref{lem:Neron extension is an open immersion}:
			\[\begin{tikzcd}[column sep=tiny]
				B \arrow[rr, "a"] \arrow[rrrd, bend right=15] & & X' \arrow[rr, dashed, "f"] \arrow[rd, hook] & & Y' \arrow[ld, hook'] \\
				& & & P
			\end{tikzcd}.\]
			If $P \to B$ is separated, then the section $B \to P$ is a closed immersion, violating the previous conclusion that the closure of the rational section $a : B \dashrightarrow Y'$ has a $1$-dimensional fiber over $b$. Hence $P \to B$ cannot be separated.
		\end{proof}
		
		\begin{remark}
			The assumption $\dim X = 4$ in \Cref{prop:equivalence of Neron mapping property of X'} is used to produce a Mukai flop (so that we can use \Cref{lem:Mukai flop of Pn}). We expect a variant of \cite[Theorem 9.1]{cho-miy-she02} holds for Lagrangian fibrations, extending the result of \Cref{prop:exceptional locus of flop}, so that our strategy works for arbitrary dimensional symplectic varieties as well. We did not pursue this direction further.
		\end{remark}
		
		\begin{proof}[Proof of \Cref{prop:subloci of B}]
			(1) Let $b \in B \setminus B_{\tor}$ be a point. That is, assume that $X' \to B$ fails to be N\'eron at $b \in B$. By \Cref{lem:equivalence between Neron and birational automorphism}, up to shrinking $B$ there exists a birational map $f : X \dashrightarrow Y$ to $\tau : Y \to B$ such that $f : X' \dashrightarrow Y'$ has an undefined locus over $b$. By Kawamata's decomposition theorem in \eqref{diag:flops} and \Cref{prop:exceptional locus of flop}, a flop arises over an irreducible component $T \subset D_d$ of a $\delta$-locus containing $b$. This means $b \in T \subset B \setminus B_{\tor}$ where $T \subset D_d$ is an irreducible component. That is, $B \setminus B_{\tor}$ is a finite union of irreducible components of the $\delta$-loci $D_d$'s.
			
			(2) $X'_1 = X^n_1$ implies $B_1 \subset B_{\tor}$. Since $B_{\tor}$ is an open subset of $B$, we can restrict to $B_{\tor}$ and say $X'_{\tor} = X^n_{\tor}$ is separated over $B_{\tor}$. This means $B_{\tor} \subset B_{\sep}$. (3) was already proved in \Cref{prop:equivalence of Neron mapping property of X'}. (4) will be presented in \S \ref{sec:main examples}.
		\end{proof}
		
		We conclude with the following weaker versions of \Cref{prop:equivalence of Neron mapping property of X'}
		
		\begin{proposition} \label{prop:birational automorphism equals automorphism implies Neron}
			If $P^a = P$ then $X' = X^n$.
		\end{proposition}
		\begin{proof}
			Immediate from \Cref{lem:equivalence between Neron and birational automorphism}.
		\end{proof}
		
		\begin{proposition}
			Assume that $\pi$ has integral fibers over $B_1$. Then $P^{\tau} = P$ if and only if $X' = X^n$.
		\end{proposition}
		\begin{proof}
			The implication $(\Longleftarrow)$ was proved in \Cref{prop:equivalence of Neron mapping property of X'}, so let us prove $(\Longrightarrow)$. Recall from \Cref{prop:dual Neron for geometrically integral fiber} that we have $\Pic^{\circ}_{\pi} = \check P$ under such assumption. If we further assume that $P^{\tau} = P$, then $\check P^{\tau} = \check P$ by the equivalence (2) $\Longleftrightarrow$ (3) in \Cref{prop:equivalence of Neron mapping property of X'} and hence $\Pic^{\circ}_{\pi}$ is numerically trivial.
			
			Let $L$ be a $\pi$-ample line bundle on $X$ and consider the morphism
			\[ \lambda : P \to \Pic^{\circ}_{\pi} ,\qquad f \mapsto [f^*L \otimes L^{-1}] ,\]
			similarly to \eqref{eq:polarization morphism} (first define it over $B_1$ using \Cref{thm:birational map of fibration codimension 1} and then extend it over $B$ by the N\'eron mapping property of $\Pic^{\circ}_{\pi}$). But $\Pic^{\circ}_{\pi}$ is numerically trivial, so $f^* L \otimes L^{-1}$ is a numerically trivial line bundle and hence $f^*L$ is $\pi$-ample by the Nakai--Moishezon criterion. This shows that every $f \in P$ is a regular translation automorphism. In other words, $P^a = P$ and hence $X' = X^n$ by \Cref{prop:birational automorphism equals automorphism implies Neron}.
		\end{proof}

	\subsection{Examples and counterexamples} \label{sec:main examples}
		We present three examples that are illustrative to understand the behaviors of $P$ and its subgroup spaces in \eqref{eq:inclusions of group spaces}. We will freely use our main results.

		\subsubsection{First example: product of minimal elliptic surfaces} \label{ex:product of ellpitic surfaces}
			Let
			\[ \pi^i : S^i \to \Delta \qquad\mbox{for}\quad i = 1, \cdots, n \]
			be proper minimal elliptic surfaces with no multiple fibers over a smooth quasi-projective curve $\Delta$. Consider a Lagrangian fibration of a smooth symplectic variety
			\[ \pi : X = S^1 \times \cdots \times S^n \to B = \Delta^{\times n} .\]
			By N\'eron--Raynaud (or \Cref{thm:Neron model of fibration codimension 1}), each $S^i$ has a translation automorphism scheme $(P^i)^a = P^i \to \Delta$ making $(S^i)'$ its torsor. By \Cref{prop:product of Neron models}, both $P = P^1 \times \cdots \times P^n \to B$ and $X' = (S^1)' \times \cdots \times (S^n)' \to B$ are N\'eron under $B_0 = \Delta_0^{\times n} \subset B$. Moreover, $P$ regularly acts on the full symplectic variety $X$, so this shows
			\[ P^{\circ\circ} \subset P^a = P^{\tau} = P ,\]
			and that $X' = X^n$ from \Cref{prop:birational automorphism equals automorphism implies Neron}. The first inclusion is strict when $S^i$ has a singular fiber with at least two integral components.

		\subsubsection{Second example: Hilbert square of an elliptic surface} \label{ex:Hilbert square of ellpitic surface}
			Let $S \to \Delta$ be a proper minimal elliptic surface over a smooth quasi-projective curve with a central singular fiber $S_{t_0}$ of Kodaira type $\text{I}_2$ over $t_0 \in \Delta$. Consider a morphism
			\[ \pi : X = S^{[2]} \to \bar X = S^{(2)} \to B = \Delta^{(2)} ,\]
			where $S^{[2]}$ is the Hilbert scheme of length $2$ subschemes on $S$, and $S^{(2)} = S^{\times 2}/\mathfrak S_2$ and $\Delta^{(2)} = \Delta^{\times 2}/\mathfrak S_2$ are symmetric squares. Note that $X$ is a smooth symplectic fourfold, $B$ is a smooth surface, and $\pi$ is a Lagrangian fibration. The complement of the discriminant locus is
			\[ B_0 = \{ (t, s) \in \Delta_0^{(2)} : t \neq s \} \subset B .\]
			
			We claim $P^a = P$, i.e., every (local) birational translation automorphism is regular. By the previous example in \S \ref{ex:product of ellpitic surfaces}, we only need to consider the case around the diagonal in $B$. Shrinking $\Delta$, we may consider a birational translation $f : X \dashrightarrow X$ over $B = \Delta^{(2)}$. It induces a birational translation $\bar f : \bar X \dashrightarrow \bar X$, which again induces an $\mathfrak S_2$-equivariant birational translation
			\[ g : S^{\times 2} \dashrightarrow S^{\times 2} \quad\mbox{over}\quad \Delta^{\times 2} .\]
			Again by the previous example, every birational translation of $S^{\times 2} \to \Delta^{\times 2}$ is regular, so this shows $g$ and hence $\bar f$ is a regular automorphism. Therefore, the undefined locus of $f$ is contained in the exceptional divisor $E$ of $X = S^{[2]} \to \bar X = S^{(2)}$. By \cite[Theorem 1.1]{wie-wis03}, the undefined locus must be a disjoint union of $\PP^2$'s. But the exceptional divisor $E$ contains no $\PP^2$ in a single fiber because it is a $\PP^1$-bundle over $S$. Hence $f$ is defined everywhere. We conclude
			\[ P^{\circ\circ} \subsetneq P^a = P^{\tau} = P .\]
			In particular, $X' = X^n$ by \Cref{prop:subloci of B} (or \Cref{prop:birational automorphism equals automorphism implies Neron}). To show the first inclusion is strict, notice that the fiber over $b = (t_0, t_0) \in B = \Delta^{(2)}$ is
			\[ X_b \cong \big( \PP^1 \mbox{-bundle over }S_{t_0} \big) \cup 2 \cdot \big( \PP^2 \cup (\PP^1)^{\times 2} \cup \PP^2 \big) ,\]
			where the $2$ indicates that the three irreducible components in $X_b$ are non-reduced with multiplicity $2$ (as a cycle). Hence $X'_b$ has two connected components and so does $P^a_b$.
			
			\bigskip
			
			Notice that the fiber $X_b$ above contains two irreducible components isomorphic to $\PP^2$. Muaki flopping one $\PP^2$ yields a new smooth symplectic algebraic space $X \dashrightarrow Y$ with a Lagrangian fibration $\tau : Y \to B$. Recall that $P^{\circ\circ}$, $P^{\tau}$, and $P$ are birational invariants. However, we have lost the $\mathfrak S_2$-symmetry on the two irreducible components isomorphic to $\PP^2$'s in the fiber $Y_b$, so this means for the new Lagrangian fibration $\tau : Y \to B$, we have
			\[ P^{\circ\circ} = P^a \subsetneq P^{\tau} = P .\]
			This modified example still satisfies $P^{\tau} = P$, so $Y' = Y^n$ is a $P$-torsor by \Cref{prop:subloci of B}. The reason for this phenomenon is because the Mukai flopped $\PP^2$ has multiplicity $2$, so the birational map $X \dashrightarrow Y$ still induces an isomorphism $X' \to Y'$ and the N\'eron mapping property is preserved by \Cref{lem:equivalence between Neron and birational automorphism}. If we can flop a \emph{reduced} irreducible component $\PP^2$ in $X_b$, then this will no longer be the case and the torsor property will fail. This leads us to the third example.

		\subsubsection{Third example: Beauville--Mukai system of a degree $2$ K3 surface} \label{ex:Beauville-Mukai system}
			We thank Antonio Rapagnetta for pointing out an example in \cite[Example 3.11]{melo-rap-viv17}, \cite[\S 13 Example (3)]{oda-ses79}, and \cite[p.271]{mum72}. Let us first present two preparatory lemmas.
			
			\begin{lemma} \label{lem:G-torsor}
				Assume $X'$ is a $G$-torsor for a smooth group space $G \to B$. Then $G$ is an open subgroup space of $P^{\tau}$.
			\end{lemma}
			\begin{proof}
				Since $X' \to B$ is finite type and separated, so is $G \to B$ by descent. Note that $X'_1$ is both a $G_1$-torsor and $P_1$-torsor by \Cref{thm:Neron model of fibration codimension 1}. This forces $G$ and $P$ to be isomorphic over $B_1$ (see \cite{moret:torsor}). By the N\'eron mapping property of $P$ under $B_1 \subset B$, this isomorphism extends to an open immersion $G \hookrightarrow P$ (\Cref{prop:extension of etale homomorphism}). Since $P^{\tau}$ is a maximal numerically trivial subgroup space of $P$, the immersion factors though $P^{\tau}$.
			\end{proof}
			
			\begin{lemma} \label{lem:Mukai flop of fourfold}
				Assume $\dim X = 2n = 4$. Say a closed fiber $X_b$ contains a reduced component isomorphic to $\PP^2$ and another (not necessarily reduced) irreducible component $S$ with $S \cap \PP^2 \neq \emptyset$. Let $f : X \dashrightarrow X^+$ be the Mukai flop of $\PP^2$, $\check \PP^2 \subset X^+_b$ the flopped locus, and $S^+ \subset X^+_b$ the strict transform of $S$. Then
				\begin{enumerate}
					\item $S \cap \PP^2 \subset \PP^2$ and $S^+ \cap \check \PP^2 \subset \check \PP^2$ are projective dual to each other.
					\item Every irreducible component of $S \cap \PP^2$ is either a point or a line in $\PP^2$. Moreover, it can contain at most three points and three lines.
				\end{enumerate}
			\end{lemma}
			\begin{proof}
				(1) By \cite[Theorem 1]{mat00}, both $\PP^2$ and $S_{\red}$ are Lagrangian subvarieties of $X$. Therefore, this is a consequence of \cite[\S 4.2]{leung02}, where the projective dual is called the Legendre transformation.
				
				(2) Since $\PP^2 \subset X_b$ is an integral component, $\PP^2 \cap X'_b$ is isomorphic to $P^{\circ\circ}_b$, which is isomorphic to $\GG_m^{\times 2}$, $\GG_m \times \GG_a$, or $\GG_a^{\times 2}$. This means every $1$-dimensional irreducible component of $S \cap \PP^2$ is a line, and there are at most three of them. If there is a $0$-dimensional irreducible component, applying the same argument to the Mukai flop says there are at most three of them.
			\end{proof}
			
			Consider a degree $2$ K3 surface $\varphi : S \to \PP^2$ branched along a smooth sextic curve in $\PP^2$. Choose the sextic sufficiently so that it meets with a line $\PP^1 \subset \PP^2$ along precisely three points with multiplicity $2$. The preimage
			\[ \varphi^{-1}(\PP^1) = C_1 \cup C_2 \]
			is a union of two smooth rational curves with three transversal intersections.
			
			Let $L = \varphi^* \mathcal O_{\PP^2}(1)$ be an ample line bundle of degree $2$ and set $v = (0, L, 1) \in H^* (S, \ZZ)$, a primitive Mukai vector. The Beauville--Mukai system
			\[ \pi : X = M_v(S) \to B = |L| \cong \PP^2 \]
			is a Lagrangian fibration of a projective hyper-K\"ahler fourfold of $\text{K3}^{[2]}$-type. Over the closed point $b = [C_1 \cup C_2] \in B$, $\pi$ has a \emph{reduced} fiber with three irreducible components
			\[ X_b = S_1 \cup S_2 \cup S_3 .\]
			The two surfaces $S_1$ and $S_3$ are isomorphic to $\PP^2$ and $S_2$ is a projective surface that normalizes to a degree $6$ del Pezzo surface $\tilde S_2 \to S_2$ (that is, $\tilde S_2$ is isomorphic to a blowup of three non-collinear points in $\PP^2$). More precisely, there are six distinct points $p_i, q_i \in \tilde S_2$ for $i = 1,2,3$ so that $S_2$ is precisely $\tilde S_2$ with $p_i$ and $q_i$ identified. The interesting configuration of $S_1$, $S_2$, and $S_3$ is well-described in \cite[p.271]{mum72} and \cite[Figure 35 in p.84]{oda-ses79}. We simply recall that $S_1 \cap S_2$ consists of three lines in $S_1 = \PP^2$ and similarly $S_2 \cap S_3$ consists of three lines in $S_3 = \PP^2$. The intersection $S_1 \cap S_3$ consists of three points (compare with \Cref{lem:Mukai flop of fourfold}).
			
			\bigskip
			
			We imitate the method in \cite[Theorem 6.2]{has-tsc10} to construct an infinite sequence of Mukai flops. Perform a Mukai flop of $S_1 \cong \PP^2 \subset X$ and yield a smooth symplectic algebraic space $X \dashrightarrow X^+$. It has a fiber
			\[ X_b^+ = S_1^+ \cup S_2^+ \cup S_3^+ ,\]
			where $S_1^+$ is the flopped locus, and $S_2^+$ and $S_3^+$ are strict transforms of $S_2$ and $S_3$. Here are the more precise descriptions of $S_i^+$:
			\begin{itemize}
				\item $S_1^+$ is a flopped locus, so it is isomorphic to $\PP^2$ and is a reduced component of $X_b^+$ (\Cref{lem:Mukai flop of Pn}).
				\item Consider the Mukai flop diagram $X \to \bar X \longleftarrow X^+$. The image $\bar S_2 \subset \bar X$ of $S_2$ is a contraction of three lines in $S_2$ and hence isomorphic to $\PP^2$ with three points $p,q,r \in \PP^2$ identified. By \Cref{lem:Mukai flop of fourfold}, the strict transform $S_2^+$ separates this point to three different points, the projective dual of three lines in $S_2$, so it is isomorphic to $\PP^2$.
				\item This time, consider the roof diagram $X \longleftarrow \tilde X \to X^+$. The strict transform $\tilde S_3 \subset \tilde X$ is the blowup of $S_3$ along three points $S_1 \cap S_3 \subset S_3$, so it is isomorphic to a degree $6$ del Pezzo surface. Its image $S_3^+$ under the contraction $\tilde X \to X^+$ identifies three pairs of points.
			\end{itemize}
			Therefore, the configuration of $X_b^+$ looks similar to $X_b$ but the roles of each irreducible components are shuffled. Now $S_1^+$ and $S_2^+$ are the two components of $X_b^+$ isomorphic to $\PP^2$. Flopping $S_1^+$ will yield back the original $X$, so the only nontrivial Mukai flop is the flop of $S_2^+$. This gives another smooth symplectic algebraic space $X^+ \dashrightarrow X^{++}$ whose fiber is
			\[ X_b^{++} = S_1^{++} \cup S_2^{++} \cup S_3^{++} .\]
			Similar argument shows that $S_1^{++}$ is now a degree $6$ del Pezzo surface (with three pairs of points identified), and $S_2^{++}$ and $S_3^{++}$ are isomorphic to $\PP^2$. Flop $S_3^{++}$ again and yield $X^{+++}$, and keep going. We have constructed an infinite sequence of Mukai flops.
			
			\begin{table}[h]
				\begin{tabular}{|c|ccc|ccc|} \hline
					& \multicolumn{3}{c|}{Fiber components} & \multicolumn{3}{c|}{degree of $L$ on each component} \\\hline
					$X$ & $\PP^2$ & dP6 & $\PP^2$ & $a$ & & $b$ \\
					$\dashdownarrow$ & \multicolumn{1}{c}{$\dashdownarrow$} & & & \multicolumn{1}{c}{$\dashdownarrow$} & & \\
					$X^+$ & $\PP^2$ & $\PP^2$ & dP6 & $-a$ & $2a+b$ & \\
					$\dashdownarrow$ & & \multicolumn{1}{c}{$\dashdownarrow$} & & & \multicolumn{1}{c}{$\dashdownarrow$} & \\
					$X^{++}$ & dP6 & $\PP^2$ & $\PP^2$ & & $-2a-b$ & $3a+2b$ \\
					$\dashdownarrow$ & & & \multicolumn{1}{c|}{$\dashdownarrow$} & & & \multicolumn{1}{c|}{$\dashdownarrow$} \\
					$X^{+++}$ & $\PP^2$ & dP6 & $\PP^2$ & $4a+3b$ & & $-3a-2b$ \\
					$\dashdownarrow$ & \multicolumn{1}{c}{$\dashdownarrow$} & & & \multicolumn{1}{c}{$\dashdownarrow$} & & \\
					$\vdots$ & \multicolumn{3}{c|}{$\vdots$} & \multicolumn{3}{c|}{$\vdots$} \\\hline
				\end{tabular}
				
				\caption{Infinite sequence of Mukai flops and the degree of $L$ on the three irreducible components of their fibers over $b$.}
				\label{table:sequence of flops}
			\end{table}
			
			\begin{lemma} \label{lem:sequence of Mukai flops is not isomorphism}
				For every $k \in \ZZ_{>0}$, the composition of $k$ Mukai flops
				\[ X \dashrightarrow X^+ \dashrightarrow \cdots \dashrightarrow X^{+k} \]
				as above is not an isomorphism.
			\end{lemma}
			\begin{proof}
				Take an arbitrary $\pi$-ample line bundle $L$ on $X$. Since $S_1$ and $S_3$ are isomorphic to $\PP^2$, their Picard groups are isomorphic to $\ZZ$. Let $a > 0$ and $b > 0$ be the degree of $L_{|S_1}$ and $L_{|S_3}$, respectively. That is, if $l$ is a line in $S_1 = \PP^2$ then $a$ is the intersection number of $L$ and $l$ (similarly for $b$). Since $\tilde S_2$ is a degree $6$ del Pezzo surface, its Picard group is $\ZZ^4$ and we can compute the intersection numbers of $L$ with the rational curve generators of $\Pic S_2$.
				
				Let us now consider the strict transform of $L$ to $X^+$ and denote it by the same letter $L$ for simplicity. Since $X^+$ was the Mukai flop of $S_1 \subset X$, the degree of $L_{|S_1^+}$ becomes $-a$. The strict transform $S_2^+$ is now isomorphic to $\PP^2$, so we can consider the degree of $L_{|S_2^+}$. It is computed by the configuration of six lines on the degree $6$ del Pezzo surface $S_2$ and their contraction; one can check the degree is precisely $2a+b$. After another flop, we can play the same game and compute the degree of $L$ restricted to each irreducible component of the fiber $X^{++}_b$. This calculation shows that on the $k$-th step $X^{+k}$, the degree of $L$ on two $\PP^2$ components of the fiber $X^{+k}_b$ are (see \Cref{table:sequence of flops})
				\[ (k+1)a + kb > 0 \qquad\mbox{and}\qquad -ka - (k-1)b < 0 .\]
				Since there is a negative degree component, $L$ cannot be $\pi$-ample on $X^{+k}$. Therefore, the composition of a series of Mukai flops $X \dashrightarrow X^{+k}$ will never be an isomorphism for every $k$.
			\end{proof}
			
			We claim that the sequence of inclusions \eqref{eq:inclusions of group spaces} for this example is
			\begin{equation} \label{eq:third example sequence of inclusions}
				P^{\circ\circ} = P^a = P^{\tau} \subsetneq P .
			\end{equation}
			The last strict inclusion can be shown as in the proof of \Cref{prop:equivalence of Neron mapping property of X'}. Namely, take two birational models $X \dashrightarrow X^+$ that are connected by a single Mukai flop. By \Cref{thm:Neron model of fibration}, their smooth parts are open subschemes of $X^n$ and we have a commutative diagram
			\[\begin{tikzcd}[column sep=tiny]
				X' \arrow[rr, dashed] \arrow[rd, hook] & & (X^+)' \arrow[ld, hook'] \\
				& X^n
			\end{tikzcd}.\]
			This forces $X^n \to B$ to be non-separated by \Cref{lem:Mukai flop of Pn} and hence $P^{\tau} \subsetneq P$ by \Cref{prop:subloci of B}. Equivalently, $X' \subsetneq X^n$ or $X'$ is not a $P$-torsor.
			
			To show the equality $P^a = P^{\tau}$, let $f$ be a local section of $P^{\tau}$ around $b$. That is, shrinking $B$, $f : X \dashrightarrow X$ is a birational translation automorphism whose self-composition $f^{\circ m} : X \dashrightarrow X$ is a regular automorphism for some $m \ge 1$. By \cite[Theorem 1.2]{wie-wis03}, $f$ decomposes into a sequence of Mukai flops
			\[ f : X \dashrightarrow X^{(1)} \dashrightarrow \cdots \dashrightarrow X^{(t)} \xlongrightarrow{\cong} X .\]
			Since $f$ is an isomorphism over $B_1 \subset B$ by \Cref{thm:birational map of fibration codimension 1}, shrinking $B$, we may assume that the Mukai flops are over $b \in B$. There is only one possible sequence of Mukai flops we can take, which is $X \dashrightarrow X^+ \dashrightarrow X^{++} \dashrightarrow \cdots$ as above. By \Cref{lem:sequence of Mukai flops is not isomorphism}, $f^{\circ m}$ will never be an isomorphism unless $f$ was an isomorphism at the beginning. This proves $P^a = P^{\tau}$. Finally, $P^a_b$ cannot have multiple connected components because it acts freely on $X'_b$ and not all three components of $X_b$ are isomorphic. This proves $P^{\circ\circ} = P^a$.
			
			We have shown that $X'$ is not a $P$-torsor. Assume that there \emph{exists} a smooth group space $G \to B$ making $X'$ a $G$-torsor. \Cref{lem:G-torsor} shows $G$ must be an open subgroup space of $P^{\tau}$. Because $P^{\tau} = P^{\circ\circ}$ has connected fibers by \eqref{eq:third example sequence of inclusions}, this forces $G = P^{\circ\circ}$. Now $X'_b$ had three connected components whereas $G_b = P^{\circ\circ}_b$ has one, which is a contradiction. This justifies \Cref{prop:X' is not G-torsor}.

	\subsection{Applications} \label{sec:applications}
		Let us finally present applications of our main results and their relations to several known results in literature.

		\subsubsection{Arinkin--Fedorov's result: the elliptic locus} \label{sec:Arinkin-Fedorov}
			The following is slightly stronger than \cite[Theorem 2]{ari-fed16}.
			
			\begin{theorem} [Arinkin--Fedorov]
				If every fiber of $\pi$ is integral, then the inclusions in \eqref{eq:inclusions of group spaces} are equalities:
				\[ P^{\circ\circ} = P^a = P^{\tau} = P .\]
				As a result, $X' = X^n$ is a $P$-torsor.
			\end{theorem}
			\begin{proof}
				Say $f : X \dashrightarrow X$ is a birational translation automorphism over $B$ (after shrinking $B$ locally). By Kawamata, $f$ decomposes to finite sequence of flops as in \eqref{diag:flops}. For a single flop \eqref{diag:single flop}, every irreducible component of the flopping locus has to be an irreducible component of $\pi^{-1}(T)$ where $T \subset D_d$ is an irreducible component (\Cref{prop:exceptional locus of flop}). But if the fibers of $\pi$ are (geometrically) integral, then $\pi^{-1}(T)$ is integral and hence this cannot be contracted to $\bar X$ in \eqref{diag:single flop} (because the fiber dimension function of $\bar X \to B$ should be upper-semicontinuous). This means $f$ is an automorphism, showing $P^a = P$. Now $X'$ admits a free $P^a$-action but $X'$ has connected fibers by assumption, so $P^a$ also has to have connected fibers. This shows $P^{\circ\circ} = P^a$.
			\end{proof}
			
			An alternative interpretation of this result is $B_{\elliptic} \subset B_{\tor}$, where the former is the open locus over which the fiber of $\pi$ is geometrically integral (see \cite[Theorem 12.2.4]{EGAIV3}) and the latter is as in \Cref{def:torsion locus}.

		\subsubsection{Domain of definition of local rational sections} \label{sec:rational section}
			Let $U \to B$ be an \'etale or smooth morphism and $a : U \dashrightarrow X$ be an \emph{(\'etale or smooth) local rational section} of the Lagrangian fibration. The N\'eron mapping property of $X' \to B$ under $B_0 \subset B_{\tor}$ shows $a$ is necessarily defined on $U_{\tor} = U \times_B B_{\tor}$.
			
			\begin{proposition}
				Let $U \to B$ be a smooth morphism. Then every rational map $U \dashrightarrow X$ over $B$ is necessarily defined over $B_{\tor}$. \qed
			\end{proposition}
			
			In particular, since $B_{\tor}$ contains $B_1$ and $B_{\elliptic}$, every local rational section is necessarily defined over both $B_1$ and $B_{\elliptic}$. Here is an alternative version: for every local rational section $U \dashrightarrow X$, there exists a birational model $f : X \dashrightarrow Y$ such that the composition $f \circ a : U \dashrightarrow Y$ is regular.

		\subsubsection{Hwang--Oguiso's result: codimension $1$ fibers} \label{sec:Hwang-Oguiso}
			We have already discussed this in \Cref{prop:Hwang-Oguiso's result} and \ref{rmk:Hwang-Oguiso's result}, so let us be short. In \cite{hwang-ogu09, hwang-ogu11}, Hwang--Oguiso studied the structure of general singular fibers of $\pi : X \to B$. Let $b$ be a general closed point in the discriminant locus and $X_b = \pi^{-1}(b)$. In loc. cit., two sets of vector fields in $X_b$ are defined, a twisted vector field $\lambda$ and Hamiltonian vector fields $\xi_1, \cdots, \xi_{n-1}$. The integration of $\lambda$ and $\xi_i$'s yields the notion of a characteristic cycle and a $\CC^{n-1}$-action on $X_b$, which are crucial to understand the structure of $X_b$ in the paper. Let us now consider the Chevalley short exact sequence
			\[ 0 \to \GG \to P^{\circ\circ}_b \to A \to 0 ,\]
			where $\GG$ is a connected algebraic group of dimension $1$ and $A$ is an abelian variety of dimension $n-1$. The characteristic cycle is a connected sequence of closures of $\GG$-orbits, and the $\CC^{n-1}$-action is the $\Lie A \cong \CC^{n-1}$-action on $X_b$. This explains why their method is related to the group $P^{\circ\circ}$.

		\subsubsection{Abasheva--Rogov's result: the Tate--Shafarevich group} \label{sec:Abasheva-Rogov}
			In \cite{aba-rogov21}, Abasheva--Rogov studied the \emph{Tate--Shafarevich group} of a higher-dimensional Lagrangian fibration (based on work of \cite[\S 7]{mar14}) and its relations with the degenerate twistor deformation in \cite{ver15, bog-deev-ver22}. Abasheva--Rogov's definition of the Tate--Shafarevich group is (in our terminology)
			\[ \Sha^{\circ\circ}_{an} (B) = H^1_{an} (B, P^{\circ\circ}) .\]
			The definition of Tate--Shafarevich group is not uniform in literature (e.g., Markman's paper or \cite{huy-mat23}). We propose to study several groups together as follows.
			
			\begin{definition}
				The \emph{strict neutral, regular, numerically trivial, and birational Tate--Shafarevich groups} of $\pi$ are the first \'etale cohomology of the groups in \eqref{eq:inclusions of group spaces}:
				\[ \Sha^*(B) = \Sha^*(B, \pi) = H^1_{\acute et} (B, P^*) \qquad\mbox{for}\quad * = \circ\circ, \ a, \ \tau, \mbox{ or } \emptyset .\]
				Their analytic versions are similarly defined:
				\[ \Sha^*_{an}(B) = \Sha^*_{an}(B, \pi) = H^1_{an} (B, P^*) \qquad\mbox{for}\quad * = \circ\circ, \ a, \ \tau, \mbox{ or } \emptyset .\]
			\end{definition}
			
			Note that \Cref{thm:Neron model of fibration} gives a natural element
			\[ [X^n] \ \ \in \ \Sha(B) = H^1_{\acute et} (B, P) .\]
			A study of these groups and their relations should be done more comprehensively in separate work. Let us present some basic observations to illustrate the interactions of these groups.
			
			\begin{proposition} [Torsion]
				\begin{enumerate}
					\item $\Sha(B)$ is torsion.
					\item If $P^{\tau} = P$ then $\Sha^{\circ\circ}(B)$ and $\Sha^a(B)$ are both torsion.
				\end{enumerate}
			\end{proposition}
			\begin{proof}
				(1) We thank Olivier Benoist and James Hotchkiss for helping the following argument. Notice the inclusion
				\[ \Sha(B) = H^1_{\acute et} (B, P) \ \subset \ \Sha(B_0) = H^1_{\acute et} (B_0, P_0) \]
				from \Cref{prop:Neron torsor restriction}. Let us further claim $H^1_{\acute et} (B_0, P_0) \subset H^1 (K, P_K)$ where $K$ is the function field of $B$. Indeed, if a $P_0$-torsor $N_0 \to B_0$ admits a $K$-rational point then its Zariski closure $T \subset N_0$ defines a proper birational morphism $T \to B_0$. Its positive-dimensional fiber is rationally chain connected by \cite[Corollary 1.5]{hacon-mcker07}, violating the fact that fibers of $N_0 \to B_0$ are abelian varieties. Hence $T \to B_0$ is an isomorphism by Zariski's main theorem, proving the desired inclusion. Now $H^1 (K, P_K)$ is torsion since every $P_K$-torsor is projective by \cite[Theorem 6.4.1]{neron} and projective torsors correspond to torsion elements by \cite[Proposition XIII.2.3]{ray:group_schemes}.
				
				\bigskip
				
				(2) It is enough to show the following two separate statements:
				\begin{itemize}
					\item $\Sha^{\circ\circ}(B) \to \Sha^a(B)$ has a finite kernel and cokernel.
					\item $\Sha^a(B) \to \Sha^{\tau}(B)$ has a torsion kernel and cokernel.
				\end{itemize}
				For the first item, let $\epsilon = \pi_0(P^a)$ be the $\pi_0$-group (\S \ref{sec:etale group space}) and take the long exact sequence of $0 \to P^{\circ\circ} \to P^a \to \epsilon \to 0$. Since $P^a$ is quasi-projective, $\epsilon$ is an \'etale constructible sheaf (of finite abelian groups) and hence $H^i (B, \epsilon)$ is finite for all $i \ge 0$. For the second one, take the long exact sequence of $0 \to P^a \to P^{\tau} \to \epsilon \to 0$ and observe that $\epsilon$ is an \'etale sheaf of torsion abelian groups.
			\end{proof}
			
			\begin{proposition} [Tate--Shafarevich twist; Markman \cite{mar14}, Abasheva--Rogov \cite{aba-rogov21}] \label{prop:TS twist}
				For each $\alpha \in \Sha^a(B)$, there exists a proper Lagrangian fibration $\pi^{(\alpha)} : X^{(\alpha)} \to B$ from a smooth algebraic space $X^{(\alpha)}$, called a \emph{Tate--Shafarevich twist} of $\pi$.
			\end{proposition}
			\begin{proof}
				This is the construction in the paragraph after \cite[Observation 1.4]{mar14} or \cite[\S 2.2]{aba-rogov21}. We copy their construction for classes in $\Sha^a(B)$ (same works for $\Sha^a_{an}(B)$). A class $\alpha \in \Sha^a(B) = H^1_{\acute et}(B, P^a)$ is representable by a \v{C}ech cocycle $\{ f_{ij} \in P^a(U_{ij}) \}$ where $B^* = \bigsqcup_i U_i \to B$ is a covering in \'etale topology, $B^{**} = \bigsqcup_{ij} U_{ij} \to B$ for $U_{ij} = U_i \times_B U_j$, and $B^{***} = B^* \times_B B^* \times_B B^*$. We may consider $X \to B$ as a descent of $X^* \to B^*$ along $B^* \to B$ with a descent datum $\varphi : p_1^* X^* \to p_2^* X^*$ satisfying a cocycle condition.
				
				Since $P^a$ is a translation automorphism scheme, $f_{ij} : X_{ij} \to X_{ij}$ is a translation automorphism over $U_{ij}$ and hence induces an isomorphism $f : (X^{(\alpha)})^* \coloneq X^* \to X^*$ over $B^*$. Define a twisted isomorphism $\varphi^{(\alpha)}$ as follows:
				\[\begin{tikzcd}
					p_1^* (X^{(\alpha)})^* \arrow[r, "\varphi^{(\alpha)}"] \arrow[d, "p_1^* f"] & p_2^* (X^{(\alpha)})^* \arrow[d, "p_2^* f"] \\
					p_1^* X^* \arrow[r, "\varphi"] & p_2^* X^*
				\end{tikzcd}.\]
				The cocycle condition of $f_{ij}$ implies the cocyle condition of the descent datum $\varphi^{(\alpha)}$, proving the existence of a descent $\pi^{(\alpha)} : X^{(\alpha)} \to B$. We add three remarks:
				\begin{itemize}
					\item The Tate--Shafarevich twist corresponding to $0 \in \Sha^a(B)$ is $\pi^{(0)} = \pi$.
					\item If $\alpha, \beta \in \Sha^a(B)$ or $\Sha^a_{an}(B)$, then $\pi^{(\alpha)} \cong \pi^{(\beta)}$ if and only if $\alpha = \beta$. This is because a translation isomorphism $h : X^{(\alpha)} \to X^{(\beta)}$ corresponds to a coboundary condition of $\alpha = (f_{ij})$ and $\beta = (g_{ij})$.
					\item The homomorphism $\Sha^{\circ\circ}(B) \to \Sha^a(B)$ may not be injective, so it is better to think about the Tate--Shafarevich twist for elements in $\Sha^a(B)$ instead of $\Sha^{\circ\circ}(B)$.\qedhere
				\end{itemize}
			\end{proof}
			
			\begin{proposition} [Analytification] \label{prop:TS group analytification}
				Assume further that $B$ is proper. Then the analytification homomorphisms have the following properties.
				\begin{enumerate}
					\item $\Sha^a(B) \to \Sha^a_{an}(B)$ is injective.
					\item $\Sha^{\circ\circ}(B) \to \Sha^{\circ\circ}_{an}(B)$ is injective and induces $\Sha^{\circ\circ}(B)_{\tor} = \Sha^{\circ\circ}_{an}(B)_{\tor}$.
				\end{enumerate}
			\end{proposition}
			\begin{proof}
				(1) An element $\alpha \in \Sha^a(B)$ defines a Tate--Shafarevich twist $\pi^{(\alpha)} : X^{(\alpha)} \to B$ by \Cref{prop:TS twist}, which is a flat proper algebraic space. If $X$ and $X^{(\alpha)}$ are proper varieties that are analytically isomorphic, they are algebraically isomorphic by GAGA.
				
				\bigskip
				
				(2) Consider the $\pi_0$-group short exact sequence $0 \to P^{\circ\circ} \to P^a \to \epsilon \to 0$ and compare its \'etale and analytic long exact sequences
				\[\begin{tikzcd}
					H^0_{\acute et} (B, P^a) \arrow[r] \arrow[d, phantom, "\parallel"] & H^0_{\acute et} (B, \epsilon) \arrow[r] \arrow[d, phantom, "\parallel"] & \Sha^{\circ\circ}(B) \arrow[r] \arrow[d] & \Sha^a(B) \arrow[d, hook] \\
					H^0_{an} (B, P^a) \arrow[r] & H^0_{an} (B, \epsilon) \arrow[r] & \Sha^{\circ\circ}_{an}(B) \arrow[r] & \Sha^a_{an}(B)
				\end{tikzcd},\]
				The second vertical arrow is bijective because $\epsilon$ is an \'etale constructible sheaf. Since $P^a \to B$ is quasi-projective, we may consider an open immersion into a projective morphism $P^a \hookrightarrow \overline{P^a} \to B$. Now every analytic section $B \to P^a$ induces an analytic section $B \to \overline {P^a}$, which is algebraic by GAGA. This means the first vertical arrow is bijective. The injectivity of the third arrow now follows from the four lemma (diagram chasing).
				
				For the second statement, we prove the $m$-torsion parts of $\Sha^{\circ\circ}(B)$ and $\Sha^{\circ\circ}_{an}(B)$ agree. Consider the multiplication by $m$ endomorphism on $P^{\circ\circ}$ and its kernel $K$. Since $P^{\circ\circ}$ is quasi-projective, $K$ is an \'etale constructible sheaf of finite abelian groups. Compare the \'etale and analytic long exact sequences of $0 \to K \to P^{\circ\circ} \to P^{\circ\circ} \to 0$:
				\[\begin{tikzcd}
					H^1_{\acute et} (B, K) \arrow[r] \arrow[d, phantom, "\parallel"] & \Sha^{\circ\circ}(B) \arrow[r, "{[m]}"] \arrow[d, hook] & \Sha^{\circ\circ}(B) \arrow[d, hook] \\
					H^1_{an} (B, K) \arrow[r] & \Sha^{\circ\circ}_{an} (B) \arrow[r, "{[m]}"] & \Sha^{\circ\circ}_{an} (B)
				\end{tikzcd}.\]
				The left vertical arrow is an isomorphism by the Artin comparison theorem, showing the isomorphism of $m$-torsion parts $\Sha^{\circ\circ}(B) [m] = \Sha^{\circ\circ}_{an} (B) [m]$. Same argument yields $\Sha^a(B) [m] = \Sha^a_{an} (B) [m]$ for every $m$ coprime to the orders of $\pi_0(P^a)$.
			\end{proof}

		\subsubsection{de Cataldo--Rapagnetta--Sacc\`a's result} \label{sec:de Cataldo-Rapagnetta-Sacca}
			In \cite{decat-rap-sacca21}, de Cataldo--Rapagnetta--Sacc\`a analyzed the singular fibers of two different but related Lagrangian fibrations of projective hyper-K\"ahler manifolds. Let $S$ be a projective K3 surface with $\Pic(S) = \ZZ H$ for an ample line bundle $H$ and consider two Mukai vectors $v = (0, 2H, 2)$ and $v' = (0, 2H, 1)$ in $H^* (S, \ZZ)$. Consider the Beauville--Mukai systems over $B = |2H| = \PP^5$:
			\[ \pi : X \to M_v(B) \to B ,\qquad \tau : Y = M_{v'}(S) \to B ,\]
			where $X$ is a OG10-type and $Y$ is a $\text{K3}^{[5]}$-type hyper-K\"ahler manifold. The relative Jacobian of the universal family of curves $\mathcal C \to B = |2H|$ defines a smooth commutative group scheme $J = P^{\circ\circ} \to B$ with connected fibers acting on both $X$ and $Y$ over $B$.
			
			The paper shows that both $\pi : X \to B$ and $\tau : Y \to B$ are $\delta$-regular abelian fibrations with respect to the $P^{\circ\circ}$-action. Moreover, it calculates all non-integral singular fibers of $\pi$ and $\tau$, showing the structure of the $\delta$-loci
			\[ D_1 = |H|^{(2)} \cup D_{1, \elliptic} ,\quad D_2 = |H| \cup D_{2, \elliptic} ,\quad D_3 = |H| \cup D_{3, \elliptic} ,\quad D_4 = D_{4, \elliptic}, \quad \cdots \]
			where the closed immersions $|H|^{(2)} = |H|^{\times 2} / \mathfrak S_2 \subset |2H|$ and $|H| \subset |2H|$ are defined naturally and the subscript $-_{\elliptic}$ indicates the locus over which the fiber is geometrically integral.
			
			In their discussions, $X_0$ and $Y_0$ are both $P^{\circ\circ}_0$-torsors over $B_0$. Therefore, from our \Cref{thm:Neron model of abelian schemes}, the birational translation automorphism spaces of $X$ and $Y$ are isomorphic: $P = P_X = P_Y$. By \Cref{thm:Neron model of fibration codimension 1}, $X_1'$ and $Y_1'$ are both $P_1$-torsors and indeed this agrees with their calculations. The inclusion $P^{\circ\circ} \subset P$ is not an equality over $B_1$, because $P_b (= X'_b = Y'_b)$ has two connected components when $b \in |H|^{(2)} \cap B_1$. Hence $X'$ and $Y'$ cannot be $P^{\circ\circ}$-torsors over $B_1$. Further analyzing their results together with our previous discussions will yield the following:
			
			\begin{proposition}
				Notation as in the discussion above about \cite{decat-rap-sacca21}. Restricting everything over $B \setminus |H|$, we have the following.
				\begin{enumerate}
					\item $X'$ is a $P$-torsor and the group spaces in \eqref{eq:inclusions of group spaces} for $\pi$ are:
					\[ P^{\circ\circ} \subsetneq P^a = P^{\tau} = P .\]
					\item Same statements hold for $Y'$ and $\tau$.
				\end{enumerate}
				In particular, both $X'$ and $Y'$ are N\'eron under $B_0 \subset B \setminus |H|$, and $Y$ is a Tate--Shafarevich twist of $X$ by an element in $\Sha^a(B \setminus |H|)$ (see \Cref{prop:TS twist}). \qed
			\end{proposition}
			
			\begin{remark}
				\begin{enumerate}
					\item Over $|H|$, the fibers of $\tau$ are nowhere reduced by computations in \cite{decat-rap-sacca21} and \cite[Theorem 1.1]{hel21}, so our results do not apply.
					\item $X$ is of OG10-type and $Y$ is of $\text{K3}^{[5]}$-type. Hence they are not deformation equivalent but yet $X$ and $Y$ are Tate--Shafarevich twist of each other in a precise sense (over $B \setminus |H|$).
				\end{enumerate}
			\end{remark}
			
			Similar descriptions are given for a generalized Kummer $6$-fold and OG6-type hyper-K\"ahler manifold in \cite{wu24}. Same type of analyzation should work. Note also from \cite[Theorem 5.1]{kim25} that the N\'eron model $P \to B$ contains global sections $K = (\ZZ/2)^{\oplus 4}$ and $\check P = P/K$, explaining the coefficients $16$ in their result.

\appendix
\section{Backgrounds} \label{sec:backgrounds}
	The use of algebraic spaces in this article will be necessary due to the following reasons:
	\begin{itemize}
		\item The theory of N\'eron models is well-behaved (e.g., smooth descent) when they are assumed to be algebraic spaces.
		\item The torsor representing an element in $H^1$ may fail to be a scheme.
		\item The relative Picard space of a flat projective family may fail to be a scheme.
		\item A Mukai flop of a symplectic variety may fail to be a scheme.
		\item Some results in this article generalize to smooth, separated, and symplectic algebraic spaces $X$ of finite type.
	\end{itemize}
	It is often immediate to upgrade a statement about schemes to that about algebraic spaces (e.g., when the statement is \'etale local on the source). In such cases, we will use the statement for algebraic spaces without explaining it explicitly.
	
	In this article, every scheme is locally Noetherian and defined over a field of characteristic $0$. Every algebraic space is locally Noetherian and \emph{locally separated} over a field of characteristic $0$, so that it is a \emph{decent space} the sense of \cite[Tag~06NK, 088J]{stacks-project} as well as quasi-separated \cite[Tag~0BB6]{stacks-project}. Since \cite{stacks-project} will be cited quite often, we will specify only its Tag number when we cite their results, e.g., (Tag~ABCD) means \cite[Tag~ABCD]{stacks-project}.

	\subsection{Scheme theory facts}
		We collect various scheme theory facts in this section.

		\subsubsection{Sections}
		The following well-known results about local sections are extensively used throughout the article. Recall that the notion of a reside field is well-defined for points on a (decent) algebraic space (Tag~0EMV).
		
		\begin{proposition} [{\cite[Proposition 2.2.14]{neron}}] \label{prop:etale local section}
			Let $f : X \to S$ be a smooth morphism between algebraic spaces. Let $s \in S$ be a point and $x \in X_s$ a closed point on the fiber. Then there exists an \'etale morphism $\phi : U \to S$ and a point $s' \in U$ lying above $s$ with $k(s') = k(x)$, such that $f$ has a $U$-section $h : U \to X$ with $h(s') = x$:
			\[\begin{tikzcd}
				U \arrow[r, "h"] \arrow[rd, "\id"'] & X_U \arrow[r] \arrow[d] & X \arrow[d, "f"] \\
				& U \arrow[r, "\phi"] & S
			\end{tikzcd}.\]
		\end{proposition}
		
		\begin{proposition} [{\cite[{$\mathrm{VI_B}$ Lemma 5.6.1}]{SGA3I}}] \label{prop:quasi-finite local section}
			Let $f : X \to S$ be a locally finite type and universally open morphism between algebraic spaces. Let $s \in S$ be a point and $x \in X_s$ a closed point. Then there exists an \'etale morphism $U \to S$ and finite surjective morphism $S' \to U$, such that the composition
			\[ \phi : S' \to U \to S \]
			satisfies $\phi^{-1}(s) = \{ s' \}$, $k(s') = k(x)$, and that $f$ has an $S'$-section $h : S' \to X$ with $h(s') = x$:
			\[\begin{tikzcd}
				S' \arrow[r, "h"] \arrow[rd, "\id"'] & X' \arrow[r] \arrow[d] & X \arrow[d, "f"] \\
				& S' \arrow[r, "\phi"] & S
			\end{tikzcd}\]
		\end{proposition}
		
		\begin{remark} \label{rmk:local section}
			We may choose the \'etale morphism $\phi : U \to S$ in \Cref{prop:etale local section} to be quasi-finite. We may similarly do so in \Cref{prop:quasi-finite local section}, so that the morphism $\phi : S' \to S$ is quasi-finite. We may furthermore assume $S'$ is integral. In this case, $\phi : S' \to S$ is equidimensional in the sense of \S \ref{sec:equidimensional morphism}.
		\end{remark}
		
		\begin{proposition} [{\cite[Proposition 3.1.2]{neron}}] \label{prop:section passes through regular point}
			Let $f : X \to S$ be a locally finite type morphism between regular algebraic spaces. Then a section $S \to X$ of $f$ passes through a regular point of $X_s$ for all $s \in S$.
		\end{proposition}
		
		\begin{example}
			\Cref{prop:section passes through regular point} fails once $X$ is singular. Consider a minimal proper elliptic surface $\pi : S \to \Delta$ over a Dedekind scheme with central fiber $S_{t_0} = C_1 \cup C_2$ of Kodaira type $\text{I}_2$. Assume there exists a section $s : \Delta \to S$ passing through $C_1$. Contracting a $(-2)$-curve $C_1$ yields a proper flat morphism $\bar \pi : \bar S \to \Delta$ from a surface $\bar S$ with an ordinary double point singularity. Then $s$ induces a section of $\bar \pi$ passing through a singular point of $\bar S_{t_0}$.
		\end{example}

		\subsubsection{Fitting subschemes} \label{sec:Fitting subscheme}
			This subsection generally follows \cite[\S 20.2]{eis:commutative}.
			
			\begin{definition} [Fitting subscheme]
				Let $F$ be a coherent sheaf on a locally Noetherian scheme $X$. Consider its finite rank locally free presentation
				\[\begin{tikzcd}
					E_1 \arrow[r, "f"] & E_0 \arrow[r] & F \arrow[r] & 0
				\end{tikzcd} \qquad\mbox{with}\quad r = \rk E_0 .\]
				For each $i \ge 0$, the wedge power $\wedge^{r-i} f : \wedge^{r-i} E_1 \to \wedge^{r-i} E_0$ induces a homomorphism
				\[ f^{r-i} : \wedge^{r-i} E_1 \otimes \wedge^{r-i} E_0^{\vee} \to \mathcal O_X .\]
				The \emph{$i$-th Fitting subscheme of $F$} is a closed subscheme $\Fitt_i F \subset X$ defined by the ideal sheaf $\im f^{r-i} \subset \mathcal O_X$.
			\end{definition}
			
			We have a descending sequence of closed subschemes
			\[ X \supset \Fitt_0 F \supset \Fitt_1 F \supset \cdots ,\]
			which becomes an empty set in finite steps. Here are the properties of Fitting subschemes we use.
			
			\begin{proposition}
				Let $F$ be a coherent sheaf on $X$.
				\begin{enumerate}
					\item The scheme-theoretic support $\Supp F$ and the $0$-th Fitting subscheme $\Fitt_0 F$ are set-theoretically identical and scheme-theoretically related by
					\[ \Fitt_0 F \supset \Supp F .\]
					
					\item $F$ is locally free of rank $i$ on the locally closed subscheme $\Fitt_{i-1} F \setminus \Fitt_i F$.
					
					\item If $f : Y \to X$ is a morphism between locally Noetherian schemes, then there exists a canonical isomorphism $\Fitt_i (f^* F) = f^{-1} (\Fitt_i(F))$.
				\end{enumerate}
			\end{proposition}
			
			\begin{proposition} [{\cite[Lemma 1.4]{kal13}}] \label{prop:fitting2}
				Let $E_1 \xlongrightarrow{f} E_0 \to F \to 0$ be a locally free presentation of a coherent sheaf $F$ with $\rk E_i = r_i$. Consider the dual of $f$ and its cokernel $E_0^{\vee} \to E_1^{\vee} \to C \to 0$. Then we have an equality
				\[ \Fitt_i F = \Fitt_{i + r_1 - r_0} C .\]
			\end{proposition}

			\begin{lemma} [{\cite[Proposition III.10.6]{hartshorne}}] \label{lem:image of the fitting subscheme}
				Let $f : X \to S$ be a proper flat morphism of relative dimension $d$ between smooth schemes over an algebraically closed field of characteristic $0$. Then
				\[ \codim f(\Fitt_{d + i-1} \Omega_f) \ge i \qquad \mbox{for all} \quad i \ge 0 .\]
			\end{lemma}
			\begin{proof}
				To see our statement is a special case of Hartshorne's, consider a locally free presentation $f^* \Omega_S \to \Omega_X \to \Omega_f \to 0$. By definition, $\Fitt_{d+i-1} \Omega_f$ consists of points $x \in X$ over which the rank of $f^* \Omega_S \to \Omega_X$ is $\le \dim S - i$. This reduces our statement to Hartshorne's.
			\end{proof}

		\subsubsection{Scheme-theoretically dense open subspaces} \label{sec:open dense subspace}
			This subsection summarizes (Tag~0831). Because we are assuming algebraic spaces are locally Noetherian, we can adopt the simpler definition in (Tag~0835).
			
			\begin{definition}
				Let $S$ be a locally Noetherian algebraic space. Its open subspace $S_0 \subset S$ is called \emph{(scheme-theoretically) dense} if the scheme-theoretic image of the open immersion $S_0 \to S$ is $S$.
			\end{definition}
			
			If $S$ is $\mathrm{(S_1)}$, i.e., it has no embedded points, then an open subspace $S_0 \subset S$ is scheme-theoretically dense if and only if it is topologically dense in Zariski topology (Tag~083P).
			
			\begin{lemma} \label{lem:density respects flat base change}
				Let $S_0 \subset S$ be a dense open subspace. Then for a flat morphism $X \to S$ form a locally Noetherian algebraic space $X$, the open subspace $X_0 = X \times_S S_0 \subset X$ is dense.
			\end{lemma}
			\begin{proof}
				The scheme-theoretic image of a qcqs morphism is compatible with arbitrary flat base change \cite[Proposition 2.5.2]{neron}. The open immersion $S_0 \subset S$ is qcqs because $S$ is locally Noetherian.
			\end{proof}
			
			\begin{lemma} [Tag~084N] \label{lem:morphism to separated space is determined by dense subspace}
				Let $f, g : X \to Y$ be morphisms of algebraic spaces over $S$ and $U \subset X$ a dense open subspace. Assume $Y$ is separated over $S$. Then $f_{|U} = g_{|U}$ implies $f = g$.
			\end{lemma}

		\subsubsection{Traits} \label{sec:trait}
			This subsection follows \cite[\S 5.1]{hol19}. A \emph{Dedekind scheme} is a regular (connected) Noetherian scheme of dimension $1$. A local Dedekind scheme is the Spec of a discrete valuation ring (dvr).
			
			\begin{definition}
				Let $S$ be an algebraic space.
				\begin{enumerate}
					\item A \emph{trait} of $S$ is a nonconstant morphism $\Delta \to S$ from a local Dedekind scheme $\Delta$. Given a point $s \in S$, a \emph{trait passing through $s$} is a trait sending the closed point $0 \in \Delta$ to $s$.
					\item A trait is called \emph{flat} if it is a flat morphism. A trait is called \emph{unramified} (resp. \emph{quasi-finite}) if the induced morphism $\Delta \to \Spec \mathcal O^h_{S, s}$ is \emph{unramified} (resp. \emph{quasi-finite}).
					\item Say we have a preferred choice of a dense open subspace $S_0 \subset S$. A trait is called \emph{non-degenerate} if it sends the generic point $\xi \in \Delta$ to a point in $S_0$.
				\end{enumerate}
			\end{definition}
			
			A trait can be an arbitrary nonconstant morphism $\Delta \to S$; for example, it can be ramified or the field extension $k(0) / k(s)$ may even be transcendental. An unramified trait is quasi-finite (Tag~02V5).
			
			\begin{proposition} [{\cite[Proposition 2.4.1]{neron}}] \label{prop:valuative criterion for flatness}
				Let $f : X \to S$ be an algebraic space locally of finite type over a reduced scheme $S$. Let $x \in X$ and $s = f(x) \in S$ be points. Then $f$ is flat at $x$ if and only if for every trait $\Delta \to S$ passing through $s$, the base change $X_{\Delta} \to \Delta$ is flat at every point $x' \in X_{\Delta}$ lying over $x$.
			\end{proposition}
			
			\begin{proposition} [{\cite[Proposition III.9.7]{hartshorne}}] \label{prop:flat over Dedekind}
				Let $f : X \to \Delta$ be an algebraic space locally of finite type over a Dedekind scheme $\Delta$. Then $f$ is flat if and only if $f$ sends every associated point of $X$ to a generic point of $\Delta$.
			\end{proposition}

		\subsubsection{Universally equidimensional morphisms} \label{sec:equidimensional morphism}
			This section summarizes \cite[\S 2.1]{sus-voe00} and some parts in \cite[\S 13--15]{EGAIV3}.
			
			\begin{definition}
				Let $S$ be a locally Noetherian scheme. A locally finite type algebraic space $f : X \to S$ is called \emph{(locally) equidimensional} if:
				\begin{enumerate}
					\item The fiber dimension function
					\begin{equation} \label{eq:fiber dimension function}
						X \to \ZZ, \qquad x \mapsto \dim_x X_{f(x)}
					\end{equation}
					is (locally) constant.
					
					\item Every generic point of $X$ is sent to a generic point of $S$.
				\end{enumerate}
				The morphism $f$ is called \emph{universally (locally) equidimensional} if $X_T \to T$ is (locally) equidimensional for arbitrary base change $T \to S$ of schemes.
			\end{definition}
			
			\begin{remark}
				The original reference \cite[\S 2.1]{sus-voe00} assumes everything is separated. However, (local) equidimensionality is an \'etale local property on both the source and target, so we can reduce to the case $\Spec A \to \Spec B$.
			\end{remark}
			
			The fiber dimension function \eqref{eq:fiber dimension function} is upper-semicontinuous by Chevalley's theorem \cite[Theorem 13.1.3]{EGAIV3}, so it is locally constant if and only if it is locally constant on closed points of $X$. The following two results show the relation between universally open and universally equidimensional morphisms.
			
			\begin{proposition} [{\cite[Proposition 2.1.7]{sus-voe00}}] \label{prop:universally equidimensional morphism 1}
				Let $f : X \to S$ be a locally finite type algebraic space.
				\begin{enumerate}
					\item If \eqref{eq:fiber dimension function} is (locally) constant and $f$ is open, then $f$ is (locally) equidimensional.
					\item $f$ is universally (locally) equidimensional if and only if \eqref{eq:fiber dimension function} is (locally) constant and $f$ is universally open.
				\end{enumerate}
			\end{proposition}
			
			\begin{proposition} [{\cite[Proposition 2.1.7]{sus-voe00}}] \label{prop:universally equidimensional morphism 2}
				If $S$ is geometrically unibranch, then $f : X \to S$ is equidimensional if and only if it is universally equidimensional.
			\end{proposition}
			
			\begin{example}
				\begin{itemize}
					\item The morphism $\AA^1_k \sqcup \Spec k \to \AA^1_k$ is not equidimensional.
					\item Define $S = (xy = 0) \subset \AA^2_{(x,y)}$ and $X = \AA^1_x \sqcup (\AA^1_y \setminus \{ y = 0 \})$. Then $X \to S$ is equidimensional but not universally equidimensional, either because its restriction to $(x = 0) \subset S$ is not equidimensional or it is not open.
				\end{itemize}
			\end{example}
			
			The following relates universal equidimensionality to flatness.
			
			\begin{theorem} [{\cite[Corollary 15.2.3]{EGAIV3}}] \label{thm:universal open implies flat}
				Let $f : X \to S$ be a locally finite type algebraic space. Let $x \in X$ and $s = f(x) \in S$. Assume
				\begin{enumerate}
					\item $f$ is universally open at generic points of the irreducible components of $X_s$ containing $x$.
					\item $X_s$ is (geometrically) reduced at $x$.
					\item $S$ is reduced at $s$.
				\end{enumerate}
				Then $f$ is flat at $x$.
			\end{theorem}
			
			\begin{example} [Flatness and universal equidimensionality] \
				\begin{itemize}
					\item A flat morphism is universally locally equidimensional.
					\item If $X$ is non-reduced, then $X_{\red} \hookrightarrow X$ is equidimensional but not flat.
					\item Let $G$ be a finite group acting on $\AA^n_k$. The quotient $\AA^n \to \AA^n / G$ is universally equidimensional, but it is flat if and only if $\AA^n / G$ is regular.
					\item Let $f : X \to \Delta$ be an algebraic space locally of finite type over a Dedekind scheme. Then $f$ is flat if and only if every associated scheme-theoretic component of $X$ dominates $\Delta$. It is universally equidimensional if and only if every set-theoretic irreducible component of $X$ dominates $\Delta$.
				\end{itemize}
			\end{example}

		\subsubsection{Linear schemes} \label{sec:linear scheme}
			This section roughly follows \cite[\S 5.3.3]{nit:fga}.
			
			\begin{definition}
				Let $S$ be a (locally Noetherian) algebraic space. A \emph{linear space} is an affine morphism of finite type
				\[ \VV_F = \Spec_S (\Sym^* F) \to S \]
				for a coherent sheaf $F$ on $S$. When $S$ is a scheme, we call it a \emph{linear scheme}.
			\end{definition}
			
			An fppf sheaf on an algebraic space $S$ is a functor $F : (\Sp/S)^{\opp} \to \Set$ satisfying the sheaf axioms for the big fppf topology on $\Sp/S$. Let $\mathcal O_S^{\fppf}$ be the fppf structure sheaf on $S$ (represented by $\AA^1_S$). An \emph{fppf $\mathcal O_S$-module} is an fppf abelian sheaf $F$ on $S$ equipped with a bilinear map $\mathcal O_S^{\fppf} \times F \to F$ satisfying the axioms of modules.
			
			A coherent sheaf $F$ on an algebraic space can define an fppf $\mathcal O_S$-module in (at least) two different ways. The first way is to consider its \emph{associated fppf sheaf} defined by
			\begin{equation} \label{eq:associated fppf sheaf}
				F^{\fppf} : (\Sp/S)^{\opp} \to \Ab, \qquad \big[ f : T \to S \big] \mapsto H^0 (T, f^* F) .
			\end{equation}
			The second way is to consider its \emph{fppf dual} defined by
			\begin{equation} \label{eq:fppf dual}
				\Hom^{\fppf} (F, \mathcal O_S) : (\Sp/S)^{\opp} \to \Ab, \qquad \big[ f : T \to S \big] \mapsto \Hom_{\mathcal O_T} (f^* F, \mathcal O_T) .
			\end{equation}
			Both are fppf $\mathcal O_S$-modules.
			
			\begin{lemma}
				Let $F$ be a coherent sheaf on an algebraic space $S$. Then the linear space $\VV_F \to S$ represents the fppf dual \eqref{eq:fppf dual} of $F$.
			\end{lemma}
			\begin{proof}
				For every affine morphism $f : T = \Spec_S \mathcal A \to S$, we have $f_* \mathcal O_T = \mathcal A$ and thus
				\[ \Hom_S (T, \VV_F) = \Hom_{\mathcal O_S} (F, \mathcal A) = \Hom_{\mathcal O_T} (f^* F, \mathcal O_T) .\qedhere\]
			\end{proof}
			
			Therefore, a linear space $V \to S$ has a zero section $S \to V$, addition $V \times_S V \to V$, and scalar multiplication $\AA^1_S \times_S V \to V$ satisfying the usual axioms of modules. In particular, it is a separated and commutative group algebraic space. Given a linear space $f : V \to S$, its associated coherent sheaf $F$ can be obtained in two different ways:
			\begin{itemize}
				\item $F$ is the conormal sheaf of the zero section $S \to V$.
				\item Consider the $\ZZ_{\ge 0}$-grading on $f_* \mathcal O_V$ from the $\AA^1_S$-scalar multiplication structure. Then $F = (f_* \mathcal O_V)_1$ is its degree $1$ part.
			\end{itemize}
			
			\begin{proposition} \label{prop:linear scheme}
				Let $V$ and $W$ be fppf $\mathcal O_S$-modules and $f : V \to W$ a homomorphism of fppf $\mathcal O_S$-modules. Assume $V$ and $W$ are representable by linear spaces associated to coherent sheaves $F$ and $G$. Then $\ker f$ is representable by a linear space associated to the coherent sheaf $\coker (G \to F)$.
			\end{proposition}
			\begin{proof}
				By Yoneda, $f$ is a morphism of linear spaces $f : \VV_F \to \VV_G$ that induces a homomorphism of coherent sheaves $G \to F$. Take its cokernel $G \to F \to C \to 0$ and apply the fppf dual functor to it.
			\end{proof}
			
			An interesting example of a linear space is a \emph{tangent space}. A morphism between algebraic spaces $f : X \to S$ induces a cotangent sheaf $\Omega_f$ on $X$. Its linear space
			\begin{equation} \label{eq:tangent scheme}
				\TT_f = \VV_{\Omega_f} \to X
			\end{equation}
			is called a \emph{tangent space} of $f$; for every geometric point $\bar x$ of $X$ over $\bar s$, the fiber of the tangent space is $\TT_{f, \bar x} = (\mathfrak m_{\bar x}/\mathfrak m_{\bar x}^2)^{\vee}$, the Zariski tangent space of the geometric fiber $X_{\bar s}$ at $\bar x$ \cite[Proposition II.8.7]{hartshorne}. It represents the fppf $\mathcal O_X$-module $T_f^{\fppf} = \SheafHom^{\fppf} (\Omega_f, \mathcal O_X)$.

		\subsubsection{Others}
			The following is a nice application of Zariski's main theorem.
			
			\begin{proposition} \label{prop:Zariski main theorem}
				Let $f : X \to Y$ be a dominant morphism between integral and separated algebraic spaces of finite type over an algebraically closed field $k$ of characteristic $0$. Assume $Y$ is normal and $f : X(k) \to Y(k)$ is injective. Then $f$ is an open immersion.
			\end{proposition}
			\begin{proof}
				We follow Jason Starr's answer in \cite{starr:isomorphism}. The morphism $f$ is quasi-finite because $f : X(k) \to Y(k)$ is quasi-finite \cite[Remark 12.16]{gor-wed:ag}. By generic smoothness, $X$ contains an open dense subscheme $U$ where $f_{|U}$ is smooth, whence \'etale by its quasi-finiteness. Now using the injectivity of $f$ on closed points, $f_{|U}$ is an open immersion so $f$ is birational.
				
				Since $f$ is quasi-finite separated and $Y$ is qcqs, Zariski's main theorem (Tag~082K) says that there exists a factorization $X \xrightarrow{j} Z \xrightarrow{g} Y$ where $j$ is an open immersion and $g$ is finite. The morphism $g$ is birational as well. Now a finite extension $\mathcal O_Y \hookrightarrow g_* \mathcal O_Z$ has to be an equality, as $\mathcal O_Y$ is integrally closed and $g_*\mathcal O_Z$ is contained in the function field $K(Y)$. This proves $g$ is an isomorphism.
			\end{proof}

	\subsection{Group scheme actions and $\delta$-regularity}
		Let us recall the definition of $\delta$-regular group schemes and their actions on weak abelian fibrations.

		\subsubsection{Group scheme actions} \label{sec:group scheme action}
			We quickly clarify our terminology for a group scheme action on an algebraic space. Our reference includes \cite[\S 0]{mum:git} and \cite[\S VI]{ray:group_schemes}. Throughout, we assume $G \to S$ is a group scheme (or a group algebraic space as in \S \ref{sec:group spaces}) and $f : X \to S$ is an algebraic space. A \emph{$G$-action} on $X$ is a morphism of algebraic spaces $G \times_S X \to X$ satisfying the axiom of group actions. Given a $G$-action on $X$, its \emph{universal translation} morphism is
			\begin{equation} \label{eq:universal translation}
				\phi : G \times_S X \to X \times_S X, \qquad (g,x) \mapsto (g.x, x) .
			\end{equation}
			
			\begin{definition}
				Consider a $G$-action on an algebraic space $f : X \to S$.
				\begin{enumerate}
					\item (\cite[Definition 0.8]{mum:git}) The action is \emph{free} if the universal translation morphism $\phi$ in \eqref{eq:universal translation} is an immersion.
					
					\item (\cite[Definition VI.1.1]{ray:group_schemes}) The action is \emph{transitive} (or $X$ is a \emph{$G$-homogeneous space}) if $\phi$ is faithfully flat.
					
					\item $X$ is a \emph{$G$-torsor} if the action is free and transitive, i.e., $\phi$ is an isomorphism.
					
					\item The action is \emph{faithful} if for every point $s \in S$, the $G_s$-action on $X_s$ is faithful.
				\end{enumerate}
			\end{definition}
			
			\begin{definition}
				Let $T \to S$ be a morphism. Given a $T$-section $x : T \to X$, we define its \emph{orbit map} $\phi_x : G_T \to X_T$ by the pullback
				\[\begin{tikzcd}
					G \times_S T \arrow[r, "\phi_x"] \arrow[d] & X \times_S T \arrow[d, "\id \times x"] \\
					G \times_S X \arrow[r, "\phi"] & X \times_S X
				\end{tikzcd}.\]
				Equivalently, it is simply defined by $\phi_x (g) = g.x$.
			\end{definition}
			
			\begin{definition}
				The \emph{stabilizer (group) scheme} $\St \to X$ is the pullback
				\[\begin{tikzcd}
					\St \arrow[r] \arrow[d] & X \arrow[d, "\Delta"] \\
					G \times_S X \arrow[r, "\phi"] & X \times_S X
				\end{tikzcd}.\]
				It is a locally closed subgroup scheme of $G \times_S X = f^*G \to X$.
			\end{definition}
			
			\begin{proposition} \label{prop:infinitesimal action}
				Let $G \to S$ be a group space acting on an algebraic space $f : X \to S$. Then there exists a left exact sequence of Lie algebra spaces on $X$:
				\[\begin{tikzcd}
					0 \arrow[r] & \Lie \St \arrow[r] & f^* (\Lie G) \arrow[r, "d \phi"] & \TT_f
				\end{tikzcd}.\]
				The map $d\phi$ is described as follows: take the fppf sheaf homomorphism $G \to \SheafAut_f$, its Lie algebra $\Lie G \to \Lie \SheafAut_f = f_* \TT_f$, and its adjoint $f^* (\Lie G) \to \TT_f$.
			\end{proposition}
			\begin{proof}
				Let $C_e$ be the conormal sheaf of the identity section of $G$. Denote by $j : X \to \St$ the identity section, and $p : G \times_S X \to X$ and $q : X \times_S X \to X$ the second projections. Consider the commutative diagram
				\[\begin{tikzcd}[column sep=tiny]
					X \arrow[rd, hook, "i"'] \arrow[r, hook, "j"] & \St \arrow[rr, "\phi'"] \arrow[d] & & X \arrow[d, "\Delta"] \\
					& G \times_S X \arrow[rr, "\phi"] \arrow[rd, "p"'] & & X \times_S X \arrow[ld, "q"] \\
					& & X
				\end{tikzcd}.\]
				Take the cotangent sequence $\phi^* \Omega_q \to \Omega_p \to \Omega_{\phi} \to 0$ and pull it back by $i^*$ to yield:
				\[ \Omega_f \longrightarrow f^* C_e \longrightarrow C_j \longrightarrow 0 .\]
				Here the last term is from isomorphisms $i^* \Omega_{\phi} = j^* \Omega_{\phi'} = C_j$. Its fppf dual is the desired left exact sequence.
				
				Let us now check the map defined in the first paragraph and the map $d\phi$ in the statement coincide. Set $\Spec k[\epsilon] = \Spec k[t]/t^2$ and realize the tangent schemes as Hom-schemes
				\[ \TT_p = \Hom_X (X \times_k k[\epsilon], G \times_S X) ,\qquad \TT_q = \Hom_X (X \times_k k[\epsilon], X \times_S X) .\]
				For each test scheme $x : T \to X$, the map $i^* \TT_p \to i^* \phi^* \TT_q$ in the first paragraph over $T$ is
				\[ \Hom_T (T \times_k k[\epsilon], G_T) \to \Hom_T (T \times_k k[\epsilon], X_T) \]
				given by the composition with the orbit map $\phi_x : G_T \to X_T$. Same description holds for the map $d\phi$ given in the statement.
			\end{proof}
			
			\begin{proposition} \label{prop:trivial stabilizer is free}
				Let $G \to S$ (resp. $f : X \to S$) be a smooth, separated, and finite type group space (resp. algebraic space with a $G$-action) over a normal scheme $S$. Assume that $X$ is a $G$-torsor over a dense open subscheme $S_0 \subset S$ and $\Lie \St \to X$ is trivial. Then the $G$-action on $X$ is free over $S$.
			\end{proposition}
			\begin{proof}
				Let us first show that the universal translation morphism $\phi : G \times_S X \to X \times_S X$ in \eqref{eq:universal translation} is quasi-finite. The morphism is defined over $X$ by second projections, so we may show the orbit map $\phi_x : G_x \to X_x$ is quasi-finite for every $x \in X$. Note that the stabilizer scheme $\St \to X$ is quasi-finite because its relative dimension is $0$ ($\Lie \St$ is trivial) and it is of finite type. Hence for every $x \in X$, the stabilizer $\St_x$ at $x$ is a finite algebraic group. For each point $y \in X_x$, its preimage under $\phi_x : G_x \to X_x$ is either empty or isomorphic to $\St_x \times_{k(x)} k(y)$. Hence it is a finite set so $\phi_x$ is quasi-finite.
				
				We now claim $\phi$ is an open immersion. Note that $\phi$ is quasi-finite, separated, and birational (isomorphic over $S_0$), and that $X \times_S X$ is normal. By Zariski's main theorem as in \Cref{prop:Zariski main theorem}, $\phi$ is an open immersion.
			\end{proof}

		\subsubsection{Ng\^o's $\delta$-regularity and weak abelian fibration} \label{sec:delta-regularity}
			Our reference for this section is Ng\^o's original article \cite[\S 7.1]{ngo10} \cite[\S 2]{ngo11} together with \cite[\S 5.4]{ari-fed16}, \cite[\S 2.2]{decat-rap-sacca21}, and \cite[\S 6]{ana-cav-lat-sac23}. In this subsection, we assume $S$ is a Noetherian scheme of characteristic $0$.
			
			Given a connected algebraic group $G$ over a field $k$ of characteristic $0$, Chevalley's structure theorem (e.g., \cite[Theorem 9.2.1]{neron}) claims there is a canonical short exact sequence
			\[ 1 \to G_{\operatorname{aff}} \to G \to G_{\operatorname{ab}} \to 1 ,\]
			where $G_{\operatorname{aff}}$ is a connected linear algebraic group and $G_{\operatorname{ab}}$ is an abelian variety, both defined over $k$. Now let $P \to S$ be a smooth commutative group scheme. For each point $s \in S$, the fiber $P_s$ is a locally algebraic group over $k(s)$, so its (strict) neutral component $P^{\circ\circ}_s$ admits a Chevalley short exact sequence as above. We call
			\begin{equation}\label{eq:delta-function}
				\delta : S \to \ZZ, \qquad s \mapsto \dim (P^{\circ\circ}_s)_{\operatorname{aff}}
			\end{equation}
			the \emph{$\delta$-function}. Notice that the $\delta$-function of $P$ is identical to that of its strict neutral component $P^{\circ\circ}$ (\S \ref{sec:strict neutral component}). Since $P^{\circ\circ} \to S$ is of finite type (\Cref{thm:smooth group space with connected fibers}), the $\delta$-function is an upper-semicontinuous function by \cite[Lemma 5.6.3]{ngo10}.
			
			\begin{definition} [$\delta$-regular group scheme]
				Let $S$ be a Noetherian scheme. A smooth commutative group scheme $P \to S$ is \emph{$\delta$-regular} if for each $i \ge 0$, the closed subset
				\begin{equation} \label{eq:appendix:delta-loci}
					D_i = \{ s \in S : \delta(s) \ge i \} \ \subset \ S
				\end{equation}
				has codimension $\ge i$ (i.e., every irreducible component of it has codimension $\ge i$).
			\end{definition}
			
			Note that if $P \to S$ is $\delta$-regular then $D_1 \subset S$ is a strict inclusion, so its neutral component $P^{\circ} \to S$ (\S \ref{sec:neutral component}) is an abelian scheme over $S_0 = S \setminus D_1$.
			
			\begin{definition} [Weak abelian fibration] \label{def:weak abelian fibration}
				Let $S$ be a Noetherian scheme. A \emph{weak abelian fibration} over $S$ is a pair $(X, P)$ such that
				\begin{enumerate}[label=\textnormal{(\roman*)}]
					\item $f : X \to S$ is a proper equidimensional morphism of relative dimension $d$.
					\item $P \to S$ is a smooth, commutative, and quasi-projective group scheme acting on $X$ over $S$. It has the same relative dimension $d$.
					\item For every point $s \in S$, the kernel of the $P_s$-action on $X_s$ is affine.
				\end{enumerate}
				We simply call $X \to S$ a weak abelian fibration if the $P$-action on $X$ is clear in context.
			\end{definition}
			
			\begin{example} \label{ex:weak abelian fibration}
				\Cref{def:weak abelian fibration} is more restrictive than Ng\^o's original definition \cite[\S 7.1.1--4]{ngo10} in two respects. (i) We required $f$ to be equidimensional. (ii) We assumed $P \to S$ is quasi-projective instead of assuming its Tate module is \emph{polarizable} as in \cite[\S 7.1.4]{ngo10}.
			\end{example}
			
			\begin{lemma} \label{lem:delta-regularity}
				Let $(X, P)$ be a weak abelian fibration over a Noetherian scheme $S$ as in \Cref{def:weak abelian fibration}. Then
				\begin{enumerate}
					\item The Tate module of $P$ is polarizable.
					\item Every point $x \in X$ has an affine stabilizer $\St_x$.
				\end{enumerate}
				In particular, $(X, P)$ is a weak abelian fibration in the sense of \cite[\S 7.1.1]{ngo10}.
			\end{lemma}
			\begin{proof}
				(1) This is Theorem~1.2 and 6.2 in \cite{anc-fra24}: every relative ample line bundle on $\nu : P^{\circ\circ} \to S$ induces a polarization of the Tate module $R^{2d-1} \nu_! \QQ_l$. (2) This is \cite[Lemma 5.16]{ari-fed16}: if $s \in S$ is a point, then a single point $x \in X_s$ has an affine stabilizer $\Longleftrightarrow$ every point $x \in X_s$ has an affine stabilizer $\Longleftrightarrow$ $\ker (P_s \to \Aut_{X_s/k(s)})$ is affine.
			\end{proof}
			
			The following is the original idea of Ng\^o that the $\delta$-regularity of a group scheme can be checked from its action on a weak abelian fibration $X \to S$.
			
			\begin{proposition} \label{prop:delta-regularity}
				Let $(X, P)$ be a weak abelian fibration over a Noetherian scheme $S$ and $\St \to X$ the stabilizer scheme of the action. Then
				\begin{enumerate}
					\item The $\delta$-function in \eqref{eq:delta-function} can be computed by
					\[ \delta(s) = \max \{ \dim \St_x : x \in X_s \} \qquad\mbox{for all}\quad s \in S .\]
					
					\item The $\delta$-loci in \eqref{eq:appendix:delta-loci} can be computed by
					\[ D_i = \pi \big( \Fitt_{i-1} C_e \big) ,\]
					where $C_e$ is the conormal sheaf of the identity section $e : X \to \St$.
				\end{enumerate}
			\end{proposition}
			\begin{proof}
				(1) This is tacitly mentioned in the discussion of algebraic integrable systems in \cite[\S 2]{ngo11} and proved in \cite[Lemma 5.17]{ari-fed16}: it is a consequence of Borel's fixed point theorem. (2) From the previous item, $D_i$ is the image of the closed set
				\[ \{ x \in X : \dim \St_x \ge i \} \ \subset \ X .\]
				The fiber dimension of $\St$ is the fiber dimension of the coherent sheaf $C_e$. Hence this closed subset is the $(i-1)$-th Fitting subscheme of $C_e$ (see \S \ref{sec:Fitting subscheme}).
			\end{proof}
			
			\begin{definition} [$\delta$-regular abelian fibration]
				A \emph{$\delta$-regular abelian fibration} is a weak abelian fibration $(X, P)$ such that $P$ is $\delta$-regular.
			\end{definition}

\printbibliography
\end{document}